\documentclass[a4paper]{article}

\usepackage[english]{babel}
\usepackage[utf8]{inputenc}
\usepackage{paralist,url,verbatim,anysize}
\usepackage{amscd}
\usepackage[usenames,dvipsnames,svgnames,table]{xcolor}
\definecolor{purple}{HTML}{961C8C}
\usepackage{manfnt}
\usepackage{nicefrac}
\usepackage{microtype}
\usepackage[e]{esvect}
\usepackage{amsmath,amsthm,amssymb,amsfonts}
\usepackage{graphicx, graphics}
\usepackage[bookmarksopen=true, backref=page, linktoc=page]{hyperref}
\hypersetup{colorlinks=false, citecolor=blue, urlcolor=TealBlue}
\usepackage{epstopdf}
\usepackage{bbm}
\usepackage[font=small,labelfont=bf]{caption}
\theoremstyle{plain}
\newtheorem*{theorem*}{\bf Theorem}
\newtheorem{theorem}{\bf Theorem}[subsection]
\newtheorem{conjecture}{Conjecture}
\renewcommand{\theconjecture}{\Roman{conjecture}}
\newtheorem{cor}[theorem]{Corollary}
\newtheorem{lemma}[theorem]{Lemma}
\newtheorem{prp}[theorem]{Proposition}

\newtheorem{problem}{\bf Problem}
\renewcommand{\theproblem}{\Roman{problem}}
\newtheorem{problemm}[theorem]{Problem}

\newtheorem{thmmain}{\bf Theorem}

\theoremstyle{definition}
\newtheorem{rem}[theorem]{Remark}
\newtheorem*{rem*}{Remark}
\newtheorem{definition}[theorem]{Definition}
\newtheorem{obs}[theorem]{Observation}

\newtheorem{example}[theorem]{Example}
\newtheorem{fml}[theorem]{Formula}
\newtheorem{examples}[theorem]{Examples}

\let\oldbibliography\thebibliography
\renewcommand{\thebibliography}[1]{
  \oldbibliography{#1}
  \setlength{\itemsep}{-1pt}
}

\DeclareMathAlphabet{\mathpzc}{OT1}{pzc}{m}{it}
\newenvironment{frcseries}{\fontfamily{frc}\selectfont}{}

\newcommand{\size}{\mathrm{size} }  
\newcommand{\sm}{\hspace{.08em}}
\newcommand{\Lk}{\mathrm{Lk}\sm}
\newcommand{\St}{\mathrm{St}\sm}
\newcommand{\intx}{\mathrm{int}\sm}
\newcommand{\rint}{\mathrm{relint}\sm} 
\newcommand{\parti}{\partial \sm}
\newcommand{\TT}{{T}}
\newcommand{\NO}{{N}}
\newcommand{\cl}{\mathrm{cl}\sm}
\newcommand{\conv}{\mathrm{conv}\sm }
\newcommand{\Sp}{\mathrm{sp} \sm }
\newcommand{\rot}{{\mathrm{r}} \sm}
\newcommand{\spi}{\mathrm{s} \sm}
\newcommand{\SSp}{\mathrm{sp} \sm}
\newcommand\Defn[1]{\emph{{#1}}}
\newcommand{\RR}{\mathrm{R} }
\newcommand{\F}{\mathrm{F} }

\newcommand{\fat}{\mathrm{fat} }
\newcommand{\R}{\mathbb{R}}

\newcommand{\Z}{\mathbb{Z}}
\newcommand{\pp}{\mathrm{p} }
\newcommand{\n}{\vv{n}}
\newcommand{\CT}{\mathrm{T}}
\newcommand{\PS}{\CT^{s}}
\newcommand{\eq}{{\mathrm{eq}}}
\newcommand{\e}{\varepsilon}
\newcommand{\T}{\mathfrak{T}}
\newcommand{\RC}{\mathfrak{C}}

\newcommand{\cR}{\mathcal{RS}}
\newcommand\DL{\delta}
\newcommand{\cm}[1]{}
\newcommand{\bigslant}[2]{{\raisebox{.3em}{$#1$} \Big/ \raisebox{-.3em}{$#2$}}}

\renewcommand{\thefigure}{\arabic{section}.\arabic{figure}}
\renewcommand{\thetable}{\arabic{section}.\arabic{table}}

\begin{document}

\author{
Karim A. Adiprasito
\thanks{This work was supported by the DFG within the research training group ``Methods for Discrete Structures'' (GRK1408) and by the Romanian NASR, CNCS – UEFISCDI, project PN-II-ID-PCE-2011-3-0533.}\\ 
\small Institut des Hautes \'Etudes Scientifiques \\
\small Le Bois-Marie 35, route de Chartres \\
\small  91440 Bures-sur-Yvette, France \\
\small \url{adiprasito@math.fu-berlin.de}
\and
\setcounter{footnote}{6}
G\"unter M. Ziegler\thanks{The research leading to these results has received funding from the European Research Council under the European Union's Seventh Framework Programme (FP7/2007-2013) / ERC Grant agreement no.~247029-SDModels and the DFG Collaborative Research Center
SFB/TR 109 ``Discretization in Geometry and Dynamics''} \\
\small Institut f\" ur Mathematik, FU Berlin\\
\small Arnimallee 2\\
\small 14195 Berlin, Germany\\
\small \url{ziegler@math.fu-berlin.de}
}

\date{\footnotesize submitted July 22, 2013\\ revised March 13, 2014}
\title{Many projectively unique polytopes}
\maketitle

\begin{abstract}
We construct an infinite family of $4$-polytopes whose realization spaces have dimension smaller or equal to $96$. This in particular settles a problem going back to Legendre and Steinitz: whether and how
the dimension of the realization space of a polytope is determined/bounded by its $f$-vector.
 
From this, we derive an infinite family of combinatorially distinct $69$-dimensional polytopes whose realization is unique up to projective transformation. This answers a problem posed by Perles and Shephard in the sixties. Moreover, our methods naturally lead to several interesting classes of projectively unique polytopes, among them projectively unique polytopes inscribed to the sphere.

The proofs rely on a novel construction technique for polytopes based on solving Cauchy problems for discrete conjugate nets in $S^d$, a new Alexandrov--van Heijenoort Theorem for manifolds with boundary and a generalization of Lawrence's extension technique for point configurations. 
\end{abstract}

\enlargethispage{5mm}

\section{Introduction}
Legendre initiated the study of the spaces of geometric realizations of polytopes, motivated by problems in mechanics. One of the questions studied in his 1794 monograph on geometry~\cite{Legendre} is:
\begin{quote}
\emph{How many variables are needed to determine a geometric realization of a given (combinatorial type~of) polytope?}
\end{quote}

In other words, Legendre asks for the dimension of the \emph{realization space} $\cR (P)$ of a given polytope~$P$, that is, the space of all coordinatizations for the particular combinatorial type of polytope (cf.\ Definition~\ref{def:realization_space}). For $2$-polytopes (i.e., $2$-dimensional polytopes), it is not hard to see that the number of variables needed is given by $f_0(P)+f_1(P)$. Using Euler's formula, Legendre concluded that for $3$-polytopes the number of variables needed to determine the polytope up to congruence equals the number of its edges 
\cite[Note VIII]{Legendre}. While his argument made use of some tacit assumptions, Legendre's reasoning was later confirmed by Steinitz who supplied the first full proof 
(cf.~\cite[Sec.\ 34]{steinenc}~\cite[Sec.\ 69]{stein}) of what we here call the Legendre--Steinitz formula:
    
\begin{theorem*}[Legendre--Steinitz Formula 
	{\cite[Note VIII]{Legendre}} {\cite[Sec.\ 34]{steinenc}}]\label{thm:legst} For any $3$-polytope $P$, the realization space has dimension $f_0(P)+f_2(P)+4=f_1(P)+6$.
\end{theorem*}

In this paper we treat two questions concerning the spaces of geometric realizations of 
higher-dimensional polytopes. The first problem originates with Perles and Shephard~\cite{PerlesShephard}: 

\renewcommand{\theproblem}{P-S}
\begin{problem}[Perles \& Shephard~\cite{PerlesShephard} Kalai~\cite{Kalai}]\label{prb:projun} 
Is it true that, for each fixed $d\ge  2$, the number of distinct combinatorial types of projectively unique $d$-polytopes is finite?
\end{problem}

For the case of $d=4$, McMullen and Shephard made a bolder conjecture, referring to a list of $11$ projectively unique $4$-polytopes constructed by Shephard in the sixties (see Appendix~\ref{ssc:Shphrdlist}).

\renewcommand{\theconjecture}{M-S}
\begin{conjecture}[McMullen \& Shephard~\cite{McMullen}]\label{conj:mcmsh}
Every projectively unique $4$-polytope is accounted for in Shephard's list of $11$ combinatorial types of
projectively unique $4$-polytopes.
\end{conjecture}
\enlargethispage{3mm}
The connection of Problem~\ref{prb:projun} to the study of the dimension of $\cR(P)$ is this: A $d$-polytope $P$ can be projectively unique only if $\dim \cR (P)$ is smaller or equal to $d(d+2)$, the dimension of the projective linear group $\mathrm{PGL}(\R^{d+1})$ of projective transformations on $\R^d$. In particular, we obtain:

\begin{compactitem}[$\circ$]
\item A $2$-polytope can be projectively unique only if $f_0(P)+f_1(P)\le  8= \dim \mathrm{PGL}(\R^3)$. 
\item A $3$-polytope can be projectively unique only if $f_1(P)+6 
    \le  15= \dim \mathrm{PGL}(\R^4)$. 
\end{compactitem}

\noindent A more careful analysis reveals that this is a complete characterization of projectively unique polytopes in dimensions up to $3$: A $2$-polytope is projectively unique if and only it has $3$ or $4$ vertices;
a $3$-polytope is projectively unique if and only if it has at most $9$ edges~\cite[Sec.~4.8, Pr.~30]{Grunbaum}.

Projectively unique polytopes in dimensions higher than $3$ are far from understood. There has been substantial progress in the understanding of realization spaces of polytopes up to “stable equivalence” (thus, in particular, up to homotopy equivalence), due to the work of Mn\"ev~\cite{Mnev} and Richter-Gebert~\cite{RG}.
Nevertheless, no substantial progress was made on the problem of Perles and Shephard since it was asked in the sixties (see~\cite{PerlesShephard}). Related results on projectively unique polytopes include:

\begin{compactitem}[$\circ$]
\item Any $d$-polytope with at most $d+2$ vertices is projectively unique.
\item There are projectively unique $d$-polytopes with exponentially many vertices~\cite{McMullen} \cite{PerlesShephard}.
\item There are projectively unique polytopes that, while realizable in $\R^8$, are not realizable such that all vertices have rational coordinates~(Perles, see~\cite[Sec.\ 5.5, Thm.\ 4]{Grunbaum}). 
\end{compactitem} 
The above discussion for the low-dimensional cases of Problem~\ref{prb:projun} motivates one to look for bounds on the parameter $\dim \cR (P)$ for $d$-dimensional polytopes $P$ as a step towards the problem of projectively unique polytopes. 
Does the Legendre--Steinitz formula have a 
high-dimensional analogue? If the \Defn{size} of a polytope is defined as the
dimension times the total number of its vertices and facets, 
\[
\size(P)\ := \ d\big(f_0(P)+f_{d-1}(P)\big),
\]
this problem can be made more concrete as follows.

\renewcommand{\theproblem}{L-S}
\begin{problem}[Legendre--Steinitz in general dimensions~\cite{ZA}]\label{prb:steinitz}
How does, for $d$-dimensional polytopes, the dimension of the realization space grow with the size of the polytope?
\end{problem}

We have $\dim\cR(P)=\frac12\size(P)$ for $d=2$ and $\dim\cR(P)=\frac13\size(P)+4$ for $d=3$, so in both cases
the dimension of the realization space grows linearly with the size of the polytope.
In contrast, it is known that for $d\ge 4$ the $f$-vector of a $d$-polytope $P$ does not suffice to 
determine $\dim \cR(P)$. One cannot even determine from the $f$-vector whether a polytope is projectively unique 
(cf.\ Gr\"unbaum~\cite[Sec.\ 4.8, Pr.\ 30]{Grunbaum}). Thus Problem~\ref{prb:steinitz} does not ask for a formula 
for $\dim \cR(\cdot)$ in terms of the size 
of~$P$, but rather for good upper and lower bounds. 
Upper bounds are given by
\[
\dim \cR(P)\le  d\cdot f_0(P)\qquad \text{and}\qquad \dim \cR(P)\le  d\cdot f_{d-1}(P);
\]
these bounds are sharp for simplicial resp.\ simple polytopes. In particular,
we always have
\[
\dim \cR(P) < \size(P).
\]
The quest for lower bounds, however, relates Problem~\ref{prb:steinitz} to the McMullen--Shephard Conjecture and Problem~\ref{prb:projun}, and is apparently a much harder problem about which little is known. Still,   
$4$-polytopes for which the dimension of the realization space grows sublinearly with the size are known:  
In~\cite{ZA}, we argued that for the “{neighborly cubical polytopes}” $\mathrm{NCP}_4[n]$ 
constructed by Joswig \& Ziegler~\cite{JZ} the dimensions of the realization spaces are low relative to their size: 
\[\dim \cR (\mathrm{NCP}_4[n]) \sim 
	(\log\size\, \mathrm{NCP}_4[n])^2.\]

\noindent\textbf{Main Results.} Trivially, the dimension of $\cR(P)$ for a $d$-polytope $P$ is always at least $d(d+1)$. Moreover, unless $P$ is the join of two of its faces, i.e.\ the convex hull of two faces positioned in skew affine subspaces, $\mathrm{PGL}(\R^{d+1})$ acts freely on the realization space $\cR(P)$ and therefore $\dim \cR(P)\ge d(d+2)$.

We answer Problem~\ref{prb:steinitz} by showing
that, in all dimensions $d\ge 4$, this trivial lower bound is asymptotically optimal: There exists, in every dimension $d\ge  4$, an infinite family of 
combinatorially distinct $d$-polytopes, which are not joins, and for which the dimension of the realization space is a 
constant that depends on $d$ --- indeed, we can bound this by $d(d+2)$ plus an absolute constant.

\enlargethispage{1mm}

\begin{thmmain}[Cross-bedding cubical torus polytopes]\label{mthm:Lowdim}
For each $d\ge  4$, there exists an infinite family  of distinct $d$-dimensional polytopes $\mathrm{CCTP}_d[n]$ with $12(n+1)+d-4$ vertices such that $\dim \cR (\mathrm{CCTP}_d[n])\le 76+d(d+1)$ for all~$n\ge1$.
\end{thmmain}

Not only does this settle Problem~\ref{prb:steinitz}, it also provides strong evidence that Conjecture~\ref{conj:mcmsh} is wrong. Building on the proof of Theorem~\ref{mthm:Lowdim} we then solve Problem~\ref{prb:projun} for high-dimensional polytopes:

\begin{thmmain}[Projectively unique cross-bedding cubical torus polytopes]\label{mthm:projun}
For each $d\ge  69$, there exists an infinite family of distinct $d$-dimensional polytopes $\mathrm{PCCTP}_{d} [n],\, n\ge 1,$ with $12(n+1)+60+d$ vertices, all of which are projectively unique.
\end{thmmain}

This provides further evidence against Conjecture~\ref{conj:mcmsh}. By exploiting algebrao-geometric properties of our construction, we are able to give constructions for even more projectively unique polytopes with additional properties (Section~\ref{sec:varieties}), among them many inscribed projectively unique polytopes.

\begin{thmmain}[cf.\ Theorem \ref{thm:anyfield}]
For any finite field extension $F\subset \R$ over $\mathbb{Q}$ and for all $d\ge 0$ large enough, there is a family of projectively unique $d$-polytopes $\mathrm{PCCTP}^F_d[n],\ n\ge 1,$ on $12(n+1)+d+\mathrm{const.}$ vertices with coordinates in $F$, that are not realizable with coordinates in any strict subfield of $F$.
\end{thmmain}

\begin{thmmain}[cf.\ Theorem \ref{thm:projuni}]
For all $d\ge 0$ large enough, there is a family of inscribed projectively unique $d$-polytopes $\mathrm{PCCTP}^{\hspace{.04em} \mathrm{in}}_{d} [n],\, n\ge 1,$ with $12(n+1)+d+\mathrm{const.}$ vertices each.
\end{thmmain}

\noindent\textbf{Ansatz.} Our construction method is novel in polytope theory: It relies on the construction of convex hypersurfaces by solving the Cauchy problem for a collection of (discrete) partial differential equations.

Our work starts with the simple observation that every realization of the $3$-cube is determined by any seven of its vertices.
(A sharpened version of this will be provided in Lemma~\ref{lem:cubecmpl}.)
Thus there are cubical complexes, and indeed cubical $4$-polytopes,
for which rather few vertices of a realization successively determine all the others.
For example, for the neighborly-cubical $4$-polytopes $\mathrm{NCP}_4[n]$, which have $f_3=(n-2)2^{n-2}$ facets,
any realization is determined by any vertex, its $n$ neighbors, and the $\binom n2$ vertices at distance~$2$, that is,
by $1+n+\binom n2$ of the $f_0=2^n$ vertices.

In order to obtain \emph{arbitrarily large} cubical complexes determined by a \emph{constant} number of vertices, 
we start with the standard unit cube tiling $\RC$ of $\R^3$; let  the $i$th “layer” of this tiling for $i\in\Z$ be given by all the 
$3$-cubes whose vertices have sum-of-coordinates between $i$ and $i+3$, and thus centers with sum-of-coordinates equal to $i+\frac32$.
Successively realizing all the $3$-cubes of the abstract cubical $3$-complex given by $\RC$, starting from those at level~$0$, amounts to a solution of a
Cauchy problem for Q-nets, as studied by Adler \& Veselov~\cite{AdlVes} in a “discrete differential geometry” setting.
Here the number of initial values, namely the vertex coordinates for the cubes in the layer~$0$ of~$\RC$, is still infinite.
However, if we divide the standard unit cube tiling $\RC$ by a suitable $2$-dimensional integer lattice $\Lambda_2$ spanned by two vectors
of sum-of-coordinates~$0$, we obtain an infinite (abstract) cubical $3$-complex $\RC/\Lambda_2$ that is homeomorphic to $(S^1\times S^1)\times\R$, for
which coordinates for a finite number $K$ of cubes determines all the others, where $K$ is given by the determinant of the lattice $\Lambda_2$.
Similarly, if we consider only the subcomplex $\RC[N+3]$ of the unit cube tiling formed by all cubes of layers $0$ to~$N$,
then the quotient $\RC[N+3]/\Lambda_2$ by the $2$-dimensional lattice $\Lambda_2$ is a finite $3$-complex homeomorphic to $(S^1\times S^1)\times[0,N]$
which are arbitrarily large (consisting of $K(N+1)$ $3$-cubes), where any realization is determined by the coordinates
for the vertices first layer of $K$ $3$-cubes. 
The cubical $3$-complexes $\CT[n]:=\RC[N+3]/\Lambda_2$, for $n\ge1$, will be called \emph{cross-bedding cubical tori}\footnote{In geology/sedimentology, \href{http://en.wikipedia.org/wiki/Cross-bedding}{“cross-bedding”} refers to horizontal structures that are internally composed from inclined layers---says \href{http://en.wikipedia.org}{Wikipedia}.}, short \emph{CCTs},
in the following. 
The major part of this work will be to construct realizations for the CCTs in convex position, that is, in the boundaries of $4$-dimensional polytopes.
For this our construction will be inspired by the fibration of $S^3$ into Clifford tori, as used in Santos' work~\cite{Santostriang, Santos}. We note that in~\cite{Santostriang}, the idea to construct polytopal complexes along tori is used to obtain a result related to ours in spirit: Santos provides simplicial complexes that admit only few geometric bistellar flips, whereas we construct complexes (and indeed, polytopes) with few degrees of freedom w.r.t.\ possible realizations.

\medskip

\noindent\textbf{Outline of the paper.} We now sketch the main steps of the paper.

\noindent In \textbf{Section~\ref{sec:nota}}, we review basic facts and definitions about polytopes, polytopal complexes and realization spaces of polytopes.  

\noindent In \textbf{Section~\ref{sec:bblocks}}, we define the family of \emph{cross-bedding cubical tori}, short \emph{CCTs}: 
Throughout the paper, $\CT[n]$ will denote a CCT of \emph{width} $n$, which is a cubical complex on $12(n+1)$ vertices. 
It is of dimension~$3$ for $n\ge  3$. These complexes allow for a natural application of Lemma~\ref{lem:cubecmpl} and 
form the basic building blocks for our constructions. For our approach to the main theorems we use a class of very 
symmetric geometric realizations of the CCTs, the \emph{ideal CCTs}.

\noindent In \textbf{Section~\ref{sec:Lowdim}} we construct in four steps the family of 
convex $4$-polytopes $\mathrm{CCTP}_4[n]$ of Theorem~\ref{mthm:Lowdim}. 
In order to avoid potential problems with unboundedness, we perform this construction in~$S^4$, 
while measuring “progress” with respect to the Clifford torus fibration of the equator $3$-sphere:

\begin{compactenum}[(1)]
\item We start with a CCT $\PS[1]$ in~$S^4$. The first two extensions of~$\PS[1]$ will be treated manually. Thus we obtain $\PS[3]$, an ideal CCT in convex position in~$S^4$. 
\item We prove that our extension techniques apply to $\PS[n-1],\ n\ge  4$, providing the existence of a family of 
polytopal complexes $\PS[n]$ in~$S^4$. The proof works in the following way: We project $\PS[n-1]$ to the equator $3$-sphere, 
prove the existence of the extension, and lift the construction back to $\PS[n]$ in~$S^4$. The existence of the extension is 
the most demanding part of the construction, even though it uses only elementary spherical geometry, since we have to ensure 
that new facets of $\PS[n]$ intersect only in ways predicted by the combinatorics of the complex. 
\item Now that we have constructed the complexes $\PS[n]$, we need to verify that they are in \emph{convex position}, 
i.e.\ that every $\PS[n]$ is the subcomplex of the boundary complex of a convex polytope. A natural corollary of 
the construction is that the $\PS[n]$ are in \emph{locally convex position} (i.e.\ the star of each vertex is 
in convex position). A theorem of Alexandrov and van Heijenoort~\cite{Heij} states that, for $d\ge  3$, 
a locally convex $(d-1)$-manifold without boundary in $\R^d$ is in fact the boundary of a convex body. 
As the complexes $\PS[n]$ are manifolds with boundary, we need a version of the Alexandrov--van Heijenoort Theorem 
for polytopal manifolds with boundary. We provide such a result (Theorem~\ref{thm:locglowib} in Section~\ref{ssc:convex}), 
and use it to prove that the complexes constructed are in convex position. 
\item Next, we introduce the family of polytopes $\mathrm{CCTP}_4[n]:=\conv \PS[n],\ n\ge  1$. 
This is the family announced in Theorem~\ref{mthm:Lowdim}: The realization space of $\mathrm{CCTP}_4[n]$ 
is naturally embedded into the realization space of $\PS[n]$, which in turn is embedded in the realization space of $\PS[1]$, 
in particular, $\dim \cR (\mathrm{CCTP}_4[n])$ is bounded from above by $ \dim \cR (\PS[1])$. 
\end{compactenum}

\noindent In \textbf{Section~\ref{sec:projun}}, we turn to the proof of Theorem~\ref{mthm:projun}. 
The idea is to use Lawrence extensions (cf.\ Richter-Gebert~\cite[Sec.\ 3.3.]{RG}), which produce 
projectively unique polytopes from projectively unique polytope--point configurations. In order to circumvent difficulties from the fact that we have only recursive descriptions of the $\mathrm{CCTP}$ available, 
we introduce the notion of \emph{weak projective triples} (Definition~\ref{def:wpt}).  Weak projective triples 
give rise to projectively unique polytope--point configurations by a variant of the wedge construction 
of polytopes (Definition~\ref{def:subd} and Lemma~\ref{lem:subdc}). To the resulting configuration we can then apply 
Lawrence extensions, obtaining the desired family of projectively unique polytopes $\mathrm{PCCTP}_{69}[n]$.

Weak projective triples and the wedge construction, as introduced in Section~\ref{sec:projun} 
of this paper, 
have already been employed successfully in subsequent work: In~\cite{AP}, they are used to prove a universality theorem for projectively unique polytopes and to provide polytopes that are not subpolytopes of any stacked polytope, thus disproving a conjecture of Shephard~\cite{Shephard74} and Kalai~\cite[p.\ 468]{Kalai}~\cite{KalaiKyoto}.

\noindent Finally, in the \textbf{Appendix} we provide the following:

\begin{compactitem}[$\circ$]
\item In \textbf{Section~\ref{sec:convps}}, we discuss the notion of a polytopal complex in (locally) convex position, and establish an extension of the Alexandrov--van Heijenoort Theorem for polytopal manifolds with boundary.
\item \textbf{Section~\ref{ssc:altproof}} sketches an alternative proof of Main Theorem~\ref{mthm:Lowdim}, which does not need the Alexandrov--van Heijenoort Theorem. The appeal of this approach, which builds on the Maxwell--Cremona relation between reciprocals and liftings, is that it proves that extensions of CCTs are naturally in convex~position.
\item In \textbf{Section~\ref{ssc:Shphrdlist}}, we record Shephard's (conjecturally complete) list of projectively unique $4$-polytopes.
\item An explicit recursion formula for vertex coordinates of the polytopes $\mathrm{CCTP}_4[n]$ is given in \textbf{Section~\ref{ssc:expformula}}.
\item \textbf{Section~\ref{sec:varieties}} presents constructions of even more projectively unique polytopes, including such polytopes with rational vertex coordinates, and projectively unique polytopes inscribed to the sphere.
\item In \textbf{Section~\ref{sec:Lemmas}}, we provide proofs for two lemmas that were deferred in order to get a more transparent presentation for the proof of Theorem~\ref{mthm:Lowdim}.
\end{compactitem}

\begin{small}
	\tableofcontents
\end{small}
\enlargethispage{4mm}
\section{Set-up}\label{sec:nota}
We consider $\R^d$ with the standard orthonormal basis $(e_1,\,\dots,\,e_d)$, and the (unit) sphere $S^d\subseteq \R^{d+1}$ with the canonical intrinsic space form metric induced by the euclidean metric on $\R^{d+1}$. For a point $x$ in~$S^d$ or in $\R^{d+1}$, we denote the coordinates of $x$ with respect to the canonical basis $(e_1,\,\dots,\,e_{d+1})$ by~$x_i$ for~$1\le  i\le  d+1$. 
 
A \Defn{subspace} of $S^d$ shall denote the intersection of a linear subspace in $\R^{d+1}$ with $S^d$. The \Defn{upper hemisphere} $S^d_+$ of $S^d$ is the open hemisphere with center $e_{d+1}$. The \Defn{equator} $S^{d-1}_\eq:=\parti S^d_+$ is the $(d-1)$-sphere, considered as a subspace of $S^d$. For points in $S^d_+\subset S^d$ we use homogeneous coordinates, that is, we normalize the last coordinate of a point in $S^d_+$ to $1$.
 
A \Defn{polytope} in $\R^d$ is the convex hull of finitely many points in $\R^d$. A \Defn{(spherical) polytope} in $S^d$ is the convex hull of a finite number of points in some open hemisphere of $S^d$. Polytopes in a fixed open hemisphere of $S^d$ are in one-to-one correspondence with polytopes in $\R^d$ via central resp.\ radial projection.

\begin{definition}[The realization space of a polytope, cf.\ Richter-Gebert~\cite{RG}]\label{def:realization_space}
 Let $P$ be a convex $d$-polytope, and consider the vertices of $P$ labeled with the integers from $1$ to $n=f_0 (P)$. A $d$-polytope $P'\subset S^d$ with a labeled vertex set is said to \Defn{realize} $P$ if there exists an isomorphism between the face lattices of $P$ and $P'$ respecting the labeling of their vertex sets. We define the \Defn{realization space of $P$} as 
\[ \cR(P):=\big\{V\in(S^{d})^{n}:\conv(V)\text{ realizes }P \big\},\]
that is, as the set of vertex descriptions of realizations of $P$. The realization space is a primary semialgebraic subset of $(S^{d})^{n}$ defined over $\Z$~\cite{Grunbaum}~\cite{ZA}; in particular, its \Defn{dimension} is well-defined. 
\end{definition}

We denote by $\Sp X$ the linear span of a set $X$ in $\R^d$. Likewise, the span of a subset $X$ of $S^d$, that is, the minimal subspace of $S^d$ containing $X$, shall be denoted by $\SSp X$. Similarly, $\conv X$ shall denote the convex hull of a set $X$, and $\cl X$, $\intx X$, $\rint X$ and $\parti X$ shall denote the closure, interior, relative interior and boundary of $X$ respectively. The \Defn{orthogonal projection} from $S^d{\setminus} \{\pm e_{d+1}\}$ to the equator $S^{d-1}_\eq$ associates to $x\in S^d{\setminus} \{\pm e_{d+1}\}$ the unique point $\pp(x)\in S^{d-1}_\eq$ that minimizes the distance to $x$ among all the elements of $S^{d-1}_\eq$. If $x,y$ are two points in $\R^d$ or two non-antipodal points in  $S^d$, then $[x,y]$ is the \Defn{segment} (i.e.\ the shortest path) from $x$ to $y$, parametrized by unit speed. Following a convention common in the literature, we will not strictly differentiate between a segment and its image in $S^d$. The midpoint $m(x,y)$ of a segment $[x,y]$ is the unique point $m$ in $[x,y]$ whose distance to $x$ equals its distance to $y$. The angle between segments sharing a common starting point $x$ is the angle between the tangent vectors of the segments at $x$; in particular it takes a value in the interval $[0,\pi]$. Finally, we denote the $d \times d$ identity matrix by $\mathrm{I}_{d}$, and the matrix representing the reflection at the orthogonal complement of a vector $\nu$ in $\R^d$ by $\spi^\nu_{d}$. For example, $\spi^{e_4}_{5}$ coincides with the diagonal $5\times 5$ matrix with diagonal entries $(1,1,1,-1,1)$.

Furthermore, we set 
\[
R(\beta)\,:=\,
\left(\begin{array}{cc} 
\cos{\beta} & -\sin{\beta} \\
\sin{\beta} & \cos{\beta} \end{array} \right)
\]
\noindent With this, we define the following rotations in~$S^4$ resp.\ $\R^5$:
\[\rot_{1,2}\,:=\,
\left(\begin{array}{ccc} 
R(\nicefrac{\pi}{2}) & 0		&0    \\
0 & \mathrm{I}_2		&0    \\
0                 & 0 & 1\end{array} \right),
 \qquad
\rot_{3,4}\,:=\, 
\left(\begin{array}{ccc} 
\mathrm{I}_2 & 0		&0    \\
0 & R(\nicefrac{\pi}{3})	&0    \\
0                 & 0 & 1 \end{array} \right).
\]

A \Defn{(geometric) polytopal complex} in $\R^d$ (resp.\ in $S^d$) is a collection of polytopes in $\R^d$ (resp.~$S^d$) such that the intersection of any two polytopes is a face of both (cf.~\cite{RourkeSanders}). Our polytopal complexes are usually finite, i.e.\ the number of polytopes in the collection is finite. An \Defn{abstract} polytopal complex is a collection of polytopes that are attached along isometries of their faces (cf.~\cite[Sec.\ 2]{DM-NP}). Two polytopal complexes $C,\, C'$ are \Defn{combinatorially equivalent} if their face posets are isomorphic.

If $v$, $u$ are vertices of $C$ connected by an edge of $C$, then we denote this edge by $[v,u]$, in an intuitive overlap in notion with the segment from $v$ to $u$. A polytope combinatorially equivalent to the regular unit cube $[0,1]^k$ shall simply be called a \Defn{cube}, and a polytopal complex is \Defn{cubical} if all its faces are cubes. The set of dimension $k$ faces of a polytopal complex $C$ will be denoted by $\F_k(C)$, and the cardinality of this set is denoted by $f_k(C).$

Let $C$ be a polytopal complex, and let $\sigma$ be a face of $C$. The \Defn{star} of $\sigma$ in $C$, denoted by $\St(\sigma, C)$, is the minimal subcomplex of $C$ that contains all faces of $C$ containing $\sigma$. If $C$ is a geometric complex in $X=S^d$ or $X=\R^d$, let $\NO_\sigma X$ denote the subspace of the tangent space $\TT_p X$ of $X$ at a point $p\in \rint\sigma$ spanned by tangent vectors orthogonal to $\TT_p \sigma$. The subspace $\NO^1_\sigma X$ of unit tangent vectors in $\NO_\sigma X$ is isometric to the unit sphere of dimension $d-\dim \sigma-1$. If $\tau$ is any face of $C$ containing $\sigma$, then the set of unit tangent vectors in $\NO^1_\sigma X$ pointing towards $\tau$ forms a polytope in $\NO^1_\sigma X$. The collection of all polytopes in $\NO^1_\sigma X$ obtained this way forms a polytopal complex in the $(d-\dim \sigma-1)$-sphere $\NO^1_\sigma X$, denoted by $\Lk(\sigma, C)$, the \Defn{link} of $C$ at $\sigma$. This is well-defined: Up to isometry, $\NO_\sigma X,\ \NO^1_\sigma X$ and $\Lk(\sigma, C)$ do not depend on the choice of the point $p$ in $\rint\sigma$. 

\vskip -2mm

\begin{rem*}
Take care that there are two notions of ``link'' in the literature; alternative to the notion adopted here, the complex $\{\tau\in \St(\sigma,C): \tau \cap \sigma =\emptyset\}$ is sometimes called the link of $\sigma$ as well (for instance by Gr\"unbaum \cite{Grunbaum}). The notion we use here is more reminiscent of tangent spaces in differential geometry, and in particular more prevalent in geometric group theory (compare for instance \cite{Gromov}).
\end{rem*}

\vskip -2mm
\enlargethispage{3mm}

\setcounter{figure}{0}
\section{Cross-bedding cubical tori}\label{sec:bblocks}

In this section, we define our basic building blocks for the proofs of Theorem~\ref{mthm:Lowdim} and, ultimately, Theorem~\ref{mthm:projun}. 
These building blocks, called \Defn{cross-bedding cubical tori}, short \Defn{CCTs}, are cubical complexes homeomorphic to $(S^1\times S^1) \times I$. 
They are obtained as quotients of periodic subcomplexes of the regular unit cube tiling in~$\R^3$. 
The section is divided into three parts: We start by defining CCTs abstractly (Section~\ref{ssc:trc}), 
then define the particular geometric realizations of CCTs used in our construction (Section~\ref{ssc:trg}) 
and close by observing some properties of these geometric realizations (Section~\ref{ssc:trp}). 
Before we start, let us remark that it is not a priori clear that CCTs exist as geometric polytopal complexes 
satisfying the constraints that we define in Section~\ref{ssc:trg}. Explicit examples will be obtained in Section~\ref{ssc:example}.

\subsection{Cross-bedding cubical tori, the combinatorial picture}\label{ssc:trc}

We will now define our building blocks, first as abstract cubical complexes, obtained as quotients of infinite complexes by a lattice. 
We then provide Lemma~\ref{lem:cubecmpl}, a sharpened version of the observation that any seven vertices of a $3$-cube determine
the eighth one, and observe that CCTs allow for a direct application of this lemma.

\begin{definition}[Cross-bedding cubical torus, CCTs, $k$-CCTs]\label{def:CCT}
Let $\RC$ be the unit cube tiling of $\R^3$ on the vertex set $\Z^3\subset \R^3$. For $k\ge  0$, let $\RC[k]$ denote the subcomplex on the vertices $v$ of $\RC$ with $0\le \langle v, \mathbbm{1} \rangle\le  k$. The complexes $\RC$ and $\RC[k]$ are invariant under translations by vectors $(1,-1,0)$, $(1,0,-1)$ and $(0,1,-1)$, and in particular under translation by $(3,-3,0)$ and $(-2,-2,4)$. A cubical complex $\CT$ in some $\R^d$ or $S^d$ is called a \Defn{$k$-CCT} (\Defn{cross-bedding cubical torus of width $k$}) if it is combinatorially equivalent to the abstract polytopal complex \[\T[k]:=\bigslant{\RC[k]}{(3,-3,0)\Z \times (-2,-2,4)\Z.}\]
A \Defn{CCT}, or \Defn{cross-bedding cubical torus}, is a finite polytopal complex that is a $k$-CCT for some $k\ge  0$. The vertices of $\T[k]$ are divided into layers $\TT_\ell[k],\ 0\le \ell \le k$, defined as the sets of vertices $\widetilde{v}$ of $\T$ such that for any representative $v$ of $\widetilde{v}$ in $\RC$, we have $\langle v, \mathbbm{1} \rangle = \ell$. For a $k$-CCT $\CT$, let $\varphi_{\CT}$ denote the isomorphism 
\[\varphi_{\CT}: \T[k] \longrightarrow \CT\]
The $\ell$-th \Defn{layer} of $\CT$, $0\le  \ell\le  k$ of $\CT$ is defined to be the vertex set $\RR(\CT,\ell):=\varphi_{\CT} (\TT_\ell)$, the \Defn{restriction} of~$\CT$ to the $\ell$-th layer. More generally, if $I$ is a subset of $\Z$, then we denote by $\RR(\CT,I)$ the \Defn{restriction} of~$\CT$ to the subcomplex induced by the vertices $\bigcup_{i\in I}\, \RR(\CT,i)$.
\end{definition}
 
\begin{rem}[On cross-bedding cubical tori]\label{rem:prp} $\quad$
\begin{compactitem}[$\circ$]

\item The $f$-vector $f(\CT)=(f_0(\CT),\,f_1(\CT),\,f_2(\CT),\,f_3(\CT))$ of any CCT is given by $f(\T[0])=\big(12,\, 0,\, 0,\, 0)$, $f(\T[1])=\big(24,\, 36,\, 0,\, 0)$ and, for $k\ge  2$, \[f(\T[k])=\big(12(k+1),\,36k,\, 36(k-1),\, 12(k-2)  \big).\]
\item A $0$-CCT consists of $12$ vertices, and no faces of higher dimension. A $1$-CCT is a bipartite $3$-regular graph on $24$ vertices. For $k\ge  2$, any $k$-CCT is homotopy equivalent to the $2$-torus. If $k=2$, it is even homeomorphic to the $2$-torus, and if $k\ge  5$, it is homeomorphic to the product of the $2$-torus with an interval. 
\end{compactitem}
\end{rem}

\begin{definition}[Extensions]
If $\CT$ is a $k$-CCT in some euclidean or spherical space, a CCT $\CT'$ of width $\ell>k$ is an \Defn{extension} of $\CT$ if $\RR(\CT',[0,k])=\CT$, where $[0,k]=\{0,\,1,\,\dots,\,k\}$. If $\ell=k+1$, then $\CT'$ is an \Defn{elementary extension} of $\CT$.
\end{definition}

The following two lemmas will show that extensions are unique.

\begin{lemma}\label{lem:cubecmpl}
Let $Q_1$, $Q_2$, $Q_3$ be three quadrilaterals in some euclidean space (or in some sphere) on vertices $\{a_1,\, a_2,\, a_3,\, a_4\}$, $\{a_1,\, a_4,\, a_5,\, a_6\}$ and $\{a_1,\, a_2,\, a_7,\, a_6\}$ respectively, such that the three quadrilaterals do not lie in a common plane. If $Q'_1$, $Q'_2$, $Q'_3$ is any second such triple with the property that $a_i=a'_i$ for all $i\in \{2,\, \dots,\, 7\}$, then we have $a_1=a'_1$ and $Q_j=Q'_j$  for all $j\in\{1,\,2,\,3\}$. In other words, the coordinates of the vertex $a_1$ can be reconstructed from the coordinates of the vertices $a_i,\, i\in \{2,\, \dots,\, 7\}$. \emph{\qed}

\begin{figure}[htbf]  
\centering 
  \includegraphics[width=0.5\linewidth]{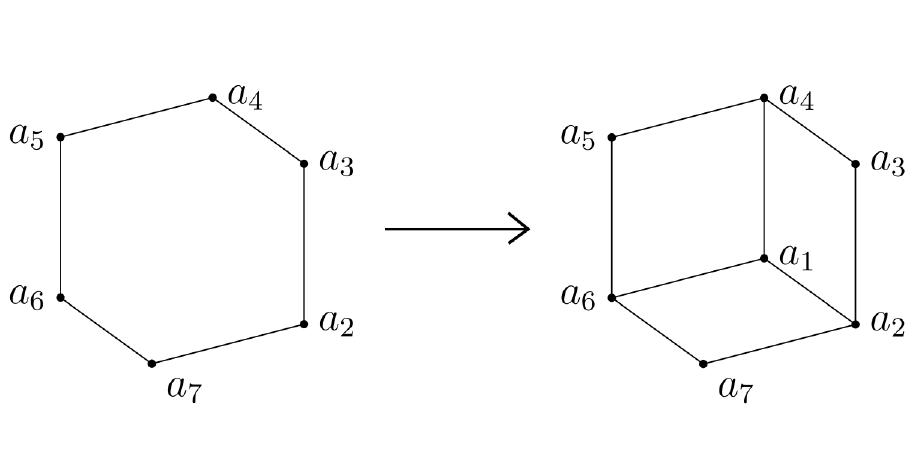} 
  \caption{\small  The vertex $a_1$ of can be reconstructed from the remaining vertices.} 
  \label{fig:cube}
\end{figure}
\end{lemma}

\begin{lemma}[Unique Extension, I]\label{lem:uniext}
Let $\CT$ be a CCT of width $k\ge  2$ in $\R^d$ or $S^d$. If $\CT'$ and $\CT''$ are two elementary extensions of $\CT$, then $\CT''$ coincides with $\CT'$.
\end{lemma}

\begin{proof}
Let $W$ denote any facet of $\CT'$ that is not in $\CT$. Then the geometric realization of $W$ is determined by $\CT$: $W\cap \CT$ consists of three $2$-faces of $W$, in particular, by Lemma~\ref{lem:cubecmpl}, $W$ is determined by $\CT$. The same applies to the facets of $\CT''$ not in $\CT$. Thus $\CT'$ and $\CT''$ coincide.
\end{proof}

\begin{rem}
The assumption $k\ge 2$ in Lemma~\ref{lem:uniext} can be weakened to $k\ge 1$ if the vertices of $\CT$ are in sufficiently general position: If for all vertices $v$ of $\RR(\CT',2)$ the faces of $\St(v,\CT')$ are \emph{not} coplanar, then $\CT'$ is uniquely determined by $\CT$.
\end{rem}

\subsection{Cross-bedding cubical tori, the geometric picture}\label{ssc:trg}

In this section we give a geometric framework for CCTs. The control structures for our construction are provided by the weighted Clifford tori in~$S^3_\eq$.

\begin{definition}[Weighted Clifford tori $\mathcal{C}_\lambda$ and Clifford projections $\mathcal{C}_\lambda$]\label{def:clifpauto}
The \Defn{weighted Clifford torus} $\mathcal{C}_\lambda$, $\lambda\in [0,2]$, is the algebraic subset of~$S^3_\eq$ given by the equations $x_1^2+x_2^2=1-\nicefrac{\lambda}{2}$ and $x_3^2+x_4^2=\nicefrac{\lambda}{2}$. If $\lambda$ is in the open interval $(0,2)$, then $\mathcal{C}_\lambda$ is a flat $2$-torus. In the extreme cases $\lambda=0$ and $\lambda=2$, $\mathcal{C}_\lambda$ degenerates to (isometric copies of) $S^1$. 

For $\lambda\in (0,2)$, we define 
\begin{eqnarray*}
\pi_\lambda :\quad	S^3_\eq{\setminus} (\mathcal{C}_0\cup \mathcal{C}_2) &\longrightarrow& \mathcal{C}_\lambda,\\ 
	y & \longmapsto & \big(\mu y_1, \mu y_2, \nu y_3, \nu y_4\big),\ \ \mu =\sqrt{ \frac{1-\tfrac{\lambda}{2}}{y_1^2+y_2^2}},\ \ \nu=\sqrt{\frac{\tfrac{\lambda}{2}}{y_3^2+y_4^2}}.
\end{eqnarray*}
Similarly, we define 
\begin{eqnarray*}
\pi_0 :\quad S^3_\eq{\setminus} \mathcal{C}_2 &\longrightarrow& \mathcal{C}_0,\qquad
	y \ \longmapsto \ \frac{1}{\sqrt{y_1^2+y_1^2}}\big(y_1, y_2, 0, 0\big),\\ 
\noalign{\noindent and}
\pi_2 :\quad S^3_\eq{\setminus} \mathcal{C}_0 &\longrightarrow& \mathcal{C}_2,\qquad
	y \ \longmapsto \ \frac{1}{\sqrt{y_3^2+y_4^2}}\big(0, 0, y_3, y_4\big). 
\end{eqnarray*}
The maps $\pi_\lambda$ are the \Defn{Clifford projections}.
\end{definition}

We now define restrictions on the geometry of CCTs that will be crucial for the most difficult steps in our proof of Theorem~\ref{mthm:Lowdim}, namely 
to establish that extensions of the CCTs \emph{exist} (Section~\ref{ssc:extension}) and that they are \emph{in convex position} (Section~\ref{ssc:convex}). 
Here is a preview on how we will use the restrictions:
\begin{compactitem}[$\circ$]
\item The property of \emph{symmetry} (Definition~\ref{def:symrig}) reduces the construction and analysis of the symmetric CCTs to local problems in a fundamental domain of the symmetry. 
\item The property of \emph{transversality} (Definition~\ref{def:trnsm}) reduces several steps of our study to ``planar'' problems in the torus $\mathcal{C}_1$. It is critically used in Lemma~\ref{lem:localem} (an important step in the proof that extensions of CCTs exist) and in Proposition~\ref{prp:loccrt1lay} (which provides a basic local to global convexity result for CCTs).
\item A CCT is \emph{ideal} (Definition~\ref{def:slor}) if it satisfies, in addition to symmetry and transversality, two technical requirements: The first requirement is an inequality on a quantity that we call the \emph{slope} of a symmetric transversal CCT. This inequality is crucially used in the proof that extensions exist (Proposition~\ref{prp:locatt}). The second condition, \emph{orientation}, is merely a convention that simplifies the notation, and ensures that we extend a CCT into a fixed direction. 
\end{compactitem} 

\begin{figure}[htbf]
\centering 
 \includegraphics[width=0.95\linewidth]{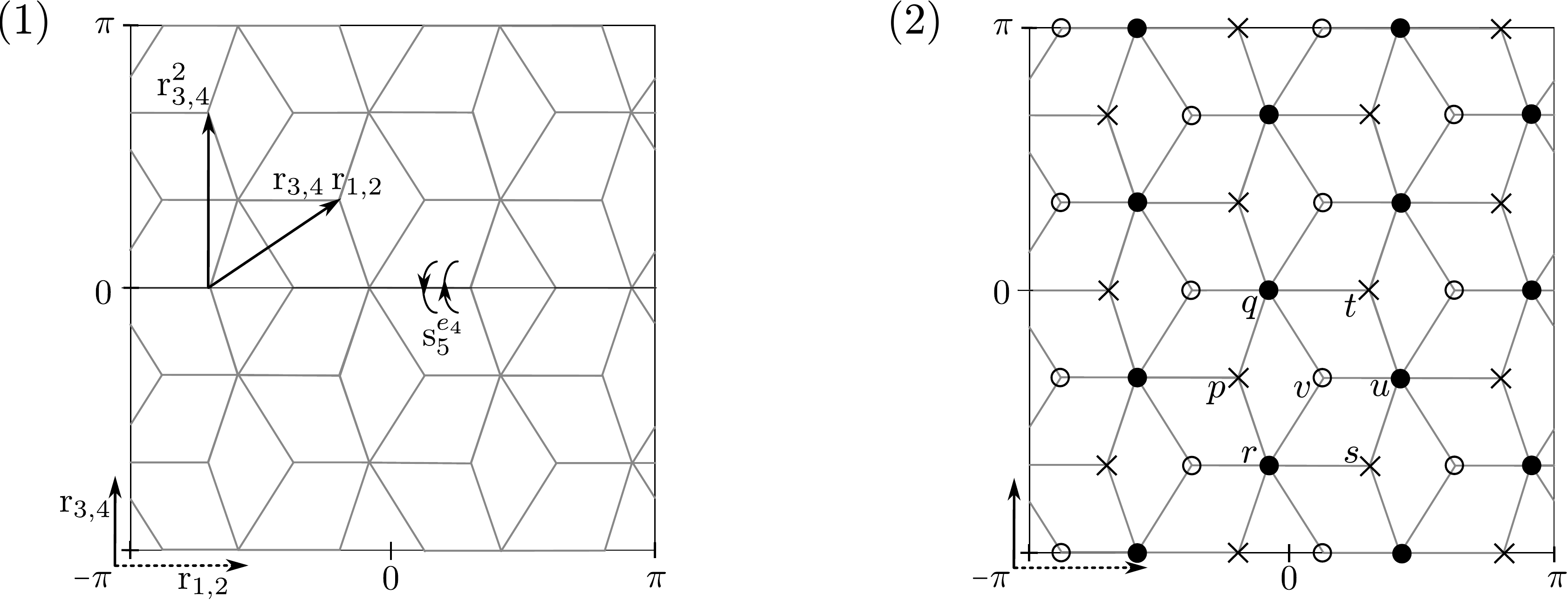} 
 \caption{\small A picture of $\pi_1(\CT)\subset \mathcal{C}_1$ for a transversal symmetric CCT $\CT$ of width $2$ in~$S^3_\eq$. \newline (1) We indicate the action of the symmetries $\mathfrak{S}$ on $\CT$. The compass rose in the lower left corner indicates the action of the rotations $\rot_{1,2} $ of the $\Sp\{e_1,\,e_2\}$-plane and the  action of the rotations $\rot_{3,4}$ of the $\Sp\{e_3,\,e_4\}$-plane. \newline (2) We labeled the images of vertices of layer $0$ by $\circ$, vertices of layer $1$ by $\bullet$ and of layer $2$ by $\times$.}
  \label{fig:protorus}
\end{figure}

\begin{definition}[Symmetric CCTs]\label{def:symrig}
A CCT $\CT$ in~$S^3_\eq$ or~$S^4$ is \Defn{symmetric} if it satisfies the following three conditions:
\begin{compactenum}[\rm(a)]
\item The automorphism $\spi^{(1,-1,0)}_{3}$ of $\T[k]$ corresponds to the reflection $\spi^{e_4}_{5}:\CT\rightarrow \CT$. In terms of the restriction 
$\varphi_{\CT}:\T[k]\rightarrow \CT$ of Definition~\ref{def:CCT}, 
\[\spi^{e_4}_{5}(\varphi_{\CT}(\T[k]))=\varphi_{\CT} (\spi^{(1,-1,0)}_{3}\T[k]).\] 
\item The automorphism of $\T[k]$ induced by the translation by vector $(-1,1,0)$ corresponds to the rotation $\rot_{3,4}^2:\CT\rightarrow \CT$, i.e.\  
\[\rot_{3,4}^2(\varphi_{\CT}(\T[k]))=\varphi_{\CT} (\T[k]+(-1,1,0)).\]
\item The automorphism of $\T[k]$ induced by the translation by vector $(0,-1,1)$ corresponds to the rotation $\rot_{3,4}\rot_{1,2}:\CT\rightarrow \CT$, i.e.\  
\[\rot_{3,4}\rot_{1,2}(\varphi_{\CT}(\T[k]))=\varphi_{\CT} (\T[k]+(0,-1,1)).\] 
\end{compactenum}
\end{definition}

\begin{definition}[Groups of symmetries $\mathfrak{S}$ and $\mathfrak{R}$]
The subgroup $\mathfrak{S}$ of $O(\R^5)$ is generated by the elements $\spi^{e_4}_{5}$, $\rot_{3,4}^2$ and $\rot_{3,4}\rot_{1,2}$. 
Its subgroup $\mathfrak{R}$ is generated by the rotations $\rot_{3,4}^2$ and $\rot_{3,4}\rot_{1,2}$.
\end{definition}

\begin{obs} The group $\mathfrak{R}$ acts simply transitively on the vertices of each layer of a symmetric CCT.
	In particular, $|\mathfrak{R}|=12$ and $|\mathfrak{S}|=2|\mathfrak{R}|=24$.
\end{obs}

\begin{obs}\label{obs:fix}
The translations by vectors $(-1,1,0)$ and $(-1,-1,2)$ induce fixed-point free actions on $\T[k]$. Thus if $\CT$ is a symmetric CCT in~$S^4$ or~$S^3_\eq$, then the actions of $\rot_{3,4}^2$ and $\rot_{1,2}^2$ on $\mathrm{T}$ are fixed point free. In particular, $\CT$ does not intersect $\Sp\{e_1,\, e_2,\, e_5\}\cup \Sp\{e_3,\, e_4,\, e_5\}$. 
\end{obs}

\begin{definition}[Control CCT]\label{def:ctrcct}
Let $\CT$ denote a symmetric CCT in~$S^4$. The orthogonal projection of $S^4{\setminus} \{\pm e_5\}$ to~$S^3_\eq$ is well-defined on $\CT$ by the previous observation. If the projection is furthermore injective on $\CT$, then the CCT arising as projection of $\CT$ to $S^3_\eq$ is a symmetric CCT, the \Defn{control CCT} of~$\CT$ in $S^3_\eq$.
\end{definition}

\begin{definition}[Transversal symmetric CCTs]\label{def:trnsm}
A symmetric $2$-CCT $\CT$ in~$S^3_\eq$ is \Defn{transversal} if the Clifford projection $\pi_1: S^3_\eq{\setminus}(\mathcal{C}_0\cup \mathcal{C}_2)\rightarrow \mathcal{C}_1 $ is injective on $\CT$. A symmetric $k$-CCT $\CT$ in~$S^3_\eq$  is \Defn{transversal} if  $\RR(\CT,[i-1,i+1])$ is transversal for all $i\in [1,k-1]$. A symmetric CCT $\CT$ in~$S^4$ is \Defn{transversal} if its control CCT in~$S^3_\eq$ is well-defined and transversal.
\end{definition}

\begin{figure}[htbf]
\centering 
  \includegraphics[width=0.72\linewidth]{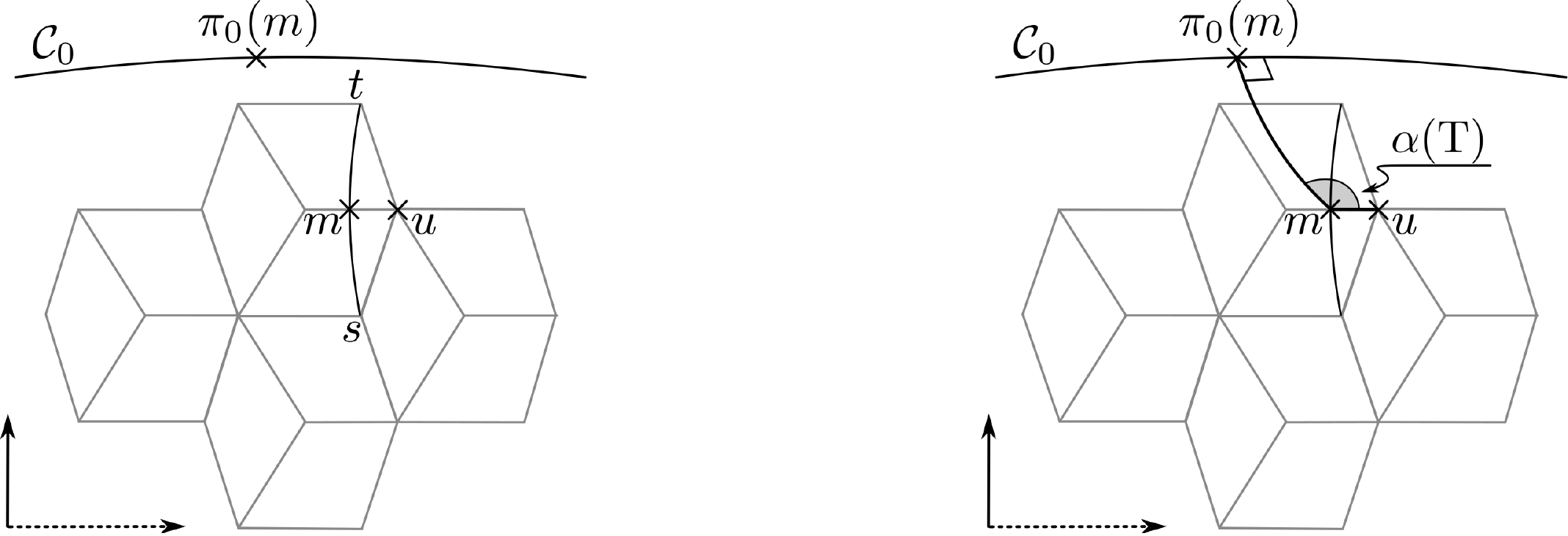}
  \caption{\small  The definition of the slope $\alpha(\CT)$ of a symmetric CCT: $\alpha(\CT)$ is defined as the angle between the segments $[m,\pi_0 (m)]$ and $[m,u]$ at $m$.}
  \label{fig:defslope}
\end{figure}

\begin{definition}[Slope, orientation and ideal CCTs]\label{def:slor}

Let $s$ be any vertex of layer $2$ of a transversal symmetric $2$-CCT $\CT$ in $S^3_\eq$. Then $t:=\rot_{3,4}^2 s$ is in $\RR(\CT,2)$, and $s$ and $t$ are connected by a unique length $2$ edge-path in $\CT$ whose middle vertex $u$ lies in $\RR(\CT,1)$, cf.\ Figure~\ref{fig:defslope}. If $m=m(s,\,t)$ is the midpoint of the geodesic segment $[s,t]$ in $S^3_\eq$, then the angle $\alpha$ between the segments $[m,\pi_0 (m)]$ and $[m,u]$ at $m$ is called the \Defn{slope} $\alpha(\CT)$ of $\CT$.

A transversal symmetric $k$-CCT $\CT$, $k\ge  2$ in $S^3_\eq$ is \Defn{ideal} if $\alpha(\RR(\CT,[k-2,k]))>\nicefrac{\pi}{2}$ and it is \Defn{oriented} towards $\mathcal{C}_0$, i.e.\ the component of $S^3_\eq{\setminus} \CT$ containing $\mathcal{C}_0$ whose closure intersects $\RR(\CT,k)$. A transversal symmetric CCT in $\R^4$ or~$S^4$ is \Defn{ideal} if the associated control CCT in $S^3_\eq$ is ideal.
\end{definition}

\subsection{Some properties of ideal cross-bedding cubical tori}\label{ssc:trp}

Now we record some properties of ideal CCTs, which follow in particular from the conditions of transversality and symmetry (Proposition~\ref{prp:alignsymm}). Furthermore, we give a tool to check transversality (Proposition~\ref{prp:inj3}). The verification of Propositions~\ref{prp:alignsymm} and~\ref{prp:inj3} is straightforward; their proof can be found in the first author's thesis \cite{AD2013}. We close the section with a Proposition~\ref{prp:slmono}, which says that the slope is monotone under extensions.
\enlargethispage{2mm}
\begin{prp}\label{prp:alignsymm}
Let $\CT$ be a symmetric $2$-CCT in~$S^3_\eq$, let $v$ be a vertex of layer $0$ of $\CT$, and let the remaining vertices of $\CT$ be labeled as in Figure~\ref{fig:protorus}. Then
\begin{compactenum}[\rm(a)]
\item $\CT\cap (\mathcal{C}_0\cup \mathcal{C}_2)$ is empty,
\item $\pi_2(s)=\pi_2(r)$, $\pi_2(p)=\pi_2(v)=\pi_2(u)$ and $\pi_2(t)=\pi_2(q)$,
\item $\pi_0(t)=\pi_0(s)$ and $\pi_0(q)=\pi_0(r)$, and
\item $\pi_2(p)$ is the midpoint of the (nontrivial) segment $[\pi_2(s),\pi_2(t)]$.
\end{compactenum}
\noindent If $\CT$ is additionally transversal, then
\begin{compactenum}[\rm(a)]
\setcounter{enumi}{+4}
\item $\pi_0(v)$, $\pi_0(s)$, $\pi_0(r)$ lie in the interior of segment $[\pi_0(u),\pi_0(p)]$,
\item $\pi_0(s)$ and $\pi_0(v)$ lie in the interior of segment $[\pi_0(u),\pi_0(r)]$, and
\item $\pi_0(r)$ lies in the interior of segment $[\pi_0(v),\pi_0(p)]$ and segment $[\pi_0(s),\pi_0(p)]$. \emph{\qed}
\end{compactenum}
\end{prp}

For $x\in S^3_\eq{\setminus} \mathcal{C}_0$, we write $\pi^{\mathrm{f}}_2 (x)$ to denote the $\pi_2$-fiber $\pi_2^{-1}(\pi_2(x))$ in~$S^3_\eq$, 
and $\pi^{\mathrm{f}}_0(x),\ x\in S^3_\eq{\setminus} \mathcal{C}_2$ to denote the $\pi_0$-fiber $\pi_0^{-1}(\pi_0(x))\subset S^3_\eq$. 
Furthermore, we denote by $\pi^{\SSp}_2 (x),\ x\in  S^3_\eq{\setminus} \mathcal{C}_0$  the hyperplane in~$S^3_\eq$ spanned by $\pi^{\mathrm{f}}_2 (x)$, i.e.\  
\[
\pi^{\SSp}_2 (x):=\SSp(\pi^{\mathrm{f}}_2 (x))=\pi^{\mathrm{f}}_2 (x)\cup \pi^{\mathrm{f}}_2 (-x)\cup \mathcal{C}_0\subset S^3_\eq,
\] 
and similarly, for $x\in S^3_\eq{\setminus} \mathcal{C}_2$, 
\[
\pi^{\SSp}_0 (x):=\SSp(\pi^{\mathrm{f}}_0 (x))=\pi^{\mathrm{f}}_0 (x)\cup \pi^{\mathrm{f}}_0 (-x)\cup \mathcal{C}_2\subset S^3_\eq.
\] 
Then the statements of Proposition~\ref{prp:alignsymm} can be reformulated using the following dictionary:

\begin{prp}\label{prp:dict}
Let $a$,$b$ and $x$ be three points in $S^3_\eq{\setminus} \mathcal{C}_{2-i}\,,\ i\in\{0,\,2\}$. Then $\pi_i(x)$ lies in the interior of the segment $[\pi_i(a),\pi_i(b)]$ if any two of the three following statements hold:
\begin{compactenum}[\rm(a)]
\item $a$ and $b$ lie in different components of $S^3_\eq{\setminus}\pi^{\SSp}_i (x)$, 
\item $a$ and $x$ lie in the same component of 
$S^3_\eq{\setminus}\pi^{\SSp}_i (b)$, 
\item $b$ and $x$ lie in the same component of 
$S^3_\eq{\setminus}\pi^{\SSp}_i (a)$.
\end{compactenum}
\noindent Conversely, if $\pi_i(x)$ lies in the interior of the segment $[\pi_i(a),\pi_i(b)]$, then all three of the above statements hold. \emph{\qed}
\end{prp}

Finally, we present our tool for checking transversality. 

\begin{prp}\label{prp:inj3}
Let $C$ denote a cubical complex that arises as the union of three quadrilaterals on vertices $\{u,\,t,\,v,\,q\}$, $\{u,\,s,\,v,\,r\}$ and $\{p,\,q,\,v,\,r\}$ respectively, such that
\begin{compactenum}[\rm(a)]
\item $C\cap (\mathcal{C}_0\cup \mathcal{C}_2)$ is empty,
\item $\pi_2(s)=\pi_2(r)$, $\pi_2(p)=\pi_2(v)=\pi_2(u)$ and $\pi_2(t)=\pi_2(q)$,
\item $\pi_0(t)=\pi_0(s)$ and $\pi_0(q)=\pi_0(r)$,
\item $\pi_2(p)$ lies in the interior of the segment $[\pi_2(s),\pi_2(t)]$,
\item $\pi_0(v)$, $\pi_0(r)$ lie in the interior of segment $[\pi_0(u),\pi_0(p)]$, and
\item $\pi_0(s)$ lies in the interior of segment $[\pi_0(u),\pi_0(r)]$.
\end{compactenum}
\noindent Then $\pi_1$ is injective on $C$. \emph{\qed}
\end{prp} 

We close with the crucial monotonicity result for the slope. 

\begin{prp}[The slope is monotone]\label{prp:slmono}
Let $\CT$ denote an ideal CCT in~$S^3_\eq$ and let $\CT'$ denote its elementary extension. Then the slope of $\CT'$ is larger or equal to the slope of $\CT$, i.e.\ $\alpha(\CT')\ge \alpha(\CT)$.
\end{prp}

\begin{proof}
Let $D'$ denote the restriction of $\CT'$ to the top three layers, i.e.\ if $\CT$ is a $k$-CCT, then $\CT'$ is a symmetric and transversal $(k+1)$-CCT and $D':=\RR(\CT,[k-1,k+1])$. Let $s$ denote any vertex of layer $k+1$ of $D'$. The vertices $t:=\rot_{3,4}^{2}s$ and $s$ of $D'$ are connected by a unique length $2$ edge-path in $\CT'$ whose middle vertex is $u\in \RR(\CT,k)$. Let $m:=m(s,t)$ denote the midpoint of the segment $[s,t]$. Let $\beta$ denote the angle between the segment $[u,m]$ and the segment $[u,\pi_0(u)]$. We claim the following two inequalities: 
\begin{compactenum}[\rm(a)]
\item $\alpha(\CT')\ge \pi-\beta$.
\item $\beta\le \pi-\alpha(\CT)$.
\end{compactenum}
This immediately implies that $\alpha(\CT')\ge \alpha(\CT)$, and thus finishes the proof. For inequality (a), consider the convex quadrilateral on the vertices $u$, $m$, $\pi_0(u)$ and $\pi_0(m)$. The slope $\alpha(\CT')$ of $\CT'$ is the angle at $m$, the angle $\beta$ is the angle at $u$. The remaining two angles of the quadrilateral measure to $\nicefrac{\pi}{2}$. The angle-sum in a convex spherical quadrilateral is at least $2\pi$, so \[\beta + \alpha(\CT') + \nicefrac{\pi}{2} + \nicefrac{\pi}{2} \ge  2\pi\Rightarrow\alpha(\CT')\ge  \pi-\beta.\]

\begin{figure}[htbf] 
\centering 
 \includegraphics[width=0.96\linewidth]{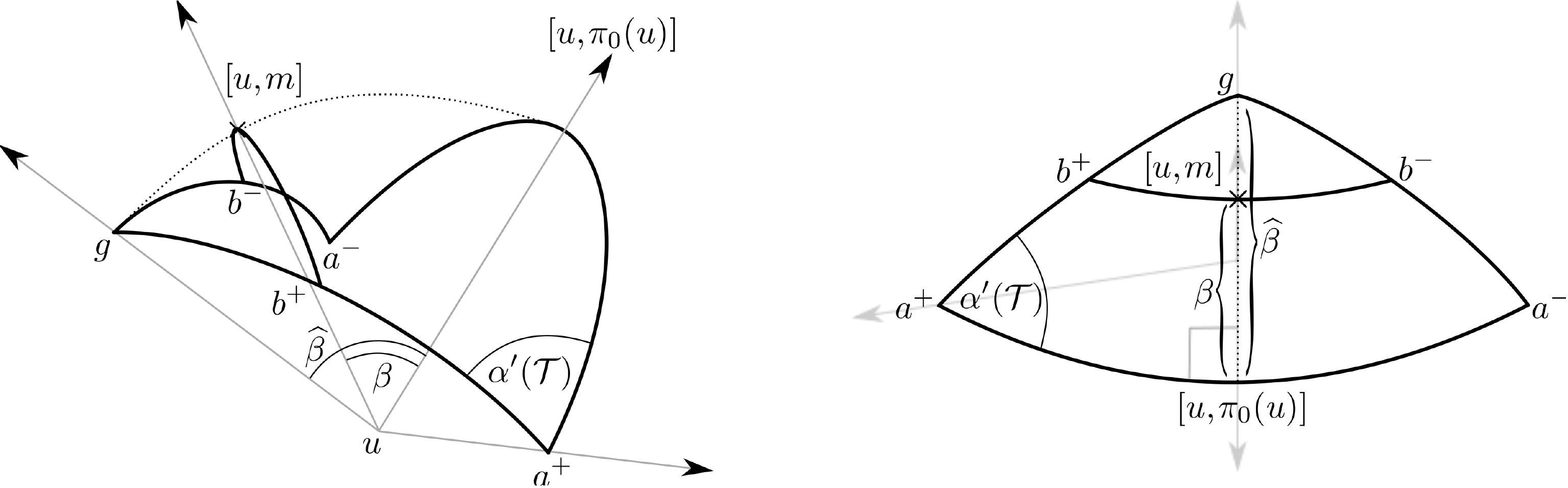} 
\caption{\small The figures show part of the sphere $\NO^1_u S^3_\eq$ from two perspectives.}
  \label{fig:sphcosine}
\end{figure}

To see inequality (b),  we work in the $2$-sphere $\NO^1_u S^3_\eq$, and identify segments $\gamma$ emanating at $u$ with their tangent direction at $u$. Consider the points $a^+:= [u,\rot_{3,4}^{2}u]$, $a^-:= [u,\rot_{3,4}^{-2}u]$, $b^+:= [u,t]$ and $b^-:= [u,s]$ in $\NO^1_u S^3_\eq$, cf.\ Figure~\ref{fig:sphcosine}. Symmetry and transversality of $D'$ imply that these four points form the vertices of a convex quadrilateral $Q$ that is symmetric under reflection at the axis $\Sp\{[u,\pi_0(u)],\,  [u,m]\} \subset \NO_u^1 S^3_\eq$. In particular, the segment $ [u,m]$ coincides with the midpoint of the edge $[b^+,b^-]\subset \NO^1_u S^3_\eq$ of $Q$, and similarly, $ [u,\pi_0(u)]$ is the midpoint of the edge $[a^+,a^-]\subset \NO^1_u S^3_\eq$ of $Q$. 

To estimate the distance $\beta$ between $ [u,m]$ in $ [u,\pi_0(u)]$ in $\NO^1_u S^3_\eq$, consider the intersection point $g$ in $\Sp\{a^+,\, b^+\} \cap \Sp \{a^-,\, b^-\}\subset \NO^1_u S^3_\eq$ for which $b^+\in [a^+,g]$, cf.\ Figure~\ref{fig:sphcosine}. Then, $ [u,m]$ lies in the segment from $g$ to $[u,\pi_0(u)]$, and thus the angle $\beta$ is bounded above by the distance $\widehat{\beta}$ of $g$ to $[u,\pi_0(u)]$ in $ S^3_\eq$. 

To bound $\widehat{\beta}$, consider the triangle in $\NO^1_u S^3_\eq$ on $[u,\pi_0(u)]$, $g$ and $a^+=[u,\rot_{3,4}^{2}u]$. Let $\eta$ denote the angle of this triangle at $g$, and notice that the angle at $[u,\pi_0(u)]$ is $\nicefrac{\pi}{2}$ by reflective symmetry at $\SSp\{[u,\pi_0(u)],\, [u,m]\}$, and that the angle at $[u,\rot_{3,4}^{-2}u]$ measures to $\pi-\alpha(\CT)$. Then the second spherical law of cosines gives \[\cos\big(\pi-\alpha(\CT)\big)=\sin\eta\cos\widehat{\beta},\]   
and since $\pi-\alpha(\CT)<\nicefrac{\pi}{2}$, we have \[\cos\big(\pi-\alpha(\CT)\big)\le \cos\widehat{\beta}\]
and consequently $\pi-\alpha(\CT)\ge \widehat{\beta}$, which gives the desired bound $\pi-\alpha(\CT)\ge \beta$.
\end{proof}

\setcounter{figure}{0}
\section{Many 4-polytopes with low-dimensional realization space}\label{sec:Lowdim}

This section is dedicated to the proof of Theorem~\ref{mthm:Lowdim}. It suffices to consider the case $d=4$; the general case follows immediately by considering iterated pyramids over the polytopes $\mathrm{CCTP}_4[n]$.
As outlined above, we will proceed in four steps:
\begin{compactitem}[$\circ$]
\item \emph{Section~\ref{ssc:example}:}
In this section, we provide the initial CCT for the construction, onto which we will build larger and larger CCTs by iterative extension. Our extension techniques for ideal CCTs develop their full power only if $k\ge  3$ (cf.\ Theorem~\ref{thm:convp}). Thus the initial example is $\PS[3]$, an ideal $3$-CCT in convex position in~$S^4$ which is constructed manually from a CCT $\PS[1]$ of width $1$.
\item \emph{Section~\ref{ssc:extension}:} Any ideal CCT can be extended:
\begin{theorem}\label{thm:ext}
Let $\CT\subset S^4$ be an ideal CCT of width $k\ge  3$. Then there exists an ideal $(k+1)$-CCT $\CT'$ extending $\CT$.
\end{theorem}

An ideal $k$-CCT $\CT$ in~$S^4$, $k\ge  2$, has an ideal elementary extension if and only if the associated ideal control CCT in~$S^3_\eq$ has an ideal elementary extension. Thus Theorem~\ref{thm:ext} is equivalent to the following theorem, which we will prove in Section~\ref{ssc:extension}.
\begin{theorem}\label{thm:exts}
Let $\CT\subset S^3_\eq$ be an ideal CCT of width $k\ge  3$. Then there exists an ideal $(k+1)$-CCT $\CT'$ extending $\CT$.
\end{theorem}
\item \emph{Section~\ref{ssc:convex}:} Extensions are in convex position:
\begin{theorem}\label{thm:convp}
Let $\CT\subset S^4$ denote an ideal CCT of width $k\ge  3$ in convex position, and let $\CT'$ be an elementary extension of $\CT$. Then $\CT'$ is in convex position as well.
\end{theorem}
The idea for the proof is to derive from the convex position of the $3$-complex $\RR(\CT, [k-3,k])\subset \CT$ the convex position of $\RR(\CT',[k-2,k+1])$, and then to apply the Alexandrov--van Heijenoort Theorem for polytopal manifolds with boundary, as given in Section~\ref{sec:convps}. 
\item \emph{Section~\ref{ssc:pfmthm1}:} The two previous steps provide an infinite family $\PS[n]$ of ideal CCTs in convex position in $S^4$ extending $\PS[1]$. We define the realization space of a complex, set $\mathrm{CCTP}_4[n]:=\conv (\PS[n])$, and get
\[
\dim \cR(\mathrm{CCTP}_4[n])\le  \dim \cR(\PS[n])\le  \dim \cR(\PS[1]).
\] 
A closer inspection gives the desired bound.
\end{compactitem}

\subsection{An explicit ideal cross-bedding cubical torus in convex position}\label{ssc:example}

The purpose of this section is to provide an ideal $3$-CCT $\PS[3]$ in convex position in the upper hemisphere $S^4_+$ of~$S^4$. We use homogeneous coordinates, 
so we describe points in $S^4_+$ by coordinates in $\R^4\times\{1\}$. 
The coordinates of the specific example we present are chosen in such a way that the complex can be reused later for the proof of Theorem~\ref{mthm:projun}.

\smallskip
The complex $\PS[3]$ is constructed by first giving a $1$-CCT $\PS[1]$ in~$S^4$, and then extending it twice. Instead of developing machinery that could provide the first two extensions, we will describe them directly in terms of vertex coordinates, and indicate how to verify that $\PS[3]$ is ideal and in convex position.

\medskip

We start off with a $1$-CCT $\PS[1]$ in~$S^4\subset\R^5$. Set
\[\vartheta_0=\big(\sqrt{2}-1,\,1-\sqrt{2},\,2,\,0,\, 1\big)\ \ \text{and}\ \ \vartheta_1:=\big(1,\,0,\,1,\,0,\, 1\big).\]
Let $L_0$ denote the orbit of $\vartheta_0$, and $L_1$ the orbit of $\vartheta_1$, under the group $\mathfrak{R}\subset O(\R^5)$ generated by $\rot_{3,4}^2$ and $\rot_{3,4} \rot_{1,2}$. The point configuration $L_0\cup L_1$ forms the vertex set of $\PS[1]$. 

The edges of $\PS[1]$ are given by
\begin{compactitem}[$\circ$]
\item the edge connecting $\vartheta_0$ with the vertex $\vartheta_1$, and its orbit under $\mathfrak{R}$,  
\item the edge connecting $\vartheta_0$ with the vertex $\vartheta_1^+=\rot_{1,2}^{-1}\rot_{3,4}^{-1}\vartheta_1$, and its orbit under $\mathfrak{R}$, and
\item the edge connecting $\vartheta_0$ with the vertex $\vartheta_1^-=\rot_{1,2}^{-1}\rot_{3,4}\vartheta_1$, and its orbit under $\mathfrak{R}$.
\end{compactitem}

\noindent $L_0$ is layer $0$ of $\PS[1]$, and consequently, the orbit $L_1$ corresponds to layer $1$ of $\PS[1]$. Let $\PS[2]$ denote the elementary extension of $\PS[1]$ (cf.\ Figure~\ref{fig:0transmit}), and let $\PS[3]$ denote the elementary extension of $\PS[2]$.

\begin{figure}[htbf]
\centering 
  \includegraphics[width=0.38\linewidth]{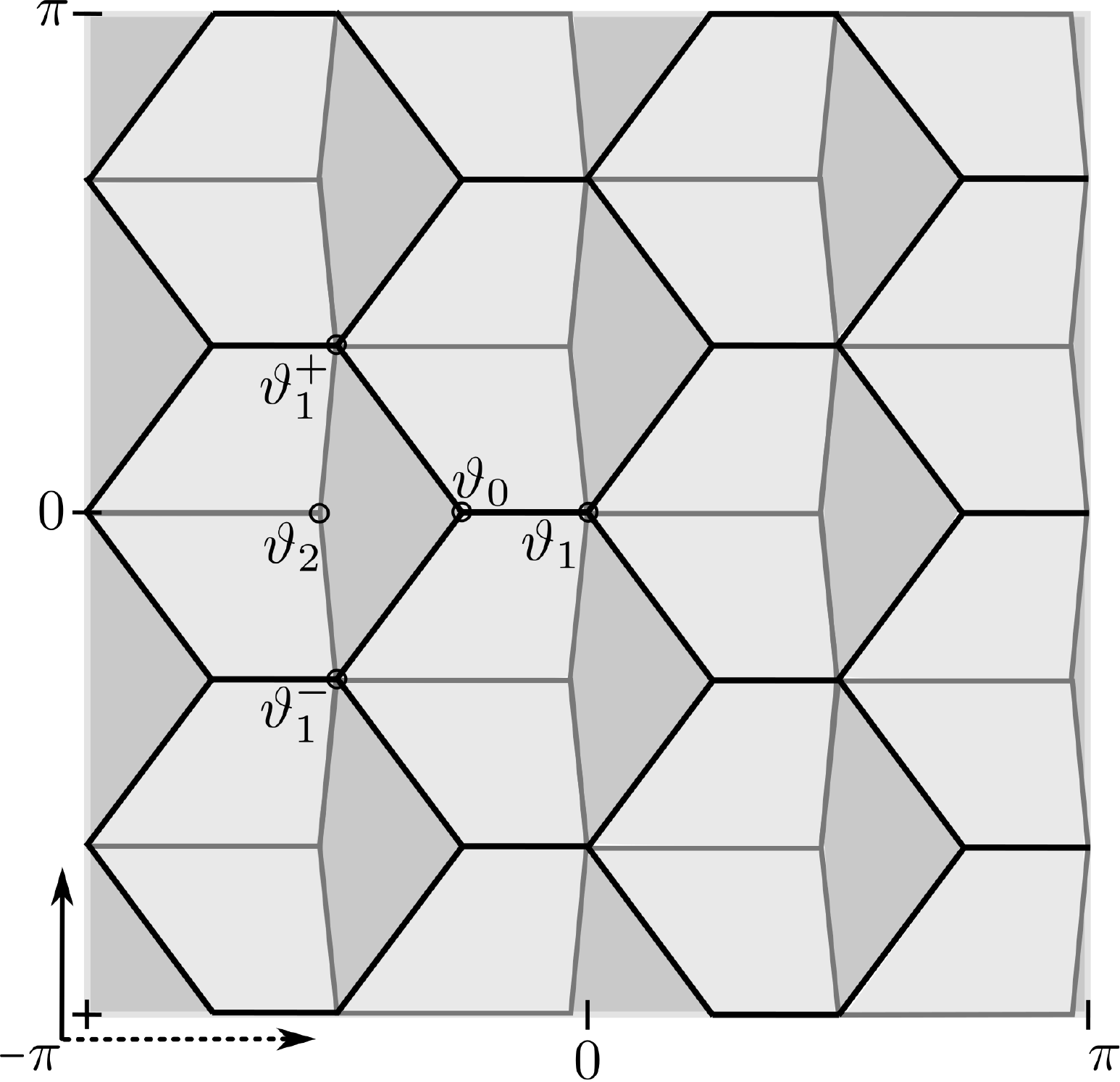} 
  \caption{\small The $2$-CCT $\PS[2]$, obtained as the elementary extension of $\PS[1]$.
     This image in $\mathcal{C}_1$ is produced by orthogonal projection to the equator $S^3_\eq$, followed by the Clifford projection~$\pi_1$.} 
  \label{fig:0transmit}
\end{figure}

We record some properties of the complexes $\PS[2]$ and $\PS[3]$ thus obtained:
\begin{compactitem}[$\circ$]
\item $\PS[2]$ is a CCT of width $2$. The coordinates of layer $2$ are obtained as \[ \vartheta_2= \big(\tfrac{1}{23}(-11+7\sqrt{2}),\, \tfrac{1}{23}(-9-11\sqrt{2}),\, \tfrac{1}{23}(16-6\sqrt{2}),\, 0,\, 1 \big)\]
and its orbit under the group of rotational symmetries $\mathfrak{R}$; see also Figure~\ref{fig:0transmit}.
\item $\PS[3]$ is a CCT of width $3$. The coordinates of layer $3$ are obtained as \[\vartheta_3=\big(\tfrac{1}{49}(37+11\sqrt{2}),\,\tfrac{1}{49}(-11+6\sqrt{2}),\tfrac{1}{49}(22-12\sqrt{2}),0,\,1\big)\]
and its orbit under $\mathfrak{R}$. Here $\vartheta_3$ is the unique vertex of $\PS[3]$ that lies in a joint facet with $\vartheta_0$.

\item $\PS[3]$ is ideal. The symmetry of $\PS[3]$ is obvious from the construction.
Transversality can be checked by straightforward computation, or by using the injectivity criterion Proposition~\ref{prp:inj3}: Using it gives directly that $\pi_1$ is injective on the star of every vertex $v$ of degree $3$ in $\RR(\PS[3],[0,2])$, and similarly $\RR(\PS[3],[1,3])$. In particular, $\pi_1$ is locally injective on the tori $\RR(\PS[3],[0,2])$ and $\RR(\PS[3],[1,3])$. It now follows from an examination of the action of the rotations that $\pi_1$ must be a trivial covering map, in particular injective. The computation of the slope and checking the orientation is again a simple calculation. 

\item $\PS[3]$ is in convex position. Outer normals of its facets are given by
\[\vv{n}=\big(7+5\sqrt{2},\,-8-5\sqrt{2},2,0, -9-5\sqrt{2}\big)\]
and its orbit under $\mathfrak{R}$. Here, $\vv{n}$ is an outer normal to the facet containing the vertex $\vartheta_0$. While verifying that $\PS[3]$ is in convex position this way is again easy, one can simplify the calculation drastically by using Proposition~\ref{prp:loccrt1lay} below.
\end{compactitem}

\subsection{Existence of the extension}\label{ssc:extension}
This section is devoted to the proof of Theorem~\ref{thm:exts}. 

\smallskip

We divide it into two parts, first proving that the elementary extension exists ``locally'', i.e.\ proving that, for every vertex $v$ of $\RR(\CT,k-2)$, there exists a $3$-cube containing $\St(v, \RR(\CT,[k-2,k]))$ as a subcomplex (Proposition~\ref{prp:locatt}), then concluding that the extension exists and is in fact an ideal CCT. For this section, we work in the equator $3$-sphere~$S^3_\eq$, and stay in the notation of Theorem~\ref{thm:exts}.

\subsubsection*{Local extension}

The goal of this section is to prove Proposition~\ref{prp:locatt}:

\begin{prp}\label{prp:locatt}
Let $\CT^\circ:=\RR(\CT,[k-2,k])$ denote the subcomplex of the ideal CCT $\CT$ induced by the vertices of the last three layers, and let $v$ be any vertex of layer $k-2$ in $\CT$. Then there exists a $3$-cube $X(v)\subset S^3_\eq$ such that $\St(v,\CT^\circ)$ is a subcomplex of $X(v)$. 
\end{prp}

We start with a lemma, the proof of which is postponed to the appendix, Section~\ref{ssc:lemdihang}:

\begin{lemma}\label{lem:dihang}
Let $v$, $\CT^\circ$ be chosen as in Proposition~\ref{prp:locatt}. Then $\St(v,\CT^\circ)$ is in convex position, and the tangent vector of $[v,\pi_0(v)]$ at $v$, seen as an element of $\NO_v^1 S^3_\eq$, lies in $\conv \Lk(v,\CT^\circ)$.
\end{lemma}

Intuitively speaking, the convex position of $\St(v,\CT^\circ)$ is a necessary condition for the existence of the cube $X(v)$, and the
conclusion that the tangent direction of $[v,\pi_0(v)]$ lies in $\conv \Lk(v,\CT^\circ)$ ensures that the new cube $X(v)$ is attached in direction of $\mathcal{C}_0$, and away from the complex $\CT$ already present. This allows us to provide the following technical statement towards the proof of Proposition~\ref{prp:locatt}. We consider, for $v$ and $\CT^\circ$ as above, the complex $\St(v,\CT^\circ)$ with vertices labeled as in Figure~\ref{fig:cubeatt}(1). Let $m$ denote the midpoint of the segment $[s,t]$.

\begin{cor}\label{cor:diffcomp}
$\pi_0(m)=\pi_0(s)$ and $v$ lie in different components of $S^3_\eq{\setminus} \SSp\{p,\, s,\, t\}$.
\end{cor}

\begin{proof}
Let $u'$ be the unique point in the weighted Clifford torus $\mathcal{C}_0$ so that $u$ lies in  $[v,u']$, and analogously let $p'$ denote the point of $\mathcal{C}_0$ for which $p\in [v,p']$; cf.\ Figure~\ref{fig:cubeatt}(2). Let $\Delta$ denote the triangle $\conv\{v,\,u',\,p'\}\subset \pi_2^{\SSp}(v)$. Since the tangent direction of $[v,\pi_0(v)]$ at $v$ lies in $\conv \Lk(v,\CT^\circ)$ by Lemma~\ref{lem:dihang}, we have that $\pi_0(v)$ lies in $ [u',p']$. Thus for all points $y$ in $\Delta$, the point $\pi_0(y)$ lies in $[u',p']$, and in particular $[\pi_0(u),\pi_0(p)]$ is a subset of $[u',p']$. Moreover, since $m$ lies in the interior of $\Delta$, $\pi_0(m)$ lies in the interior of $[u',p']$. Let us consider the polygon $\Theta:=\conv\{m,\, \pi_0(m),\,p',\,p\}\subset  \Delta$. We have the following:

\begin{compactitem}[$\circ$]
\item By Proposition~\ref{prp:alignsymm}(e), $\pi_0(m)$ is a point in the relative interior of $[\pi_0(u),\pi_0(p)]\subset [u',p']$. In particular, $p$ and $p'$ lie in the same component of $\pi_2^{\SSp}(v){\setminus}\SSp\{m,\, \pi_0(m)\}$. Thus the segment $[m,\pi_0(m)]$ is an edge of $\Theta$, since $m$ does not coincide with $\pi_0(m)$ by Proposition~\ref{prp:alignsymm}(a).
\item Since $m$ and $p$ lie in $\pi_2^{\mathrm{f}}(v)$ by Proposition~\ref{prp:alignsymm}(b) and (d), the segment $[p',\pi_0(m)]\in \mathcal{C}_0$ is exposed by the subspace $\mathcal{C}_0$ in $\pi_2^{\SSp}(v)$. Consequently, since $p'$ does not coincide with $\pi_0(m)$, we obtain that $[p',\pi_0(m)]$ is an edge of $\Theta$.
\item By an analogous argument, $[p,p']$ is an edge of $\Theta$: The point $m$ lies in the interior of $\Delta$, and $\pi_0(m)$ lies in the interior of $[u',p']$. Consequently, $m$ and $\pi_0(m)$ lie in the same component of $\pi_2^{\SSp}(v){\setminus}\SSp\{p,\, p'\}$. Thus the segment $[p,p']$ is an edge of $\Theta$, since we have $p\neq p'$ by Proposition~\ref{prp:alignsymm}(a).
\end{compactitem}

\noindent Thus $\Theta$ is a convex quadrilateral, the remaining edge of which is given by $[m,p]$, cf.\ Figure~\ref{fig:cubeatt}. In particular $\pi_0(m)$ and $p'$ lie in the same component of $\pi_2^{\SSp}(v){\setminus} \SSp\{p,\, m\}=\pi_2^{\SSp}(v){\setminus} \SSp\{p,\, s,\, t\}$, and since $p'$ and $v$ lie in different components of $\pi_2^{\SSp}(v){\setminus} \SSp\{p,\, m\}$, we obtain that $\pi_0(m)$ and $v$ lie in different components of $\pi_2^{\SSp}(v){\setminus} \SSp\{p,\, m\}$, as desired.
\end{proof}

As announced, we conclude this section with a proof of Proposition~\ref{prp:locatt}.

\begin{proof}[\textbf{Proof of Proposition~\ref{prp:locatt}}]

We continue to use the labeling of Figure~\ref{fig:cubeatt}. We already proved that the vertices of $\St(v,\CT^\circ)$ are not all coplanar (Lemma~\ref{lem:dihang}), so in order to prove $X(v)$ exists, it remains to be proven that the vertices $r,\, q,\, u$ lie in the same component of $S^3_\eq{\setminus} \SSp\{p,\, s,\, t\}$ as $v$. For $r$ and $q$, this follows directly from reflective symmetry of $\St(v,\CT^\circ)$ at the hyperplane $\pi^{\SSp}_2(v)$, so we only have to verify the property for $u$. Let $m$ denote the midpoint of the segment $[s,t]$, as before. Denote the component of $S^3_\eq{\setminus} \SSp\{p,\, s,\, t\}$ containing $v$ by $H^v$. Consider now the following segments:  
\begin{compactitem}[$\circ$]
\item the segment $[m,\pi_0(m)]$ from $m$ to $\pi_0(m)$,
\item the segment $[m,u]$ from $m$ to the vertex $u$, and
\item the segment $[m,p]$ from $m$ to the vertex $p$.  
\end{compactitem}
The slope $\alpha(\CT)$ of $\CT$ coincides with the angle between $[m,u]$ and $[m,\pi_0(m)]$.
The angle between $[m,p]$ and $[m,\pi_0(m)]$ is denoted by $\gamma$.
  
\begin{figure}[htbf]
\centering 
  \includegraphics[width=0.78\linewidth]{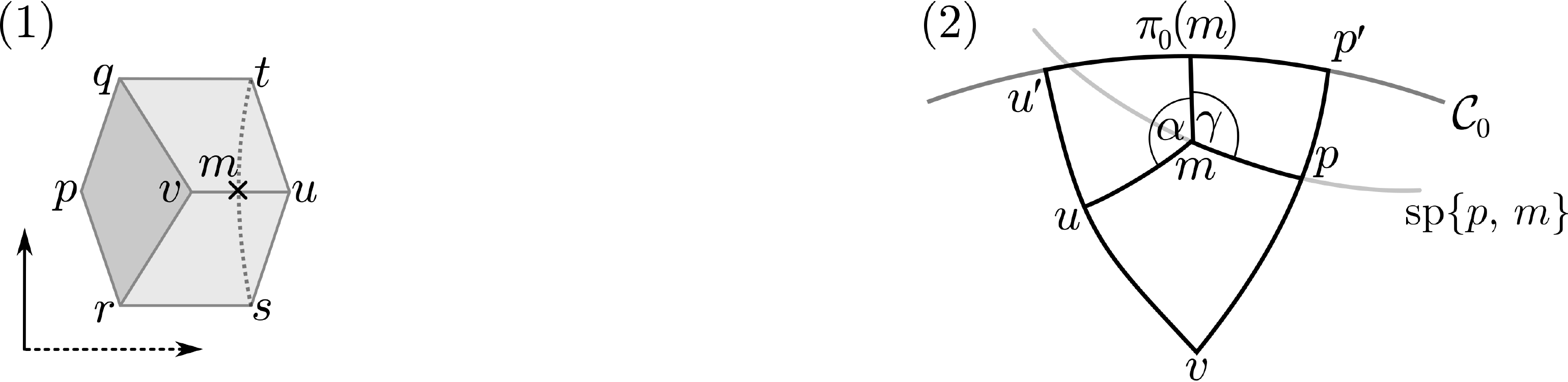} 
  \caption{\small  Illustrations for the proof of Proposition~\ref{prp:locatt}. \newline (1) The labeling of the vertices of $\St(v,\CT^\circ)$. The complex $\St(v,\CT^\circ)$ is the subcomplex of the boundary of a $3$-cube $X(v)$ if and only if the vertices $r,\, q,\, u$ lie on the same side of the hyperplane spanned by $p,\, s,\, t$ as $v$. \newline (2) The triangle $\Delta\in \pi_2^{\SSp}(v)$ on the vertices $u'$, $v$ and $p'$. By Corollary~\ref{cor:diffcomp}, $\pi_0(m)$ and $v$ are in different components of $S^3_\eq{\setminus} \SSp\{p,\, s,\, t\}$, and by Proposition~\ref{prp:alignsymm}, $u$ and $p$ lie in different components of $S^3_\eq{\setminus} \pi_2^{\SSp}(m)$. Thus $v$ and $u$ lie in the same components of $S^3_\eq{\setminus} \SSp\{p,\, s,\, t\}$ iff $\alpha(\CT)> \pi-\gamma$.} 
  \label{fig:cubeatt}
\end{figure}

As $\CT^\circ$ is transversal, $u$ and $p$ lie in different components of $S^3_\eq{\setminus} \pi_2^{\SSp}(m)=S^3_\eq{\setminus} \pi_2^{\SSp}(s)$ by Proposition~\ref{prp:alignsymm}(e). Furthermore, $\pi_0(m)$ and $v$ lie in different components of $S^3_\eq{\setminus} \SSp\{p,\, s,\, t\}$ by the previous corollary, so $u$ lies in $H^v$ if and only if $\alpha(\CT)>\pi-\gamma$. Since $\CT^\circ$ is ideal we have $\alpha(\CT)>\nicefrac{\pi}{2}$. Thus let us determine $\gamma$. After possibly applying a rotation of $\Sp\{e_1,\,e_2\}$-plane and $\Sp\{e_3,\,e_4\}$-plane, we may assume that the coordinates of $p$ are given as 
\[
	\big(p_1,\,0,\,p_3,\,0,\, 0\big),\ p_1,\, p_3 > 0,\ p_1^2+p_3^2=1.
\] 
Then 
\[
	\sqrt{(1-\tfrac{3}{4}p_3^2 )}m=\big(0,\,p_1,\,\tfrac{1}{2} p_3,\,0,\, 0\big),\ \ \pi_0(m)=\big(0,\,1,\,0,\,0,\, 0\big)
\] 
and  
\[
\cos(\gamma)=\frac{-p_3}{\sqrt{1+p_3^2}}
\]   
and, consequently, $\cos(\gamma)<0$ and $\gamma>\nicefrac{\pi}{2}$. Thus 
\[
\alpha(\CT)>\tfrac{\pi}{2}> \pi-\gamma. \qedhere
\] 
\end{proof}

\subsubsection*{The global extension exists and is ideal}

In this section, we prove Theorem~\ref{thm:exts}.

\smallskip

We need to prove that if we attach all cubes of the extension, we obtain a polytopal complex, and that the resulting complex is an ideal CCT. For this, we first prove that the attachment of a cube $X(v)$, where $v$ denotes a vertex of $\RR(\CT,k-2)$ as before, does not change the image of $\St(v,\CT^\circ),\, \CT^\circ=\RR(\CT,[k-2,k])$, under $\pi_1$ (Lemma~\ref{lem:localem}), which allows us to prove both the existence and the transversality of the extension. Theorem~\ref{thm:exts} then follows easily.

\begin{lemma}\label{lem:localem}
Let $X(v)$ denote the facet attached in Proposition~\ref{prp:locatt}, let $x(v)$ be the vertex of $X(v)$ not in $\CT$, and let $\widehat{X}(v)$ denote
the complex $\St(x(v),\parti X(v))$. Then $\pi_1$ is injective on $\widehat{X}(v)$.
\end{lemma}

The set-up for Lemma~\ref{lem:localem} is illustrated in Figure~\ref{fig:inj}.

\begin{figure}[htbf]
\centering 
  \includegraphics[width=0.76\linewidth]{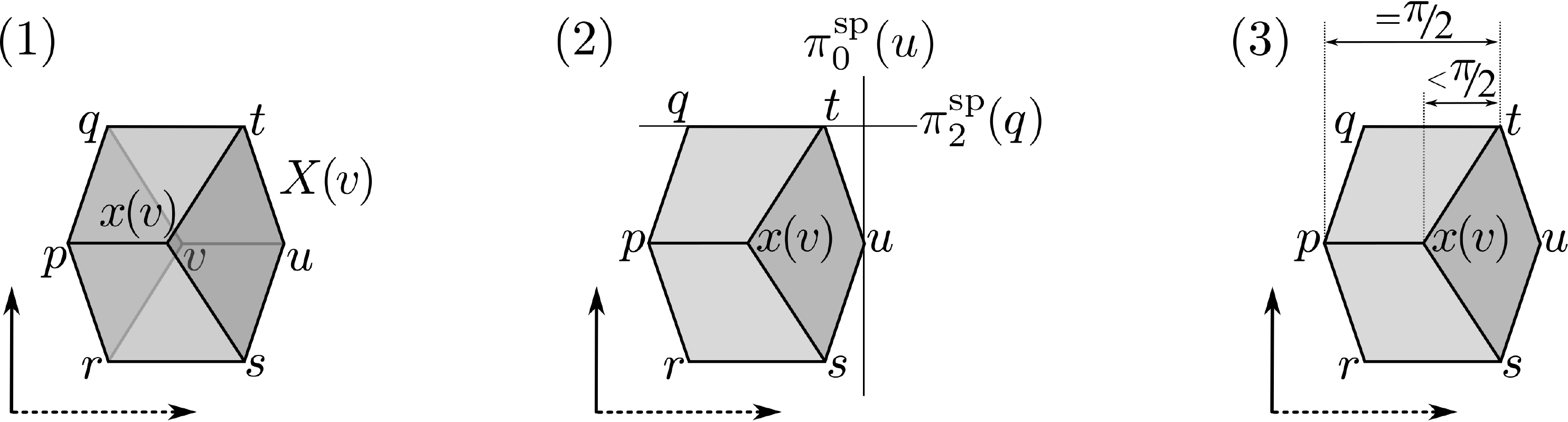} 
  \caption{\small  Illustration for Lemma~\ref{lem:localem}:  \newline 
(1) The set-up: We proved that there exists a $3$-cube $X(v)$ with $\St(v,\CT^\circ)$ as a subcomplex, and wish to prove that $\pi_1$ is injective on $\widehat{X}(v)=\St(x(v),\parti X(v))$. \newline 
(2) and (3) We collect the information needed to apply Proposition~\ref{prp:inj3}.} 
  \label{fig:inj}
\end{figure}

\begin{proof}
We wish to apply Proposition~\ref{prp:inj3} to prove that $\pi_1$ is injective on $\widehat{X}(v)$. All information on vertices not involving $x(v)$ needed for Proposition~\ref{prp:inj3} can be inferred from applying Proposition~\ref{prp:alignsymm} to $\parti\widehat{X}(v)=\parti \St(v,\CT^\circ)$. It remains to be proven that 
\[\widehat{X}(v)\cap (\mathcal{C}_0\cup \mathcal{C}_2)=\emptyset,\ \pi_2(x(v))=\pi_2(u),\ \text{and}\ \pi_0(x(v))\in \rint[\pi_0(u),\pi_0(p)]\] to satisfy the requirements of Proposition~\ref{prp:inj3}. For the proof of these statements, we use the convention that for a $2$-dimensional subspace $H$ in~$S^3_\eq$ and a point $y$ not contained in $H$, $H^y$ denotes the component of the complement of $H$ in~$S^3_\eq$ that contains $y$.

Since $t\in\pi_2^{\mathrm{f}}(q)$ by Proposition~\ref{prp:alignsymm}(b) and $p\in \pi_2^{\SSp}(q)^v$ by Proposition~\ref{prp:alignsymm}(d), the remaining vertex $x(v)$ of the quadrilateral $\conv \{x(v),\,r,\,s,\,p\}$ must lie in $\pi_2^{\SSp}(r)^v$. Analogously, $x(v)$ lies in $\pi_2^{\SSp}(q)^v$. Furthermore, by Proposition~\ref{prp:alignsymm}(e), $t,s \in\pi_0^{\SSp}(u)^v$, so the remaining vertex $x(v)$ of the quadrilateral $\conv \{x(v),\,u,\,s,\,t\}$ is contained in $\pi_0^{\SSp}(u)^v$ as well. This suffices to complete the proof of the first two statements.

\begin{compactitem}[$\circ$]
\item Since all vertices of $X(v)$ except $u$, $s$ and $q$ lie in $\pi_2^{\SSp}(q)^v\cap\pi_0^{\SSp}(u)^v$, the $3$-cube $X(v)$ can only intersect $\mathcal{C}_0\cup \mathcal{C}_2$ if either $u$ lies in $\mathcal{C}_0\cup \mathcal{C}_2$, or $[s,q]$ intersects $\mathcal{C}_0\cup \mathcal{C}_2$. This is excluded by Proposition~\ref{prp:alignsymm}(a) since $u,\, [s,q]\in \St(v,\CT^\circ)\subset \CT^\circ$. Thus $X(v)\cap (\mathcal{C}_0\cup \mathcal{C}_2)=\emptyset$, in particular, $\widehat{X}(v)\cap (\mathcal{C}_0\cup \mathcal{C}_2)=\emptyset$.
\item By reflective symmetry at $\pi_2^{\SSp}(u)$, $x(v)$ lies in $\pi_2^{\SSp}(u)$, and we can conclude that $x(v)\in \pi^{\mathrm{f}}_2(u)=\pi_2^{\SSp}(u)\cap \pi_2^{\SSp}(q)^v$.
\end{compactitem}

\noindent To prove the last statement, i.e.\ that $\pi_0(x(v))\in \rint[\pi_0(u),\pi_0(p)]$, consider the midpoint $m$ of the segment $[s,t]$, and the triangle $\conv\{\pi_0(x(v)),\,\pi_0(m),\,m\}$ in $\pi^{\SSp}_2(v)$. The angle at the vertex $\pi_0(m)$ is a right angle, the angle $\delta$ at vertex $m$ is bounded above by $\pi-\alpha(\CT)<\nicefrac{\pi}{2}$ and the angle at $\pi_0(x(v))$ shall be labeled~$\zeta$. The second spherical law of cosines implies that the distance $d$ between $\pi_0(m)$ and $\pi_0(x(v))$ satisfies
\[\cos d\sin\zeta=\cos\delta\ge \cos\big(\pi-\alpha(\CT)\big).\] 
Since $\delta\le \pi-\alpha(\CT)<\nicefrac{\pi}{2}$, we thus obtain \[d\le \delta\le \pi-\alpha(\CT)<\nicefrac{\pi}{2}.\]
By Proposition~\ref{prp:dict}, we have that $\pi_0(x(v))$ lies in $\rint[\pi_0(u),\pi_0(p)]$ if $\pi_0(x(v))$ lies in $\pi_0^{\SSp}(u)^v$
and in $\pi_0^{\SSp}(p)^v$.
\begin{compactitem}[$\circ$]
\item $\pi_0(x(v))\in\pi_0^{\SSp}(u)^v$ was already proven explicitly.
\item The distance between $\pi_0(m)$ and $\pi_0(p)$ is $\nicefrac{\pi}{2}$, and the distance between $\pi_0(x(v))$ and $\pi_0(m)$ is smaller than $\nicefrac{\pi}{2}$, and we conclude that $x(v)\in\pi_0^{\SSp}(p)^m$. By Proposition~\ref{prp:alignsymm}(e), $\pi_0^{\SSp}(p)^m$ coincides with $\pi_0^{\SSp}(p)^v$.
\end{compactitem}
\noindent Thus all conditions of Proposition~\ref{prp:inj3} are met by $\widehat{X}(v)$, and its application gives the injectivity of $\pi_1$ on $\widehat{X}(v)$.
\end{proof}

Since $\widehat{X}(v)\cap(\mathcal{C}_0\cup \mathcal{C}_2)=\emptyset$, the map $\pi_1$ is well-defined and continuous on $\widehat{X}(v)$. In particular, since $\pi_1(x(v))\in\pi_1(\St(v,\CT^\circ))$ and $\parti\widehat{X}(v)=\parti \St(v,\CT^\circ)$, we have:

\begin{cor}\label{cor:localem}
$\pi_1$ embeds $\widehat{X}(v)$ into $\pi_1(\St(v,\CT^\circ))\subset \mathcal{C}_1$. In particular, $\pi_1(\widehat{X}(v))= \pi_1(\St(v,\CT^\circ))$.
\end{cor}

We can now prove that attaching the facets $X(v)$ provides a polytopal complex, i.e.\ that the elementary extension exists, and that it is ideal as well.

\begin{proof}[\textbf{Proof of Theorem~\ref{thm:exts}}]
By combining Corollary~\ref{cor:localem} and Lemma~\ref{lem:dihang}, we see that $X(v){\setminus} \St(v,\CT^\circ)$ lies in the set \[K(v):=\bigcup_{x\in \mathrm{int}(\St(v,\CT^\circ))} [x,\pi_0(x)]\]
for all vertices $v\in \RR(\CT,k-2)$. The sets $K(\cdot)$ are disjoint subsets of the component of $S^3_\eq{\setminus} \CT$ containing $\mathcal{C}_0$ for different elements of $\RR(\CT,k-2)$ since $\pi_1$ is injective on $\CT^\circ$. Consequently, $X(v)\cap \CT$ coincides with $\St(v,\CT^\circ)$, and the complex $\CT'$ obtained as the union
\[\CT':=\CT\cup\ \bigcup_{v\in \RR(\CT,k-2)} X(v)\]
is a polytopal complex, and consequently a CCT that extends $\CT$. To finish the proof of Theorem~\ref{thm:exts}, it remains to be proven that $\CT'$ is ideal.
Symmetry is an immediate consequence of Lemma~\ref{lem:uniext} and the symmetry of $\CT$. Corollary~\ref{cor:localem}  shows that $\RR(\CT',[k-1,k+1])$ is embedded by $\pi_1$ into $\mathcal{C}_1$, thus $\CT'$ is transversal.  The fact that layer $k+1$ of $\CT'$ intersects the component containing $\mathcal{C}_0$ is immediate as well, since the segment $[x(v),\pi_0(x(v))]$ (where $x(v)$ is any layer $(k+1)$-vertex, cf.\ Lemma~\ref{lem:localem}) does not intersect $\CT'$ in any other point than $x(v)$, by transversality. Thus, to show that $\CT'$ is ideal it only remains to be proven that 
\[
\alpha(\CT')=\alpha\big(\RR(\CT',[k-1,k+1])\big)>\nicefrac{\pi}{2}.
\] This, however, is immediate from Proposition~\ref{prp:slmono}. 
\end{proof}

\subsection{Convex position of the extension}\label{ssc:convex}

Now we are concerned with the proof of Theorem~\ref{thm:convp}. We will get it from the following proposition.

\begin{prp}\label{prp:convpext}
Let $\CT$ denote an ideal CCT in $S^4$ of width $k\ge  3$ in locally convex position, and let $\CT'$ denote the elementary extension of $\CT$. Then 
\begin{compactenum}[\bf(I)] 
\item \emph{[Convex position for the fattened boundary]} the subcomplex $\RR(\CT',[k-2,k+1])$ of $\CT'$ is in convex position, and
\item \emph{[Locally convex position]} for every vertex $v$ in $\CT'$, $\St(v,\CT')$ is in convex position.
\end{compactenum}
\end{prp}

\begin{proof}[\textbf{Proposition~\ref{prp:convpext} implies Theorem~\ref{thm:convp}}]
Let $\CT$ denote an ideal $k$-CCT in convex position in $S^4$, $k\ge  3$, and choose $\ell:= \max\{k+1, 5\}$. Let $\CT''$ denote the ideal $\ell$-CCT that is an extension of $\CT$. Then
\begin{compactitem}[$\circ$]
 \item $\CT''$ is a manifold with boundary,
 \item the elementary extension $\CT'$ of $\CT$ is a subcomplex of $\CT''$,
\item  $\RR(\CT'',[0,3])=\RR(\CT,[0,3])$ is in convex position by assumption,
\item $\RR(\CT'',[\ell-2,\ell+1])$ is in convex position by Proposition~\ref{prp:convpext}{\bf (I)}, and 
\item $\CT''$ is in locally convex position by Proposition~\ref{prp:convpext}{\bf(II)}.
\end{compactitem}
Thus, by the Alexandrov--van Heijenoort Theorem for manifolds with boundary (Theorem~\ref{thm:locglowib}), 
$\CT''$ is in convex position. Since $\CT'$ is a subcomplex of $\CT''$, the elementary extension $\CT'$ of $\CT$ is in convex position as well.
\end{proof}

The proof of Proposition~\ref{prp:convpext} will occupy us for the rest of this section. We work in $S^4\subset \R^5$, and start off with a tool to check the convex position of ideal CCTs of width $3$:

\begin{prp}[Convex position of ideal $3$-CCT]\label{prp:loccrt1lay}
Let $\CT$ be an ideal $3$-CCT in~$S^4$, such that for every facet $\sigma$ of $\CT$, there exists a closed hemisphere $H(\sigma)$ containing $\sigma$ in the boundary, and such that $H(\sigma)$ contains all remaining vertices of $\RR(\CT,1)$ connected to $\sigma$ via an edge of $\CT$ in the interior. Then each $H(\sigma)$ contains all vertices of $\F_0(\CT){\setminus} \F_0(\sigma)$ in the interior.
\end{prp}

This is a strong local-to-global statement for the convex position of ideal CCTs, since it reduces the decision whether an ideal CCT is or is not in convex position to the evaluation of a local criterion. The proof requires a detailed and somewhat lengthy discussion of several cases, and is therefore given in the appendix, 
Section~\ref{ssc:localtoglobal}. Here is an immediate corollary. 

\begin{cor}\label{Cor:localglobal2}
Let $\CT$ be an ideal $3$-CCT in~$S^4$. Then $\CT$ is in locally convex position if and only if it is in convex position.
\end{cor}

Finally, a simple observation:

\begin{obs}[Halfspace selection]\label{obs:righthalfspace}
Let $C$ denote a polytopal complex in a closed hemisphere of $S^d$ with center $x\in S^d$ that consists of only two $(d-1)$-dimensional facets $\sigma$, ${\tau}$ such that $\sigma$ and ${\tau}$ intersect in a $(d-2)$-face, and there exists a closed hemisphere $H(\sigma)$ exposing $\sigma$ in $C$. Then $C$ is in convex position. 

If, additionally, $H(\sigma)$ contains $x\in S^d$, and the hyperplane $\SSp (x\cup(\sigma\cap {\tau}))$ separates $\sigma$ and ${\tau}$ into different components, then the hemisphere $H({\tau})$ exposing ${\tau}$ in $C$ contains $x$ as well.
\end{obs}

\begin{proof}[\textbf{Proof of Proposition~\ref{prp:convpext}}]
Consider first the case of (II) where $v$ is a vertex of layer $k-2$. The complex $\St(v,\CT)$ is in convex position by assumption. Let us prove that $\St(v,\CT')$ is in convex position as well. For this, denote by $X=X(v)$ the facet added in the extension from $\CT$ to $\CT'$ at the vertex $v$ of $\CT$. Set $\CT^+:=\RR(\CT',[k-3,k+1])$, and $\CT^-:=\RR(\CT',[k-3,k])$.
\smallskip

\noindent\textbf{(1)} \emph{$\St(v,\CT'),\ v\in\RR(\CT',k-2)$ is in convex position if $\St(v,\CT^+)$ is in convex position.}
\smallskip

$\St(v,\CT')$ is in convex position iff $\Lk(v,\CT')$ is in convex position, and by Theorem~\ref{thm:locglowib}, $\Lk(v,\CT')$ is in convex position if and only if it is in locally convex position, i.e.\ if for every edge $e$ of $\CT'$ containing $v$, $\Lk(e,\CT')$ is in convex position. This is already known for all edges that are not edges of ${X}$ since $\CT$ is in convex position. Thus $\St(v,\CT')$ is in convex position if $\Lk(e,\CT')$ is in convex position for edges in ${X}\cap \CT$ containing $v$.

We wish to reduce this further by showing the following:  By Lemma~\ref{lem:convglue}, $\Lk(e,\CT')$ is in convex position (for any edge $e\in X\cap \CT$ containing $v$) if $\Lk(e,\CT^+)$ is in convex position. There are two cases to consider:
\begin{compactitem}[$\circ$]
 \item If $\CT$ is a $3$-CCT, $\Lk(e,\CT^+)=\Lk(e,\CT')$, so the convex position of $\Lk(e,\CT^+)$ is clearly equivalent to the convex position of $\Lk(e,\CT')$.
 \item If $\CT$ is a $k$-CCT, $k\ge  4$, $\Lk(e,\CT)$ is a $1$-ball (and so is $\Lk(e,{X})$) in convex position in the $2$-sphere $\NO^1_e S^4$. Consequently, it follows from
 Lemma~\ref{lem:convglue}, applied to the pair of complexes $\Lk(e,\CT)$ and $\Lk(e,{X})$, that if $\Lk(e,\CT^+)$ is in convex position, then so is  $\Lk(e,\CT')$.
\end{compactitem}

\noindent Thus $\St(v,\CT')$ is in convex position if $\Lk(e,\CT^+)$ is in convex position for all edges of ${X}\cap \CT$ containing $v$. Since the convex position of $\St(v,\CT^+)$ clearly implies the convex position of $\Lk(e,\CT^+)$ for every edge of ${X}\cap \CT$ containing $v$, we have the desired statement.
\smallskip

\noindent\textbf{(2)} \emph{$\St(v,\CT^+),\ v\in\RR(\CT',k-2)$ is in convex position.}
\smallskip

Let $\sigma$ be any facet of $\St(v,\CT^-)$, and let $H(\sigma)$ denote the hemisphere exposing $\sigma$ in $\St(v,\CT)$. Let us first prove that
\[\F_0({X}){\setminus}\F_0(\sigma)\subset \intx H(\sigma),\]
i.e.\ $H(\sigma)$ exposes $\sigma$ in $\sigma\cup {X}$. Since $\St(v,\CT)$ is in convex position, it suffices to prove 
\[
	\F_0({X}){\setminus}\F_0(\St(v,\CT))\subset \intx H(\sigma).
\]
Every vertex $\F_0({X}){\setminus}\F_0(\St(v,\CT))$ (there is only one) lies in the interior of the convex cone $\Gamma$ with apex $v\in H(\sigma)$ and spanned by the vectors $w-v$,
\[
	w\in \F_0({X}\cap \RR(\CT,[k-1,k]))\in H(\sigma).
\]
Thus to show $\F_0({X}){\setminus}\F_0(\St(v,\CT))\subset \intx H(\sigma)$, it suffices to prove that the interior of $\Gamma$ does lie in $\intx H(\sigma)$, or equivalently that there exists one $w\in \F_0({X}\cap \RR(\CT,[k-1,k]))$ such that $w\in \intx H(\sigma)$. This, however, follows from the locally convex position of $\CT$.

\begin{figure}[htbf]
\centering 
 \includegraphics[width=0.35\linewidth]{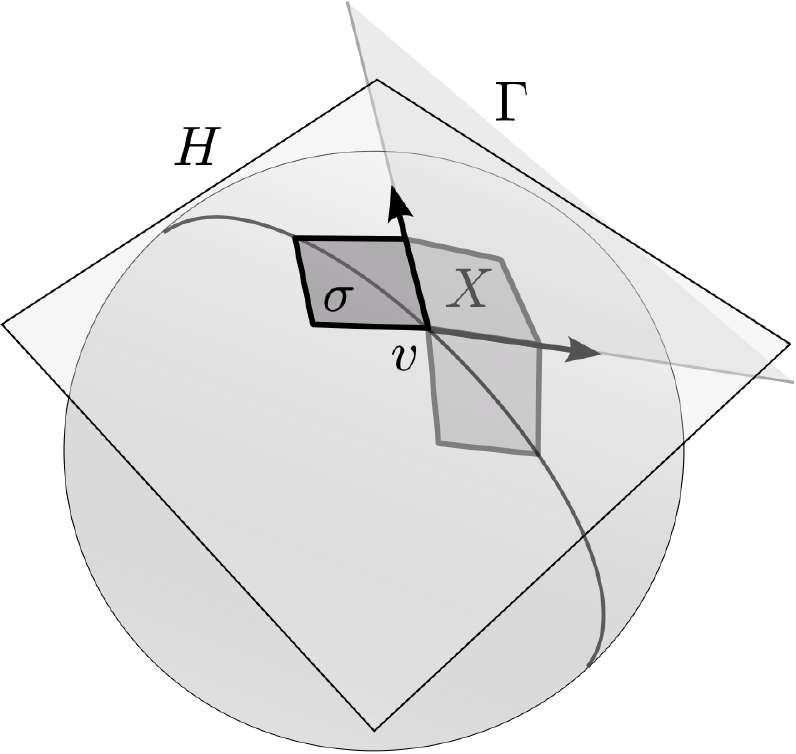} 
\caption{\small The facet ${X}$ added to the star of $v$ in the extension from $\CT$ to $\CT'$ lies $H(\sigma)$, since it is confined to a convex cone in $H(\sigma)$. The figure illustrates a $2$-dimensional analogue.} 
  \label{fig:localconv}
\end{figure}

It remains to find a hemisphere $H({X})$ exposing ${X}$ in $\St(v,\CT^+)$.  Every facet of $\St(v,\CT^-)$ intersects ${X}$ in a $2$-face, thus, by the preceding argument, for every facet $\sigma$ of $\St(v,\CT^-)$ there exists a hemisphere $H_\sigma({X})$ exposing ${X}$ in $\sigma\cup {X}$. It remains to be proven that $H_\sigma({X})$ does not depend on $\sigma$, i.e.\ it we have to prove that for every choice of facets $\sigma'$, $\sigma''$ in $\St(v,\CT^-)$ we have $H_{\sigma'}({X})=H_{\sigma''}({X}).$

For this, we can use Observation~\ref{obs:righthalfspace}: $\CT^-$ is in convex position, since it is an ideal $3$-CCT in locally convex position (Corollary~\ref{Cor:localglobal2}). We may assume, after possibly a reflection at $\Sp\{e_1,\, e_2,\, e_3,\, e_4\}\subset \R^5$, that $\CT$ is chosen in such a way that $\CT^-$ lies in the hemisphere of~$S^4$ with center $e_5$, so that by symmetry of $\CT$ every hemisphere exposing a facet of it contains the point $e_5$. Furthermore, the orthogonal projection $\pp$ of $S^4{\setminus} \{\pm e_5\}$ to~$S^3_\eq$ is injective on $\CT'$ since $\CT'$ is ideal, so if ${X}$, $\sigma$ are adjacent facets of $\CT'$, then $\SSp (e_5\cup(\sigma\cap {X}))$ separates them into different components.
By Observation~\ref{obs:righthalfspace}, the hemisphere $H_\sigma({X})$ that exposes ${X}$ in $\sigma\cup {X}$ is uniquely determined as the one containing $e_5$, independent of $\sigma$. In particular, $H({X})=H_\sigma({X})$ exposes ${X}$ in $\St(v,\CT^-)$, and consequently, $\St(v,\CT^+)$ is in convex position. 

As already observed, this proves that $\St(v,\CT')$ is in convex position. We can now complete the proofs of the claims \textbf{(I)} and \textbf{(II)} of Proposition~\ref{prp:convpext}.
\smallskip

\noindent\textbf{(I)}
If $H({X})$ is the hyperplane exposing ${X}$ in $\St(v,\CT')$ (where $v$ is the layer $(k-2)$-vertex of $\CT'$ in ${X}$, as above), then all vertices of $\St(v,\CT')$ that are not in ${X}$ lie in the interior of $H$. In particular, all vertices of layers $k-2$ of $\CT'$ being connected to ${X}$ via an edge lie in the interior of $H({X})$, since they lie in $\St(v,\CT')$. Thus $\RR(\CT',[k-2,k+1])$ lies in convex position by Proposition~\ref{prp:loccrt1lay}.
\smallskip

\noindent\textbf{(II)}
We have to prove that for every vertex $v$ of $\CT'$ the star $\St(v,\CT')$ is in convex position. For vertices $\RR(\CT',k-2)$, we already proved that $\St(v,\CT')$ is in convex position. For vertices in $\RR(\CT',[0,k-3])$, this follows from the assumptions on $\CT$. Thus it remains to consider the case in which $v$ is a vertex of layer $k-1$, $k$ or $k+1$. If $v$ is in layer $k+1$, this is trivial, since in this case the star consists of a single facet. If $v$ is in layer $k-1$ or $k$, this follows from Theorem~\ref{thm:locglowib}, applied to the complex $\Lk(v,\CT')$: 
Since $\RR(\CT',[k-2,k+1])$ is in convex position, so is $\Lk(v,\RR(\CT',[k-2,k+1]))$. Thus, by Theorem~\ref{thm:locglowib}, it remains to be proven that $\Lk(v,\CT')$ is in locally convex position, i.e., to show that $\Lk(e,\CT')$ is in convex position for every $e$ edge of $\CT'$ containing $v$. There are, again, three cases to consider:
\begin{compactitem}[$\circ$]
\item Edges $e$ between vertices of layer $k-2$ and $k-1$ (for which the convex position of $\Lk(e,\CT')$ follows from the convex position of the stars of layer $k-2$ vertices),
\item edges between vertices of layer $k-1$ and $k$ (for which the convex position of $\Lk(e,\CT')$ follows from statement~\textbf{(I)}), and
\item edges between vertices of layers $k$ and $k+1$ (which, again, is a trivial case). \qedhere
\end{compactitem}
\end{proof}

\subsection{Proof of Theorem~\ref{mthm:Lowdim}}\label{ssc:pfmthm1}

Let $C$ denote a polytopal complex in $\R^d$. A polytope $P$ in $\R^d$ induces a convex position realization of~$C$ if the boundary complex of $P$ contains a subcomplex combinatorially equivalent to $C$. We generalize the notion of the realization space of a polytope to polytopal complexes. The \Defn{(convex) realization space} of $C$ is defined as
\[ \cR(C):=
\big\{V\in\R^{d \times f_0(C)}: \conv(V)\, \text{induces a convex position realization of }\, C \big\},
\]
and just like to the realization space of a polytope, the realization space of $C$ is a semi-algebraic set in~$\R^{d \times f_0(P)}$ and its dimension is well-defined. 
We will use the following elementary lemma.

\begin{lemma}\label{lem:incl}
Let $C$ denote a polytopal complex in convex position. Then $ \cR(\conv C)\subseteq \cR(C)$. \emph{\qed}
\end{lemma}

Furthermore, we reformulate Lemma~\ref{lem:uniext} in a slightly stronger form:

\begin{lemma}[Unique Extension]\label{lem:uniextem}
Let $\CT$ denote a CCT of width at least $1$, and let $\CT'$ denote its extension. Then $\cR(\CT')$ embeds into $\cR(\CT)$.
\end{lemma}

\begin{proof}
Any pair of distinct $2$-faces in any convex position realization of $\CT$ are not coplanar. Consequently, by Lemma~\ref{lem:cubecmpl}, the inclusion of vertex sets induces an embedding of $\cR(\CT')$ into $\cR(\CT)$.
\end{proof}

\begin{proof}[\textbf{Proof of Theorem~\ref{mthm:Lowdim}}]
Let $\PS[n]$, $n\ge4$, denote the CCT of width $n$ extending the $3$-CCT $\PS[3]$, which exists by Theorem~\ref{thm:ext} and is in convex position by Theorem~\ref{thm:convp}. 
We define
\[\mathrm{CCTP}_4[n]:=\conv \PS[n].	\]
This is the family announced in Theorem~\ref{mthm:Lowdim}:
\begin{compactitem}[$\circ$]
\item As $\PS[n]$ is in convex position, it lies in some open hemisphere of $S^4$; thus, the sets $\mathrm{CCTP}_4[n]$ are polytopes.
\item The polytopes $\mathrm{CCTP}_4[n]$ are of dimension $4$ by construction.
\item The polytopes $\mathrm{CCTP}_4[n]$ are combinatorially distinct: \[f_0(\mathrm{CCTP}_4[n])=f_0(\PS[n])=12(n+1).\]
\item By Lemma~\ref{lem:incl}, $\cR\mathrm{CCTP}_4[n])$ embeds into $\cR(\PS[n])$, which in turn embeds into the realization space $\cR(\PS[1])$ by Lemma~\ref{lem:uniextem}. It remains to estimate the dimension of the last space:
\[\dim \cR\big(\mathrm{CCTP}_4[n])\big)\le  \dim\cR\big(\PS[n]\big) \le\dim\cR(\PS[1]) \le  4f_0(\PS[1])=96.\qedhere\]
\end{compactitem}
\end{proof}

\setcounter{figure}{0}
\section{Many projectively unique polytopes in fixed dimension}\label{sec:projun}

We will now prove Theorem~\ref{mthm:projun}, where we may restrict to the case $d=69$ as in the reasoning for Theorem \ref{mthm:Lowdim}.
For the proof, we might wish to start by finding a finite set $R$ of points outside the polytope $\mathrm{CCTP}_{4}[n]$ such that the $24$~vertices in the first two layers of $\mathrm{CCTP}_{4}[n]$, combined with $R$ form a projectively unique “polytope--point configuration” $(\mathrm{CCTP}_{4}[n],R)$ for each $n>1$. We could then apply Lawrence extensions to the points in~$R$
to obtain the desired family $\mathrm{PCCTP}_{69}[n]$ of projectively unique polytopes. 

We have to deal, however, with the fact that Lawrence equivalence recognizes only those affine dependences
that are external to the polytope in question. Thus in order to establish projective uniqueness
after the Lawrence extensions, we would need intricate details about the relation between the polytopes 
involved and the external point configuration, which are hard to obtain for the infinite family of cross-bedding 
cubical torus polytopes $\mathrm{CCTP}_{4}[n]$.
We solve this problem as follows:
\begin{compactenum}[(1)]
\item Going beyond the classical set-up of projectively unique polytope--point configurations,
we introduce \Defn{weak projective triples} (Definition~\ref{def:wpt}). Such a triple consists of a polytope $P$, 
a subset $Q$ of its vertex set, and a point configuration $R$.
\item The \Defn{subdirect cone} (Definition~\ref{def:subd}) is an operation that produces a projectively unique 
polytope--point configuration from any weak projective triple (Lemma~\ref{lem:subdc}).
\item In Lemma~\ref{lem:affine}, we provide a family of weak projective triples that consists of the polytopes $\mathrm{CCTP}_{4}[n]$, the subset $\F_0(\PS[1])$ of their vertex sets, and a finite point configuration $R$ in $\R^4$. The sets $\F_0(\PS[1])$ and $R$ do not depend on $n$, and $K:= \F_0(\PS[1])\cup R$ is of cardinality $64$. Using the subdirect cone defined in the previous step, we are then able to provide an infinite family of projectively unique PP configurations $(\mathrm{CCTP}_{4}[n]^v, K)$ in $S^5$.
\item From this we then obtain, using Lawrence extensions, the desired family of polytopes $\mathrm{PCCTP}_{69}[n]$ 
in~$S^{69}$.
\end{compactenum}

\noindent In the next section, we recall the basic facts and notions about projectively unique polytopes, point configurations and polytope--point configurations, and review the classic technique of Lawrence extensions. In Section~\ref{ssc:pfmthm2} we define weak projective triples, show how to obtain projectively unique PP configurations from them, and close with the proof of Theorem~\ref{mthm:projun}. A crucial part of this proof, Lemma~\ref{lem:affine}, is stated in Section~\ref{ssc:pfmthm2}, but the details of its verification are delayed to Section~\ref{ssc:constr}.

\subsection{Projectively unique polytopes and polytope--point configurations}\label{ssc:prun}
\enlargethispage{3mm}
We follow Gr\"unbaum~\cite{Grunbaum} for the notions of projectively unique polytopes and point configurations, and Richter-Gebert~\cite{RG} for Lawrence equivalence and Lawrence extensions. Recall that we work with spherical geometry instead of using the more common euclidean set-up.

A polytope $P$ in $S^d$ is \Defn{projectively unique} if the group $\mathrm{PGL}(\R^{d+1})$ of projective transformations on~$S^d$ acts transitively on~$\cR(P)$. In particular, for projectively unique polytopes we have $\dim \cR(P) \le \dim \mathrm{PGL}(\R^{d+1}) = d(d+2)$, with equality if the vertex set of the polytope contains a projective basis. 
\begin{definition}[PP configurations, Lawrence equivalence, projective uniqueness]
A \Defn{point configuration} is a finite collection $R$ of points in some open hemisphere of~$S^d$. 
If $H$ is an oriented hyperplane in $S^d$, then we use $H_+$ resp.\ $H_-$ to denote the open hemispheres 
bounded by $H$. If $P$ is a polytope in $S^d$ such that $P\cap R=\emptyset$ and $P\cup R$ is contained
in an open hemisphere of~$S^d$, then the pair $(P,R)$ is a 
\Defn{polytope--point configuration}, or \Defn{PP configuration}. A hyperplane $H$ is \Defn{external to $P$} 
if $H\cap P$ is a face of~$P$.

Two PP configurations $(P,R)$, $(P',R')$ in $S^d$ are \Defn{Lawrence equivalent} if there is a 
bijection $\varphi$ between the vertex sets of $P$ and $P'$ and the sets of $R$ and $R'$, 
such that, if $H$ is any hyperplane for which the closure of $H_-$ contains $P$, there exists an oriented hyperplane $H'$ for which the closure of $H'_-$ contains $P'$ and  
\[
    \varphi(\F_0(P)\cap H_-)=\F_0(P')\cap H'_-, \qquad 
    \varphi(R\cap H_+)={R'}\cap H'_+, \qquad 
    \varphi(R\cap H_-)={R'}\cap H'_-.
\]
A PP configuration $(P,R)$ in~$S^d$ is \Defn{projectively unique} if for any PP configuration $(P',R')$ in~$S^d$ 
Lawrence equivalent to it, and every bijection $\varphi$ that induces the Lawrence equivalence, there is a projective transformation $T$ that realizes $\varphi$.
\end{definition}

For example, $(P,\emptyset)$ and $(P',\emptyset)$ are Lawrence equivalent if and only if the polytopes $P$ and $P'$ are combinatorially equivalent. They are projectively equivalent as PP configurations if and only if $P$ and $P'$ are projectively equivalent as polytopes.

Two point configurations $R,\, R'$ are Lawrence equivalent if the PP configurations $(\emptyset,R)$ and $(\emptyset,R')$ are Lawrence equivalent. For example, the Perles Configuration~\cite[Thm.\ 5.5.3]{Grunbaum} is projectively unique. This example is of particular interest since it gives rise to an irrational polytope via Lawrence extensions.

\begin{prp}[Lawrence extensions, cf.\ {\cite[Lem.\ 3.3.3 and 3.3.5]{RG}}, {\cite[Thm.\ 5]{ZNonr}}] \label{prp:mlwextn}
Let $(P,R)$ be a projectively unique PP configuration in $S^d$. Then there exists a $(\dim P+f_0(R))$-dimensional polytope on $f_0 (P) + 2 f_0(R)$ vertices that is projectively unique and that contains $P$ as a face.
\end{prp}

\begin{proof}[Sketch of Proof]
Set $k=f_0(R)$, and denote the elements of $R$ by $r_i,\, 1\le  i\le  k$. To obtain the Lawrence extension of $(P,R)$, consider any embedding of $S^{d}$ (and with it $(P,R)$) into $S^{d+k} $.

\begin{figure}[htbf]
\centering 
  \includegraphics[width=0.53\linewidth]{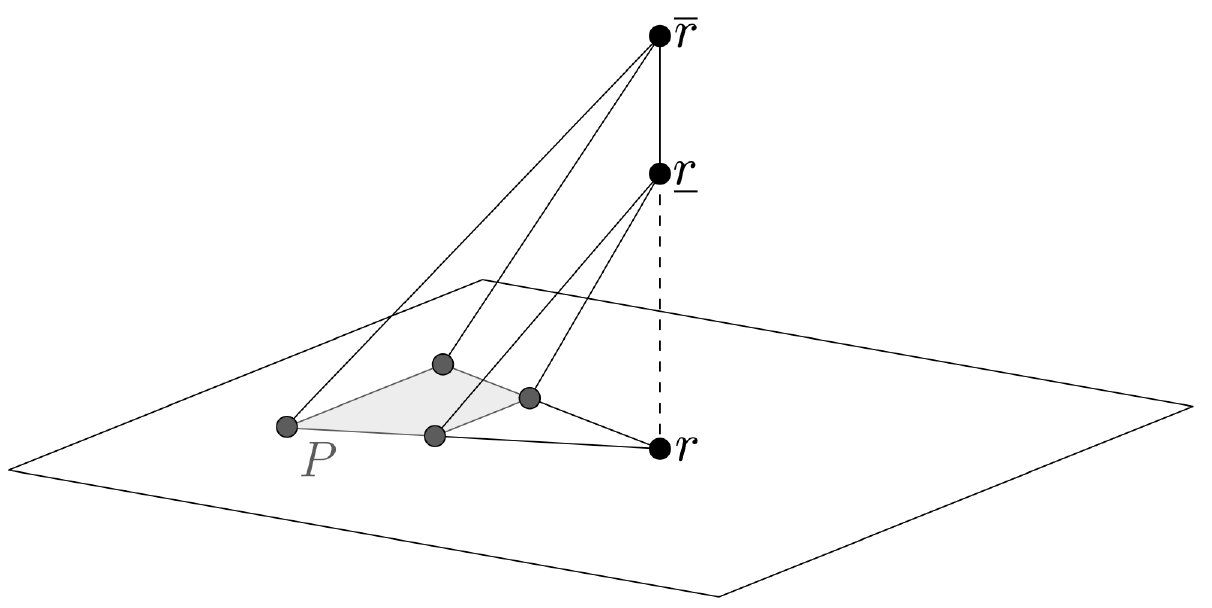} 
  \caption{\small  The Lawrence polytope associated to a polytope $P$ and a point $r\notin P$ is obtained by lifting $r$ to two new points $\underline{r}$ and $\overline{r}$.} 
  \label{fig:lawrence}
\end{figure}

\noindent The \Defn{Lawrence polytope} of a PP-configuration $(P,R)$ is defined as \[\mathrm{L}(P,R):=\conv \Big(P \cup \bigcup_{1\le i \le k} \{ \underline{r}_i,\, \overline{r}_i\} \Big),\]
where $\underline{r}$ and $\overline{r}$ are two points associated to every $r\in R$ such that $\underline{r}$ lies in the segment $[r\overline{r}]$, and such that for every $1\le  j\le  k$, the line $\SSp\{r,\, \overline{r}_j\}$ intersects the subspace 
\[\SSp\Big(P \cup \bigcup_{1\le i \le k} \{r_i,\, \overline{r}_i\} \Big)\]
transversally. This polytope has the desired number of vertices, dimension, and contains $P$ as a face. One can show that $\mathrm{L}(P,R)$ is projectively unique. Here, we discuss only the special case where $R$ is a single point $r$; the full case is analogous.

Consider a realization $\mathrm{L}'$ of $\mathrm{L}(P,R)$. Intersecting the ray from $\underline{r}'$ through $\overline{r}'$ with the affine span of the facet $P'$ gives a point $r'$, and the pair $(P',r')$ can be seen to be Lawrence equivalent to $(P,r)$~\cite[Lem.\ 3.3.5]{RG}. We construct a projective map from $\mathrm{L}$ to $\mathrm{L}'$ in three steps.
\begin{compactitem}[$\circ$]
\item Since $(P,r)$ is projectively unique, there exists a projective transformation $T$ of $S^{d}$ that maps~$(P,r)$ to~$(P',r')$.
\item The lines $\SSp\{T(\underline{r}),\, T(\overline{r})\}$ and $\SSp\{\underline{r}',\, \overline{r}'\}$ intersect in $r'$. Hence, we can use a projective transformation $U$ restricting to the identity on $\SSp P'$ to map  $\SSp\{T(\underline{r}),\, T(\overline{r})\}$ to $\SSp\{\underline{r}',\, \overline{r}'\}$, and $T(\underline{r})$ to $\underline{r}'$.
\item A projective transformation $V$ fixing $\underline{r}'$ and $\SSp P'$ can then be used to map $(UT)(\overline{r})$ to $\overline{r}'$. 
\end{compactitem}
\noindent The configuration of maps $V\circ U\circ T$ gives the desired projective transformation from $\mathrm{L}(P,R)$ to $\mathrm{L}'$.
\end{proof}

\subsection{Weak projective triples and the proof of Theorem~\ref{mthm:projun}}\label{ssc:pfmthm2}

In this section, we prove the second main theorem. Let us start by reformulating the idea of Lemma~\ref{lem:cubecmpl}.

\smallskip

\begin{definition}[Framed polytopes] Let $P$ denote a polytope in $S^d$, and let $Q$ be any subset of its vertex set, that is, \[Q=\{q_1,\, q_2,\, q_3,\, \dots\}\subseteq \F_0(P).\]
Let $P'$ be any polytope in $S^d$ combinatorially equivalent of $P$. Let $\varphi$ denote the labeled isomorphism from the faces of $P$ to the faces of $P'$. We say that the polytope $P$ \Defn{is framed by} the set of vertices $Q$ if $P=P'$ for all choices of $P'$ and $\varphi$ that satisfy $\varphi(q)=q$ for all $q\in Q$.\end{definition}

\begin{examples}\label{ex:stdet} We record some instances of framed polytopes, the last of which is important to the proof of Theorem~\ref{mthm:projun}.
\begin{compactenum}[\rm(a)]
\item If $P$ is any polytope, then $\F_0(P)$ frames $P$.
\item If $P$ is a projectively unique polytope, and $Q\subseteq \F_0(P)$ is a projective basis for its span, then $Q$ frames $P$. 
\item A $3$-cube $W$ is framed by any $7$ of its vertices, even by $6$ vertices with the property that no quadrilateral of $W$ contains all of them. This is the key idea of Lemma~\ref{lem:cubecmpl}.
\item By Lemma~\ref{lem:uniextem}, the vertex set $\F_0(\PS[1])\subset \F_0(\mathrm{CCTP}_4[n])$ frames $\mathrm{CCTP}_4[n]$ for all $n \ge  1$.
\end{compactenum}
\end{examples}

\begin{definition}[Weak projective triple in $S^d$]\label{def:wpt}
A triple $(P,Q,R)$ of a polytope $P$ in $S^d$, a subset $Q$ of $\F_0(P)$ and a point configuration $R$ in $S^d$ is a \Defn{weak projective triple} in $S^d$ if and only if

\begin{compactenum}[\rm(a)]
\item the three sets are contained in some open hemisphere in $S^d$,
\item $(\emptyset, Q \cup R)$ is a projectively unique point configuration,
\item $Q$ frames the polytope $P$, and
\item some subset of $R$ spans a hyperplane $H$, the \Defn{wedge hyperplane}, which does not intersect $P$.
\end{compactenum}
\end{definition}

\begin{figure}[htbf]
\centering 
  \includegraphics[width=0.27\linewidth]{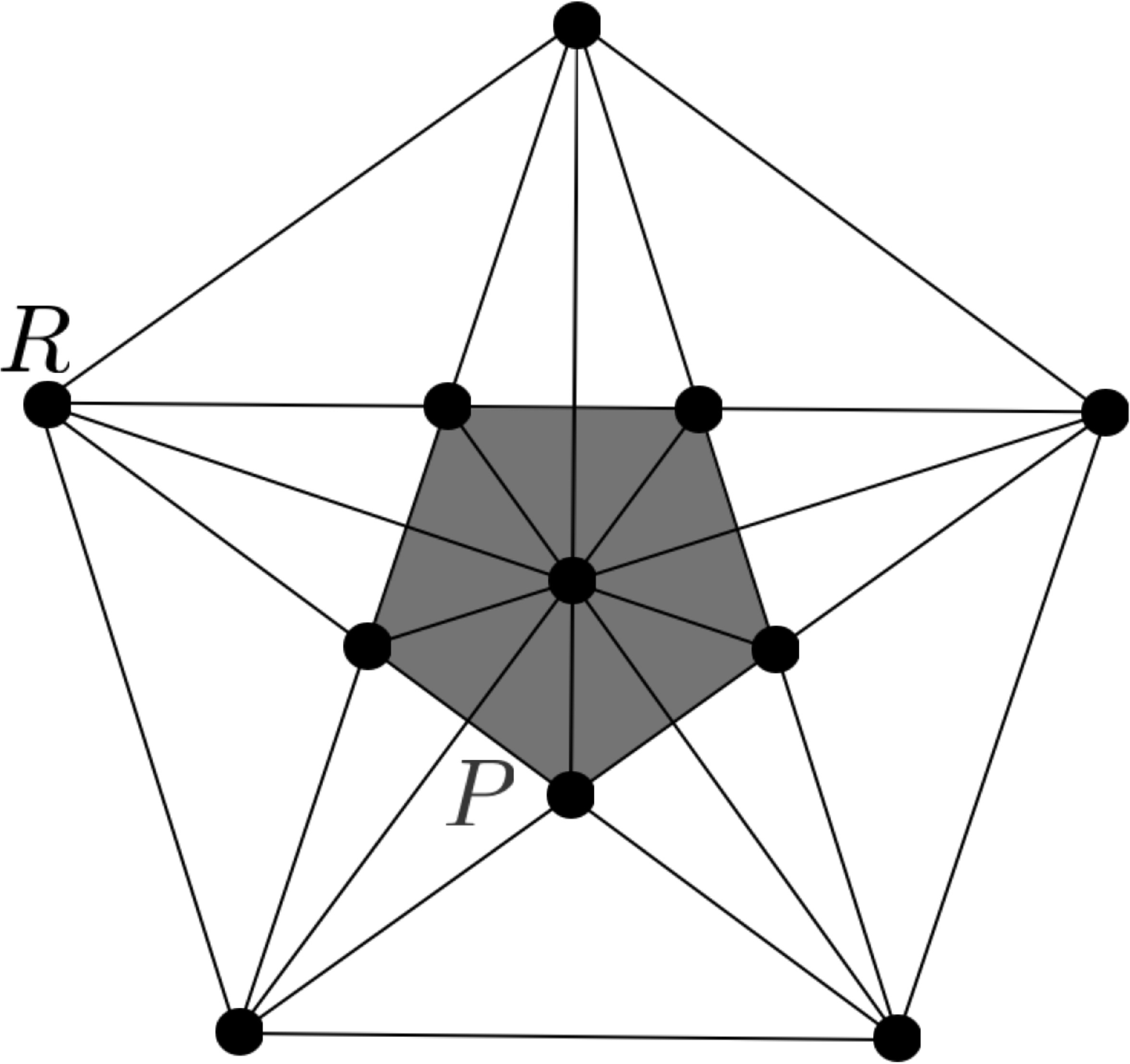} 
  \caption{\small A weak projective triple in the special case $Q=\F_0(P)$.} 
  \label{fig:pentg}
\end{figure}

\begin{definition}[Subdirect Cone]\label{def:subd}
Let $(P,Q,R)$ be a weak projective triple in $S^d$, and identify $S^d$ with the equator $S^d_\eq$ in $S^{d+1}$. Let $H$ denote the wedge hyperplane in $S^d_\eq$ spanned by vertices of $R$ with $H\cap P=\emptyset$. Let $v$ denote any point not in $S^d_\eq$, and let $\widehat{H}$ denote any hyperplane in $S^{d+1}$ such that $\widehat{H}\cap S^d_\eq=H$ and $\widehat{H}$ separates $v$ from $P$. For every vertex $p$ of $P$ consider the point $p^v=\conv\{v,\,p\}\cap \widehat{H}$. Denote by $P^v$ the pyramid
\[P^v:=\conv \Big( v\cup \bigcup_{p\in \F_0(P)} p^v \Big).\]
 The PP configuration $(P^v,Q \cup R)$ in $S^{d+1}$ is a \Defn{subdirect cone} of $(P,Q,R)$.
\end{definition}

Figure~\ref{fig:pentg} shows a weak projective triple, and Figure~\ref{fig:subdirect} shows the associated subdirect cone. 

\begin{figure}[htbf]
\centering 
  \includegraphics[width=0.68\linewidth]{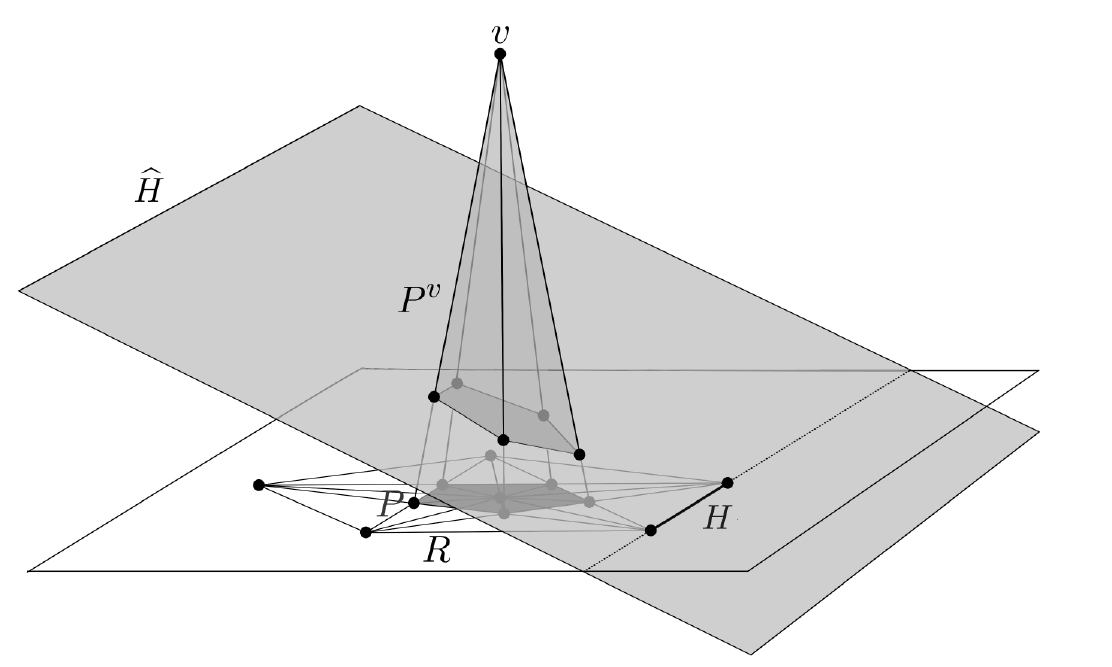} 
  \caption{\small The subdirect cone of the weak projective triple of Figure~\ref{fig:pentg}} 
  \label{fig:subdirect}
\end{figure}

\begin{lemma}\label{lem:subdc}
For any weak projective triple $(P,Q,R)$, the subdirect cone $(P^v,Q \cup R)$ is a projectively unique PP configuration.
\end{lemma}	

\begin{proof}
Define $Q^v$ as the set of all points $q^v=\conv\{v,\,q\}\cap \widehat{H}$, where $q$ ranges over all elements of $Q$. Then $Q^v$ frames the polytope $P^v\cap \widehat{H}$.

Consider any realization $(A^w,B\cup C)$ of $(P^v,Q \cup R)$. Let $H$ denote the wedge hyperplane of the triple $(P,Q,R)$, and let $I$ denote the hyperplane spanned by the points of $C$ corresponding to $H\cap R$ under the Lawrence equivalence. Let $\widehat{I}$ denote the hyperplane containing $I$ and the facet of $A^w$ not containing $w$. 

The point configuration $B\cup C$ is a realization of $Q \cup R$ since both sets lie in hyperplanes not intersecting the polytopes $A^w$ resp.\ $P^v$. 
Thus we may assume, after a projective transformation, that $B=Q$, $C=R$, $w=v$ and $\widehat{I}=\widehat{H}$. The vertices of $B^w\subset\F_0(A^w)=\F_0(A^v)$ are determined as the intersections of $\SSp\{v,\,b\},\, b\in B$
with $\widehat{I}=\widehat{H}$. In particular $B^w=B^v=Q^v$. As $Q^v$ frames the facet $P^v\cap \widehat{H}$, $A^w\cap \widehat{I}=P^v\cap \widehat{H}$, and consequently, $A^w=P^v$.\end{proof}

From here on, there are two possibilities to provide infinite families of projectively unique polytopes in fixed dimension: 
\begin{compactenum}[\rm(a)]
\item Simply exhibit a weak projective triple for the pair $(\mathrm{CCTP}_4[n],\F_0(\PS[1]))$.                                                                                                                              
\item Establish general existence results for weakly projective triples.
\end{compactenum}
Both approaches have their merit. The first possibility leads to the low-dimensional families of projectively unique polytopes, and hence to the proof of Theorem~\ref{mthm:projun} as stated. The second possibility leads to families of projectively unique polytopes that exhibit further interesting combinatorial and geometric properties but may well be of a very high (but fixed) dimension. These polytopes are interesting in their own right, and solve more problems in polytope theory. We will discuss them in detail in Section~\ref{sec:varieties}, but follow path (a) first.

\begin{lemma}\label{lem:affine}
There is a point configuration $R$ of $40$ points in $S^4$ such that for all $n\ge  1$, the triple $(\mathrm{CCTP}_4[n],\F_0(\PS[1]),R)$ is a weak projective triple.
\end{lemma}

We defer the construction of $R$ to the next section. Using Lemma~\ref{lem:affine}, we can now prove Theorem~\ref{mthm:projun}.

\begin{proof}[\textbf{Proof of Theorem~\ref{mthm:projun}}]
We stay in the notation of this section. Consider, for any $n\ge  1$, the triple $(\mathrm{CCTP}_4[n],\F_0(\PS[1]),R)$. By Lemma~\ref{lem:subdc}, the subdirect cone
\[(\mathrm{CCTP}_4[n]^v,K),\ \ K:=\F_0(\PS[1])\cup R\]
of the triple $(\mathrm{CCTP}_4[n],\F_0(\PS[1]),R)$ is a projectively unique PP configuration. We can now apply the Lawrence extension (Proposition~\ref{prp:mlwextn}) to $(\mathrm{CCTP}_4[n]^v,K)$ to obtain the family
\[\mathrm{PCCTP}_{69}[n]:=\mathrm{L}(\mathrm{CCTP}_4[n]^v,K).\] This is the family of polytopes announced in Theorem~\ref{mthm:projun}:

\begin{compactitem}[$\circ$]
\item The polytopes $\mathrm{PCCTP}_{69}[n]$ are projectively unique by Proposition~\ref{prp:mlwextn} since the $(\mathrm{CCTP}_4[n]^v,K)$ are projectively unique PP configurations by Lemma~\ref{lem:subdc}.
\item The polytopes $\mathrm{PCCTP}_{69}[n]$ are combinatorially distinct: By Proposition~\ref{prp:mlwextn}, \[f_0(\mathrm{PCCTP}_{69}[n])=f_0(\mathrm{CCTP}_{4}[n])=12(n+1)+129.\] 
\item The dimension of $\mathrm{PCCTP}_{69}[n]$ is given (Proposition~\ref{prp:mlwextn}) as the sum of the cardinality of $K$, which is $24+40 = 64$, and the dimension of the PP configuration $(\mathrm{CCTP}_4[n]^v,K)$, which is $5$.\qedhere \end{compactitem}
\end{proof}

\subsection{Construction of a weak projective triple}\label{ssc:constr}

To prove Lemma~\ref{lem:affine}, we construct a projectively unique point configuration $K:=R\cup \F_0(\PS[1])$ in the closure of $S^4_{+}$. For this we start with a projectively unique PP configuration (the vertex set of a product of two triangles $\Delta_2 \times \Delta_2$) and extend it to the desired point configuration $K$ step by step. We do so in such a way that the ultimate point configuration constructed is determined uniquely (up to Lawrence equivalence) from the vertices of $\Delta_2 \times \Delta_2$.

\smallskip

The following three notational remarks should help the reader navigate through the construction:
\begin{compactitem}[$\circ$]
\item Coordinates for vertices in $S^4_+$ are given in homogeneous coordinates.
\item We start the construction by giving coordinates to the vertex set of $P_8=\Delta_2 \times \Delta_2$. Then we construct another copy $\widetilde{P}_8$ of $\Delta_2 \times \Delta_2$ antipodal to the first. This gives us the $6$-fold symmetry that is inherent in the vertex set $\F_0(\PS[1])$. After this, we will construct the remaining points of $K$ from the vertices of $P_8$ and of $\widetilde{P}_8$ with little interdependence between the two constructions. Thus, we mark points that correspond to $\widetilde{P}_8$ with a tilde. 
\item The symbols $+$ and $-$ will be used to signify parity in the first and second coordinate. A pre-superscript ${\circ}$ will indicate that a point has zero first coordinate function, but nonzero second coordinate function, where there exists a corresponding point with nonzero first coordinate function, but zero second coordinate function.
\item The first three construction steps are direct, in the sense that we obtain new points as intersections of subspaces spanned by points already obtained. From step four on, our construction steps are more involved; the point configurations depend on a parameter $\lambda$ whose value will be determined at the end.
\end{compactitem}

\medskip

\noindent \textbf{I. \emph{The framework, (a).}} The polytope $\Delta_2 \times \Delta_2$ in $S^4$ is projectively unique (cf.\ Shephard's List polytope $P_8$ in Section~\ref{ssc:Shphrdlist}), and so is consequently its vertex set. We choose as vertices of $\Delta_2 \times \Delta_2$
    \vskip -24pt
\begin{align*}
      \parbox{5.5cm}{\begin{align*}
        a_1^\pm&:= \big(\pm 1,\, 0 ,\,-2 ,\,0,\, 1\big),\\
        {^\circ}a_1^+&:= \big( 0,\, 1 ,\,-2 ,\,0,\, 1\big),
      \end{align*}} 
      \parbox{5.5cm}{\begin{align*}
        a_2^\pm&:=  \big(\pm 1,\, 0 ,\, 1 ,\, \sqrt{3},\, 1\big),\\
        {^\circ}a_2^+&:=  \big( 0,\, 1 ,\, 1 ,\, \sqrt{3},\, 1\big),
      \end{align*}}
        \parbox{5.5cm}{\begin{align*}
        a_3^\pm&:=  \big(\pm 1,\, 0 ,\, 1 ,\,-\sqrt{3},\, 1\big),\\
        {^\circ}a_3^+&:=  \big( 0,\, 1 ,\, 1 ,\,-\sqrt{3},\, 1\big).
      \end{align*}}
\end{align*}
    \vskip -12pt
\noindent For $1\le  i\neq j\le  3$, let us denote by $b_{ij}=b_{ji}$ the intersection of $\SSp\{a_i^+,\, a_j^-\}$ and $\SSp\{a_i^-,\, a_j^+\}$ in $S^d_+$, and let ${^\circ}a_i^-$ denote the intersection point of the affine $2$-plane $\SSp\{ a_i^+,\, a_i^-,\,{^\circ}a_i^{+}\}$  with $\SSp\{b_{ij},\, {^\circ}a_j^+\}$. This yields
\vskip -24pt
 \begin{align*}
      \parbox{5.5cm}{\begin{align*}
        b_{23}&:= \big(0,\,0 ,\,1 ,\,0,\, 1\big),\\
        {^\circ}a_1^-&:= \big(0,\,- 1  ,\,-2 ,\,0,\, 1\big),
      \end{align*}} 
      \parbox{5.5cm}{\begin{align*}
        b_{13}&:=  \big(0,\,0,\, -\tfrac{1}{2} ,\, -\tfrac{\sqrt{3}}{2},\, 1\big),\\
        {^\circ}a_2^-&:=  \big(0,\,- 1 ,\, 1 ,\, \sqrt{3},\, 1\big),
      \end{align*}}
        \parbox{5.5cm}{\begin{align*}
        b_{12}&:=  \big(0,\,0,\, -\tfrac{1}{2} ,\, \tfrac{\sqrt{3}}{2},\, 1\big),\\
        {^\circ}a_3^-&:=  \big(0,\,- 1  ,\, 1 ,\,-\sqrt{3},\, 1\big).
      \end{align*}}
    \end{align*}
        \vskip -12pt
\noindent Let us denote by $b_0$ the intersection of the $2$-planes  $\SSp\{a_i^+,a_j^-,a_k^-\}$, $i,\, j,\, k \in \{1,2,3\},\, i\neq j\neq k\neq i$ in~$S^d_+$. The coordinates of this point are determined by $b_0=(0,\,0,\,0,\,0,\,1)$.
\bigskip

\noindent \textbf{II. \emph{The framework, (b).}} For $1\le  i\neq j\le  3$, let us denote by $k$ the last remaining index $\{1,2,3\}{\setminus} \{i,j\}$, and let $\widetilde{a}_{k}^+$ denote the intersection of the $2$-plane $\SSp\{a_1^+,\,a_2^+,\,a_3^+\}$ with the line $\SSp\{b_0,\, a_{ij}^-\}$ in $S^d_+$. We obtain, 
     \vskip -24pt
    \begin{align*}
      \parbox{5.5cm}{\begin{align*}
       \widetilde{a}_1^+&:= \big( 1,\, 0 ,\,2 ,\,0,\, 1\big),
      \end{align*}} 
      \parbox{5.5cm}{\begin{align*}
       \widetilde{a}_2^+:=  \big(1,\, 0 ,\, -1 ,\, -\sqrt{3},\, 1\big),
      \end{align*}}
        \parbox{5.5cm}{\begin{align*}
        \widetilde{a}_3^+:=  \big(1,\, 0 ,\, -1 ,\,\sqrt{3},\, 1\big).
      \end{align*}}
    \end{align*}
     \vskip -12pt
\noindent Analogously, we obtain
     \vskip -24pt
    \begin{align*}
      \parbox{5.5cm}{\begin{align*}
       \widetilde{a}_1^-&:= \big( -1,\, 0 ,\,2 ,\,0,\, 1\big),\\
        {^\circ} \widetilde{a}_1^\pm&:= \big( 0,\, \pm 1  ,\,2 ,\,0,\, 1\big),
      \end{align*}} 
      \parbox{5.5cm}{\begin{align*}
        \widetilde{a}_2^-&:=  \big( -1,\, 0 ,\, -1 ,\,-\sqrt{3},\, 1\big),\\
        {^\circ} \widetilde{a}_2^\pm&:=  \big( 0,\,\pm 1  ,\, -1 ,\,-\sqrt{3},\, 1\big),
      \end{align*}}
        \parbox{5.5cm}{\begin{align*}
       \widetilde{a}_3^-&:=  \big(-1,\, 0 ,\, -1 ,\,\sqrt{3},\, 1\big),\\
        {^\circ} \widetilde{a}_3^\pm&:=  \big(0,\,\pm 1 ,\, -1 ,\,\sqrt{3},\, 1\big)
      \end{align*}}
    \end{align*}
         \vskip -12pt
\noindent uniquely from the points already constructed.
\bigskip

\noindent \textbf{III. \emph{The points of layer 1.}} Again for $1\le  i\neq j\le  3$, we denote by $\psi_{k}^+$ the intersection point of the lines $\SSp\{a_{k}^+,\, \widetilde{a}_{k}^+\}$ with $\SSp\{a_i^+,\, a_j^+\}$ in $S^d_+$. We obtain
     \vskip -24pt
    \begin{align*}
      \parbox{5.5cm}{\begin{align*}
       \psi_{1}^+&:= \big( 1,\, 0 ,\,1 ,\,0,\, 1\big),
      \end{align*}} 
      \parbox{5.5cm}{\begin{align*}
       \psi_{2}^+&:=  \big(1,\, 0 ,\, -\tfrac{1}{2} ,\, -\tfrac{\sqrt{3}}{2},\, 1\big),
      \end{align*}}
        \parbox{5.5cm}{\begin{align*}
        \psi_{3}^+&:=  \big( 1,\, 0 ,\, -\tfrac{1}{2} ,\,\tfrac{\sqrt{3}}{2},\, 1\big).
      \end{align*}}
    \end{align*}
         \vskip -12pt
\noindent Analogously, we obtain 
     \vskip -24pt
    \begin{align*}
      \parbox{5.5cm}{\begin{align*}
       \psi_{1}^-&:= \big( -1,\, 0 ,\,1 ,\,0,\, 1\big),\\
        {^\circ}\widetilde{\psi}_{1}^\pm&:= \big( 0,\, \pm 1 ,\, -1 ,\,0,\, 1\big),
      \end{align*}} 
      \parbox{5.5cm}{\begin{align*}
        \psi_{2}^-&:=  \big(-1,\, 0 ,\, -\tfrac{1}{2} ,\, -\tfrac{\sqrt{3}}{2},\, 1\big),\\
        {^\circ}\widetilde{\psi}_{2}^\pm&:=  \big(0,\, \pm 1 ,\, \tfrac{1}{2} ,\, \tfrac{\sqrt{3}}{2},\, 1\big),
      \end{align*}}
        \parbox{5.5cm}{\begin{align*}
       \psi_{3}^-&:=  \big(-1,\, 0 ,\, -\tfrac{1}{2} ,\,\tfrac{\sqrt{3}}{2},\, 1\big),\\
        {^\circ}\widetilde{\psi}_{3}^\pm &:=  \big(0,\, \pm 1 ,\, \tfrac{1}{2} ,\,-\tfrac{\sqrt{3}}{2},\, 1\big)
      \end{align*}}
    \end{align*}
     \vskip -12pt
\noindent from the points already constructed. The points ${^\circ}\widetilde{\psi}_{i}^\pm$ and $\psi_{i}^\pm$, $i\in \{1,2,3\}$, together form $\RR(\PS[1],1)$, with  $\psi_{1}^+=\vartheta_1$, cf.\ Section~\ref{ssc:example}.
\bigskip

\noindent \textbf{IV. \emph{Transition to layer 0.}} Define \[b_1:=\SSp\{a_1^+,\, a_1^-\}\cap\SSp\{{^\circ}a_1^+,\,{^\circ} a_1^-\}\cap S^d_+\]
Similarly, define \[b_{12}^{++}:=\SSp\{a_1^+,\,{^\circ} a_2^+\}\cap\SSp\{a_2^+,\,{^\circ} a_1^+\}\cap S^d_+\quad \text{and}\quad b_{12}^{+-}:=\SSp\{a_1^+,\,{^\circ} a_2^-\}\cap \SSp\{a_2^+,\,{^\circ} a_1^-\}\cap S^d_+.                                                                                                                                                                                                                 \]
We obtain
     \vskip -24pt
    \begin{align*}
      \parbox{5.5cm}{\begin{align*}
       b_1&:=\big(0,\,0,\,-2,\,0,\, 1\big),
      \end{align*}} 
      \parbox{5.5cm}{\begin{align*}
       b_{12}^{++}&:=\big(\tfrac{1}{2},\,\tfrac{1}{2},\,-\tfrac{1}{2},\,\tfrac{\sqrt{3}}{2},\, 1\big),
      \end{align*}}
        \parbox{5.5cm}{\begin{align*}
        b_{12}^{+-}&:=\big(\tfrac{1}{2},\,-\tfrac{1}{2},\,-\tfrac{1}{2},\,\tfrac{\sqrt{3}}{2},\, 1\big).
      \end{align*}}
    \end{align*}
         \vskip -12pt
\begin{figure}[htbf]
\centering 
 \includegraphics[width=0.34\linewidth]{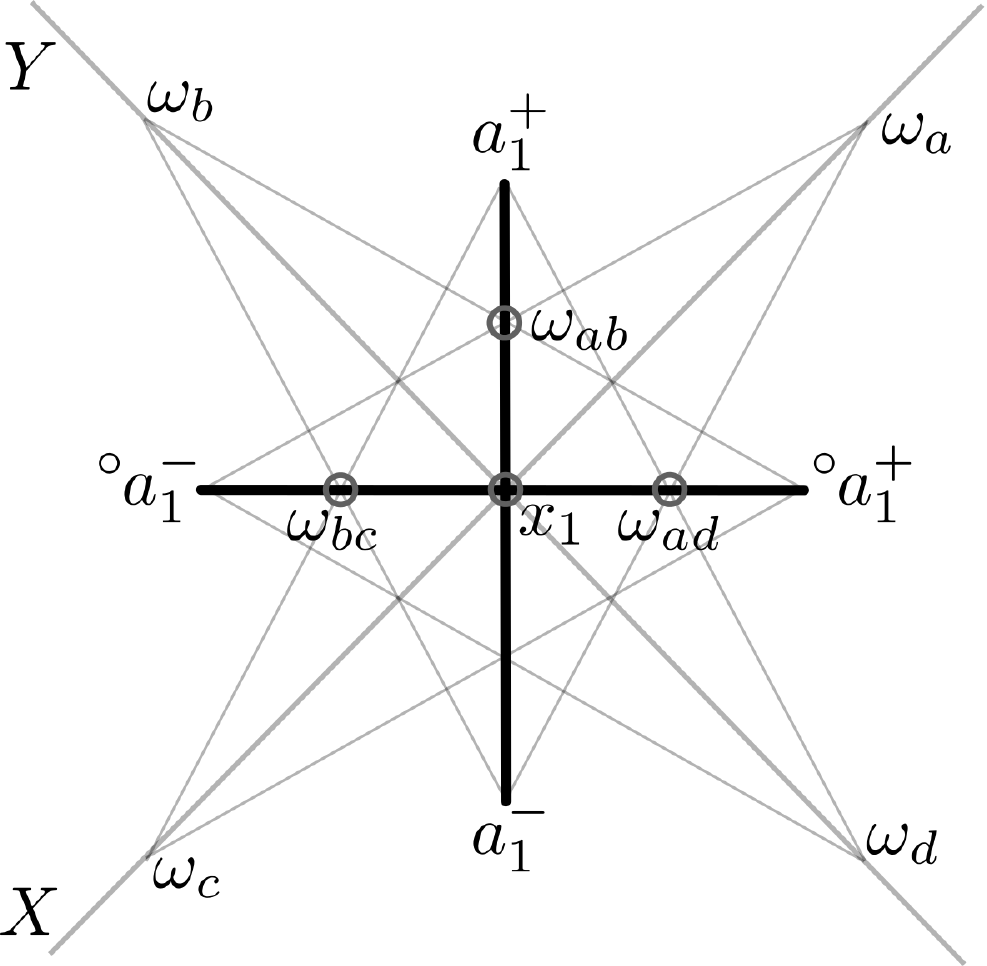} 
\caption{\small  The quadrilateral $W_1$ on vertices $\{\omega_a,\, \omega_b,\,\omega_c,\,\omega_d\}$ is, up to dilation, determined by the quadrilateral on vertices $\{a_1^+,\, a_1^-,\,{^\circ}a_1^+,\,{^\circ} a_1^-\}$ and the lines $X$ and $Y$ in $\SSp\{a_1^+,\, a_1^-,\,{^\circ}a_1^+,\,{^\circ} a_1^-\}$.} 
  \label{fig:quad}
\end{figure}

\noindent Let us now consider any set of distinct points $\{\omega_a,\, \omega_b,\,\omega_c,\,\omega_d\}$ in $S^d_+$, cf.\ Figure~\ref{fig:quad}, such that
\begin{compactitem}[$\circ$]
\item all the points lie in the affine $2$-plane $\SSp\{a_1^+,\, a_1^-,\,{^\circ}a_1^+,\,{^\circ} a_1^-\}$,
\item $\omega_a$, $\omega_c$ lie in the intersection $X$ of $\SSp\{b_1,\, b_{12},\, b_{12}^{++}\}$ with $\SSp\{a_1^+,\, a_1^-,\,{^\circ}a_1^+,\,{^\circ} a_1^-\}$,
\item $\omega_b$, $\omega_d$ lie in the intersection $Y$ of $\SSp\{b_1,\, b_{12},\, b_{12}^{+-}\}$ with $\SSp\{a_1^+,\, a_1^-,\,{^\circ}a_1^+,\,{^\circ} a_1^-\}$,
\item the intersection point $\omega_{ad}$ of $\SSp\{\omega_a, a_1^-\}$ with 
$\SSp\{\omega_d, a_1^+\}$ and the intersection point $\omega_{bc}$ of $\SSp\{\omega_b, a_1^-\}$ with $\SSp\{\omega_c, a_1^+\}$) 
 lie in $\SSp\{{^\circ}a_1^+,\,{^\circ} a_1^-\}$,
\item the intersection point $\omega_{ab}$ of $\SSp\{\omega_a,{^\circ}a_1^-\}$ with $\SSp\{\omega_b, {^\circ}a_1^+\}$ lie in $\SSp\{a_1^+,\, a_1^-\}$, and
\item $\omega_a$ and ${^\circ}a_1^+$ lie in the same component of $\SSp\{a_1^+,\, a_1^-,\,{^\circ}a_1^+, \,{^\circ} a_1^-\}{\setminus} \SSp\{a_1^+,\, a_1^-\}$.
\end{compactitem}
\noindent The quadrilateral $W_1$ on points $\{\omega_a,\, \omega_b,\,\omega_c,\,\omega_d\}$ is a square; its barycenter coincides with $b_1$, but the dilation factor $\lambda$ of $W_1$ is not determined yet. Since the vertices of $W_1$ are of the form $\big(s_3 \lambda,\, s_4 \lambda,\,-2,\,0,\, 1\big)$, where $\lambda>0$ and $s_3,s_4\in\{+,-\}$,
we relabel \[\omega_1^{+,+}:=\omega_a,\ \ \omega_1^{+,-}:=\omega_b,\ \ \omega_1^{-,-}:=\omega_c\ \ \text{and}\ \omega_1^{-,+}:=\omega_d\]
with
\[\omega_1^{s_3,s_4}:=\big(s_3 \lambda,\, s_4 \lambda,\,-2,\,0,\, 1\big),\  \lambda>0,\, s_3,s_4\in\{+,-\}.\]
\noindent Call this PP configuration $K'_\lambda$. Before we turn to the issue of fixing $\lambda$, let us construct analogous quadrilaterals $W_2$ and $W_3$ and $\widetilde{W}_1$, $\widetilde{W}_2$ and $\widetilde{W}_3$.
\bigskip

\noindent \textbf{V. \emph{The points of layer 0, (a).}} Let $i\in \{2,3\}$. Define, for any choice of signs $s_3$ and $s_4$, $s_3=s_4$, the points $\omega_i^{s_3,s_4}$ in $S^d_+$ by
\[\omega_i^{s_3,s_4}=\SSp\big\{a_i^+,\, a_i^-,\,{^\circ}a_i^+,\,{^\circ} a_i^-\big\}\cap \SSp\{a_1^+,\, a_i^+,\,\omega_1^{s_3,s_4}\}  \cap \SSp\{{^\circ}a_1^+,\, {^\circ}a_i^+,\,\omega_1^{s_3,s_4} \}\cap  S^d_+.\]
We obtain:
\[\omega_2^{s_3,s_4}:=\big(s_3 \lambda,\, s_4 \lambda,\,1,\,\sqrt{3},\, 1\big),\ \quad \omega_3^{s_3,s_4}:=\big(s_3 \lambda,\, s_4 \lambda,\,1,\,-\sqrt{3},\, 1\big),\  \lambda>0,\, s_3,s_4\in\{+,-\},\, s_3=s_4.\]

\begin{figure}[htbf]
\centering 
 \includegraphics[width=0.4\linewidth]{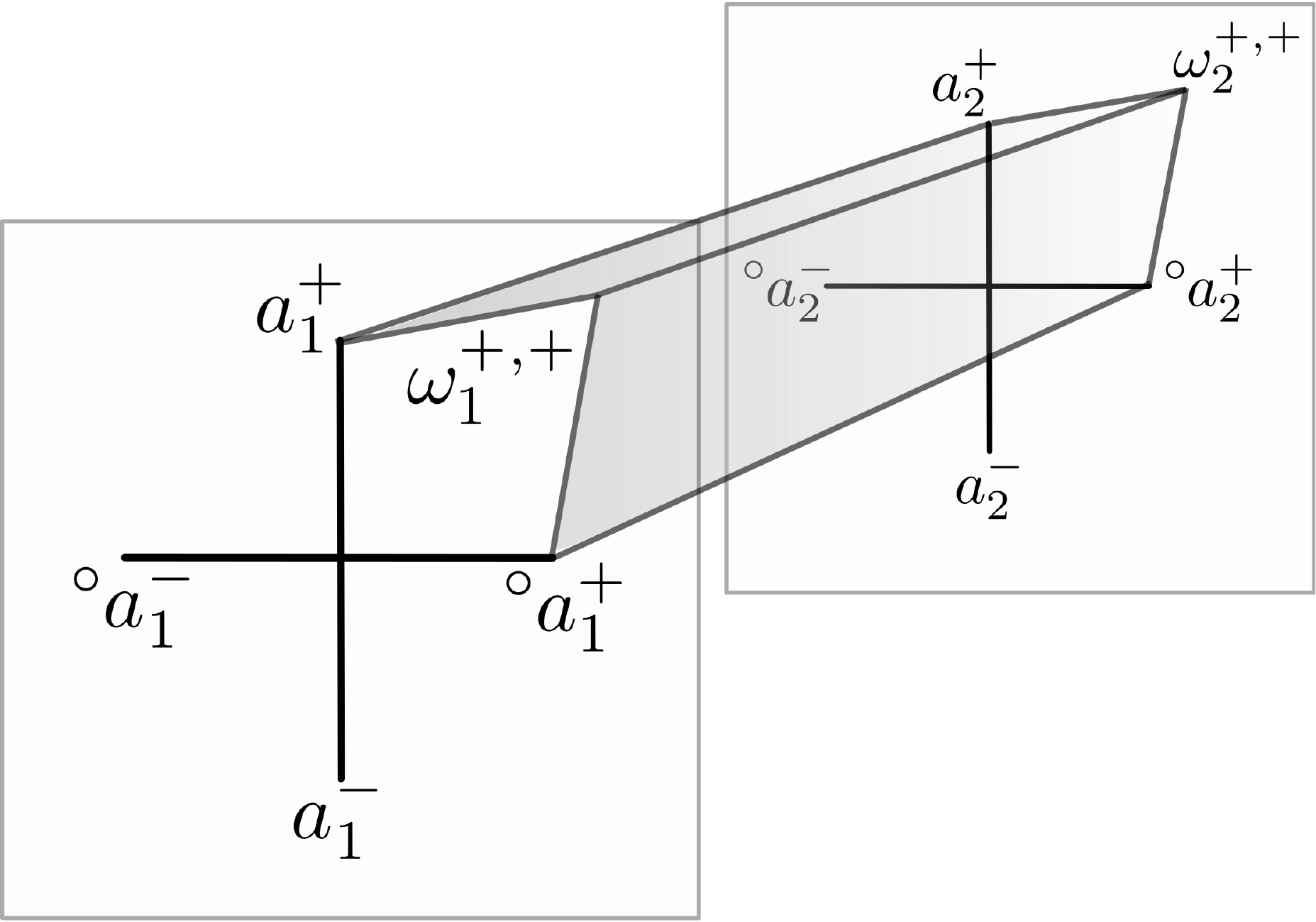} 
\caption{\small Transferring points between planes: We determine $\omega_2^{+,+}\in \SSp\{a_2^+,\, a_2^-,\,{^\circ}a_2^+,\,{^\circ} a_2^-\}$ from the points $a_1^+,\, a_1^-,\,{^\circ}a_1^+,\,{^\circ} a_1^-$, $\,a_2^+,\, a_2^-,\,{^\circ}a_2^+,\,{^\circ} a_2^-$ and $\omega_1^{+,+}\in \SSp\{a_1^+,\, a_1^-,\,{^\circ}a_1^+,\,{^\circ} a_1^-\}$.} 
  \label{fig:quadtrans}
\end{figure}

\noindent Similarly, for $i\in \{1,2,3\}$ and for any choice of signs $s_3$ and $s_4$ with $s_3\neq s_4$, define the points $\widetilde{\omega}_i^{s_3,s_4}$  in $S^d_+$ by
\[\widetilde{\omega}_i^{s_3,s_4}=\SSp\big\{\widetilde{a}_i^+,\, \widetilde{a}_i^-,\,{^\circ}\widetilde{a}_i^+,\,{^\circ} \widetilde{a}_i^-\big\}\cap \SSp\{a_1^+,\, \widetilde{a}_i^+,\,\omega_1^{s_3,s_4}\} \cap \SSp\{{^\circ}a_1^+,\, {^\circ}\widetilde{a}_i^+,\,\omega_1^{s_3,s_4} \}\cap  S^d_+ .\]
We have, for $\lambda>0$ and $s_3,s_4\in\{+,-\}$, $s_3\neq s_4$,
     \vskip -24pt
    \begin{align*}
      \parbox{5.5cm}{\begin{align*}
      \widetilde{\omega}_1^{s_3,s_4}&:=\big(s_3 \lambda,\, s_4 \lambda,\, 2,\,0 ,\, 1\big),
      \end{align*}} 
      \parbox{5.5cm}{\begin{align*}
       \widetilde{\omega}_2^{s_3,s_4}&:=\big(s_3 \lambda,\, s_4 \lambda,\,-1,\,-\sqrt{3},\, 1\big),
      \end{align*}}
       \parbox{5.5cm}{\begin{align*}
       \widetilde{\omega}_3^{s_3,s_4}&:=\big(s_3 \lambda,\, s_4 \lambda,\,-1,\,\sqrt{3},\, 1\big).
      \end{align*}}
    \end{align*}
         \vskip -12pt
\noindent The point configuration $K_{\lambda}$ constructed depends on a parameter $\lambda>0$.
\bigskip

\noindent \textbf{VI. \emph{The points of layer 0, (b).}} The point set \[\Omega(\lambda):=\big\{\widetilde{\omega}^{\pm,\mp}_i:\ i \in \{1,2,3\}\} \cup\, \{\omega^{\pm,\pm}_i:\ i \in \{1,2,3\}\big\}\]
is obtained from the point 
\[\widetilde{\omega}^{+,-}_1=\big(\lambda,\,-\lambda,\,2,\,0,\, 1\big),\	 \lambda>0\]
as its orbit under $\mathfrak{R}$. Recall that the points of $\PS[1]$ are given as the orbits of $\vartheta_0$ and $\vartheta_1$ under the group of rotational symmetries $\mathfrak{R}\subset O(\R^5)$, respectively (cf.\ Section~\ref{ssc:example}). We need to determine the values of~$\lambda$ for which the point set 
\[\Omega'(\lambda):=\big\{\psi_{i}^\pm:\ i\in \{1,2,3\}\big\}\cup \big\{{^\circ}\widetilde{\psi}_{i}^\pm:\ i\in \{1,2,3\}\big\}\cup \Omega(\lambda)\]
is Lawrence equivalent to the point configuration $\F_0(\PS[1])$, where we label the points such that
\begin{itemize}[$\circ$]
 \item $(\rot_{1,2}\rot_{3,4})^\ell (\rot_{3,4})^{2m} \widetilde{\omega}^{+,-}_1$ corresponds to $(\rot_{1,2}\rot_{3,4})^\ell (\rot_{3,4})^{2m} \vartheta_0$ for all integers $m, \ell$. In particular, the point configuration $\Omega(\lambda)\subset\Omega'(\lambda)$ corresponds to the point configuration $\RR(\PS[1], 0)\subset \F_0(\PS[1])$.
 \item $(\rot_{1,2}\rot_{3,4})^\ell (\rot_{3,4})^{2m} \psi_{1}^+$ corresponds to $(\rot_{1,2}\rot_{3,4})^\ell (\rot_{3,4})^{2m} \vartheta_1$ for all integers $m, \ell$. In particular, the point configuration $\{\psi_{i}^\pm:\ i\in \{1,2,3\}\}\cup \{{^\circ}\widetilde{\psi}_{i}^\pm:\ i\in \{1,2,3\}\}\subset\Omega'(\lambda)$ corresponds to the point configuration $\RR(\PS[1], 1)\subset \F_0(\PS[1])$.
\end{itemize}

\smallskip

\noindent This requirement on $\Omega'(\lambda)$ imposes strong conditions on $\lambda$:
The points  \[\vartheta_1,\ \rot_{1,2}\rot_{3,4}\vartheta_1,\ \rot_{1,2}\rot^{-1}_{3,4}\vartheta_1,\ \rot_{1,2}^2\vartheta_0,\ \rot_{1,2}\rot_{3,4}^{-1}\vartheta_0\ \text{and}\ \rot_{1,2}\rot_{3,4}\vartheta_0\]
of $\F_0(\PS[1])$ lie in a common hyperplane. In particular, the corresponding points \[\psi_{1}^+=\vartheta_1,\ \rot_{1,2}\rot_{3,4}\psi_{1}^+,\ \rot_{1,2}\rot^{-1}_{3,4}\psi_{1}^+,\ \rot_{1,2}^2\widetilde{\omega}^{+,-}_1=\widetilde{\omega}^{-,+}_1,\ \rot_{1,2}^{-1}\rot_{3,4}^{-1}\widetilde{\omega}^{-,+}_1\ \text{and}\ \rot_{1,2}^{-1}\rot_{3,4}\widetilde{\omega}^{-,+}_1\] 
of $\Omega'(\lambda)$ with coordinates
\[
\big(1,\,0,\,1,\,0,\, 1\big), \big(0,\,1,\,\tfrac{1}{2},\,\tfrac{\sqrt{3}}{2},\, 1\big), \big(0,\,1,\,\tfrac{1}{2},\,-\tfrac{\sqrt{3}}{2},\, 1\big),
\big(-\lambda,\,\lambda,\,2,\,0,\, 1\big), \big(\lambda,\,\lambda,\,1,\,\sqrt{3},\, 1\big)\ \text{and}\ \big(\lambda,\,\lambda,\,1,\,-\sqrt{3},\, 1\big)\]
 must lie in a common hyperplane. In particular, $\lambda$ must satisfy
\[\lambda^2+2\lambda-1=0\Longleftrightarrow \lambda=\pm \sqrt{2}-1.\]
\noindent Since $\lambda>0$, the only choice for $\lambda$, assuming $\F_0(\PS[1])'(\lambda)$ is Lawrence equivalent to $\F_0(\PS[1])$, is \[\lambda= \sqrt{2}-1.\]
\noindent In particular, $K_{\sqrt{2}-1}$ is projectively unique, and it contains $\F_0(\PS[1])$. 
\bigskip

\noindent \textbf{VII. \emph{Four points at infinity}.}
No subset of $K_{\sqrt{2}-1}$ spans a hyperplane that does not intersect the $\mathrm{CCTP}_{4}[n]$, but there is an easy way to obtain such a point set from $K_{\sqrt{2}-1}$:
Denote the intersection of $\SSp\{a_1^+,\,a_2^+\}$ with $\SSp\{a_1^-,\,a_2^-\}$ (in our case, this intersection point lies in the equator $S^3_\eq$) by $\infty_1$. Similarly, denote the intersection of $\SSp\{a_3^+,\,a_2^+\}$ with $\SSp\{a_3^-,\,a_2^-\}$ by $\infty_2$, the intersection of $\SSp\{a_1^+,\,a_1^-\}$ with the ray $\SSp\{a_2^+,\,a_2^-\}$ by $\infty_3$, and the intersection of $\SSp\{{^\circ}a_1^+,\,{^\circ}a_1^-\}$ with  $\SSp\{{^\circ}a_2^+,\,{^\circ}a_2^-\}$ by $\infty_4$. The points $\infty_1,\,\dots,\ \infty_4$ span $S^3_\eq$. Denote the union of these four points with $K_{\sqrt{2}-1}$ by $K$, and set $R=K{\setminus} \F_0(\PS[1])$. For the proof of Lemma~\ref{lem:subdc}, we now only have to examine the point configuration $K$ carefully.

\begin{proof}[\textbf{Proof of Lemma~\ref{lem:affine}}]
First, we note that $f_0(K)=64$ and $f_0(R)=40$. It remains to verify that $(\mathrm{CCTP}_4[n],\F_0(\PS[1]),R)$ is a weak projective triple:
\begin{compactenum}[\rm(a)]
\item All points of $K_{\sqrt{2}-1}\subset K$ lie in $S^4_+$, as well as the polytopes $\mathrm{CCTP}_4[n]$. The $4$ points $\infty_1,\,\infty_2,\, \infty_3,\ \infty_4$ form the vertices of a tetrahedron in $S^3_\eq$. By perturbing the closure of $S^4_{+}$ a little, we can find an open hemisphere that contains the points of $K$ and the polytopes $\mathrm{CCTP}_4[n]$ in the interior.
\item $K=R\cup \F_0(\PS[1])$ is projectively unique, since it is uniquely determined from the projectively unique PP configuration $\F_0(\Delta_2\times \Delta_2)$ up to Lawrence equivalence.
\item $\F_0(\PS[1])$ determines $\mathrm{CCTP}_{4}[n]$ for all $n$. See also Example~\ref{ex:stdet}(d).
\item Since the polytopes $\mathrm{CCTP}_4[n]$ lie in the interior of the upper hemisphere, we obtain that the subspace $\SSp\{\infty_1,\,\dots,\ \infty_4\}$ does not intersect any of the cross-bedding cubical torus polytopes $\mathrm{CCTP}_4[n]$. \qedhere
\end{compactenum}
\end{proof}

\appendix
\setcounter{figure}{0}
\section{Appendix}\label{sec:app}
\renewcommand{\thetable}{A.\arabic{table}}
\renewcommand{\thefigure}{A.\arabic{figure}}
\subsection{Convex position of polytopal complexes}\label{sec:convps}

In this section, we establish polyhedral analogues of the Alexandrov--van Heijenoort theorem, a classical result in the theory of convex hypersurfaces. 

\begin{theorem}[{\cite[Main Theorem]{Heij}}]\label{thm:heijor}
Let $M$ be an immersed topological $(d-1)$-manifold without boundary in $\R^d,\, d\ge  3$, such that
\begin{compactenum}[\rm(a)]
\item $M$ is complete with respect to the metric induced on $M$ by the immersion,
\item $M$ is connected,
\item $M$ is locally convex at each point (that is, every point $x$ of $M$ has a neighborhood, w.r.t.\ the topology induced by the immersion, in which $M$ coincides with the boundary of some convex body~$K_x$), and
\item $M$ is strictly convex in at least one point (that is, for at least one $x\in M$, there exists a hyperplane $H$ intersecting the convex body~$K_x$ only in $x$).
\end{compactenum}
\noindent Then $M$ is embedded, and it is the boundary of a convex body. 
\end{theorem}

This remarkable theorem is due to Alexandrov~\cite{Aleks} in the case of surfaces. Alexandrov did not state it explicitly (his motivation was to prove far stronger results on intrinsic metrics of surfaces), and his proof does not extend to higher dimensions. Van Heijenoort also proved Theorem~\ref{thm:heijor} only for surfaces, but his proof extends to higher dimensions. Expositions of the general case of this theorem and further generalizations are available in~\cite{TrudingerWang} and~\cite{Rybnikov}. In this section, we adapt Theorem~\ref{thm:heijor} to polytopal manifolds with boundary. We start off by introducing the notion of immersed polytopal complexes.

\begin{definition}[Precomplexes] Let $C$ denote an abstract polytopal complex, and let $f$ denote an immersion of $C$ into $\R^d$ (resp.\ $S^d$) with the property that $f$ is an isometry on every face $\sigma$ of $C$. Then $f(C)$ is called a \Defn{precomplex} in $\R^d$ (resp.\ $S^d$). While $f(C)$ is not necessarily a polytopal complex, the subset $f(\St(\sigma,C))$ is a polytopal complex combinatorially equivalent to $\St(\sigma,C)$ for every face $\sigma$ of~$C$. 

A polytopal complex $C$, abstract or geometric, is a \Defn{$d$-manifold} if for every vertex $v$ of $C$, $\St(v,C)$ is PL-homeomorphic to a $d$-simplex~\cite{RourkeSanders}. If $C$ is a manifold, then $f(C)$ is called a \Defn{premanifold}. If $\sigma$ is a face of $C$, then $f(\sigma)$ is a \Defn{face} of the polytopal precomplex $f(C)$. The complexes $\St(f(\sigma),f(C))$ and $\Lk(f(\sigma),f(C))$ are defined to be the polytopal complexes $f(\St(\sigma,C))$ and $\Lk(f(\sigma), f(\St(\sigma,C)))$ respectively. 
\end{definition}

\begin{definition}[Gluing two precomplexes along a common subcomplex]
Consider two precomplexes $f(C)$, $g(C')$, and assume there are subcomplexes $D$, $D'$ of $C$, $C'$ respectively with $f(D)=g(D')$ such that for every vertex $v$ of $f(D)=g(D')$, $\St(v,f(C))\cup\St(v,g(C'))$ is a polytopal complex. Let $\Sigma$ be the abstract polytopal complex given by identifying $C$ and $C'$ along the map $g^{-1}\circ f:D\rightarrow D'$, and let $s$ denote the immersion of $\Sigma$ defined as 
\[s(x):=\Big\{\begin{array} {cl} 
f(x) &\ \text{for }x\in C,    \\
g(x)&\ \text{for }x\in C'. \end{array}\]
Then the \Defn{gluing} of $f(C)$ and $g(C')$ at $f(D)=g(D')$, denoted by $f(C)\sqcup_{f(D)} g(C')$, is the polytopal precomplex obtained as the image of $\Sigma$ under $s$. 
\end{definition}

For the rest of the section, we will use the term \Defn{halfspace} in $S^d$ synonymously for a hemisphere in~$S^d$. 

\begin{definition}[Locally convex position and convex position for polytopal (pre-)complexes]
A pure polytopal (pre-)complex $C$ in $S^d$ or in $\R^{d}$ is \Defn{in convex position} if one of the following three equivalent conditions is satisfied:
\begin{compactitem}[$\circ$]
\item For every facet $\sigma$ of $C$, there exists a closed halfspace $H(\sigma)$ containing $C$ such that $\parti H(\sigma)$ contains the vertices of $\sigma$ but no other vertices of $C$. We say that such a halfspace $H(\sigma)$ \Defn{exposes} $\sigma$ in $C$.
\item Every facet is exposed by some linear functional, i.e.\ there exists, for every facet $\sigma$ of $C$, a vector $\n(\sigma)$ such that the points of $\sigma$ maximize the linear functional $\langle\n,x\rangle$ among all points $x\in C$. For $C\subset S^d$, we additionally demand $\langle\n(\sigma),x\rangle =0$ for all $x$ in $\sigma$.
\item $C$ is a subcomplex of the boundary complex of a convex polytope. 
\end{compactitem}

\noindent Likewise, a polytopal (pre-)complex $C$ in $S^d$ or $\R^{d}$ is \Defn{in locally convex position} if for every vertex $v$ of $C$, 
the link $\Lk(v,C)$, seen as a subcomplex in the $(d-1)$-sphere $\NO^1_v \R^d$ resp.\ $\NO^1_v S^d$, is in convex position. As $\Lk(v,C)$ is in convex position if and only if $\St(v,C)$ is in convex position, $C$ is in locally convex position if and only if $\St(v,C)$ is in convex position for every vertex $v$ of $C$.
\end{definition}

It is obvious that ``convex position'' implies ``locally convex position.'' The Alexandrov--van Heijenoort theorem describes conditions under which this observation can be reversed. We start with a direct analogue of the Theorem~\ref{thm:heijor} for precomplexes. Notice that a precomplex without boundary in locally convex position is locally convex at every point and strictly convex at every vertex in the sense of Theorem~\ref{thm:heijor}.

\begin{theorem}[AvH for closed precomplexes]\label{thm:heij}
Let $C$ be a $(d-1)$-dimensional connected closed polytopal premanifold in $\R^d$ or $S^d$, $d\ge  3$, in locally convex position. Then $C$ is in convex position.
\end{theorem}

\begin{proof}
In the euclidean case, the metric induced on $C$ is complete because $C$ is finite, and $C$ is locally convex at each point since $C$ is in locally convex position. Furthermore, $C$ is strictly convex at every vertex of $C$. Thus by Theorem~\ref{thm:heijor}, $C$ is the boundary of a convex polytope. Since every facet of $C$ is exposed by a linear functional, the boundary complex of this polytope coincides with $C$.

For the spherical case, let $v$ be a vertex of $C$, and let $P$ be the polyhedron in $S^d$ that is obtained by intersecting the halfspaces exposing the facets of $\St(v,C)$. Let $H$ denote a closed halfspace containing $P$, chosen so that $\parti H\cap  \St(v,C)=v$. As $C$ is polytopal, the complex $C$ contains at least one vertex $w$ in the interior of $H$. Consider any central projection $\zeta$ mapping $\intx H$ to $\R^d$. Then $|\zeta(C)|$ is the boundary of a convex polyhedron $K$ in $\R^d$ by Theorem~\ref{thm:heijor}. In particular $K\subseteq \zeta(P)$. Since $K$ is pointed at the vertex $\zeta(w)$, $K$ contains no line, and consequently, $\cl\zeta^{-1}(K)$ contains no antipodal points. Thus \[\cl\zeta^{-1}(K)=\zeta^{-1}(K)\cup v\subsetneq P\] is a polytope in $S^d$. Since every facet of $C$ is exposed by a halfspace, the boundary complex of $\zeta^{-1}(K)$ is~$C$, as desired.
\end{proof}

\begin{example}
A polytopal premanifold $C$ is called \Defn{simple} if for every vertex $v$ of $C$, $\Lk(v,C)$ is a simplex. Consider now any simple, closed and connected $k$-premanifold $C$ in $\R^d$, where $d> k\ge  2$. Since $C$ is simple and connected, it is contained in some affine $(k+1)$-dimensional subspace of $\R^d$. Since $C$ is simple, it is either in locally convex position, or locally flat (i.e.\ locally isometric to $\R^k$). Since $C$ is furthermore compact, only the former is possible. To sum up, $C$ is a $k$-dimensional premanifold that is closed, connected and in locally convex position in some $(k+1)$-dimensional affine subspace. Hence $C$ is in convex position by Theorem~\ref{thm:heij} (cf.~\cite[Sec.\ 11.1, Pr.\ 7]{Grunbaum}).
\end{example}

\begin{definition}[Fattened boundary] Let $C$ be a polytopal $d$-manifold, and let $B$ be a connected component of its boundary. The \Defn{fattened boundary} $\fat(B,C)$ of $C$ at $B$ is the minimal subcomplex of $C$ containing all facets of $C$ that intersect $B$ in a $(d-1)$-face.
\end{definition}

\begin{lemma}[Gluing lemma]\label{lem:convglue}
Let $C$, $C'$ denote two connected polytopal $(d-1)$-manifolds with boundary in $S^d$ or $\R^d,\, d\ge  2 $ with $B:=\parti C= \parti C'$. Assume that $C$ and $C'\cup \fat(B,C)$ are in convex position.
Then $C\cup C'$ is the boundary complex of a convex polytope.
\end{lemma}

\begin{proof}
We proceed by induction on the dimension. First, consider the case $d=2$, whose treatment differs from the case $d> 2$ since Theorem~\ref{thm:heijor} is not applicable. We use the language of curvature of polygonal curves, cf.~\cite{Sullivan}. If $C$, $C'$ are in $S^2$, use a central projection to transfer $C$ and $C'$ to complexes in convex position in $\R^2$. If there are two curves $C$ and $C'$ in convex position in $\R^2$ such that $C'\cup \fat(B,C)$ is in convex position, then $C\sqcup_{\fat(B,C)} C'$ is a 1-dimensional premanifold whose curvature never changes sign, and which is of total curvature less than $4\pi$ since the total curvature of $C'\cup \fat(B,C)$ is smaller or equal to $2\pi$, and $C$ has total curvature less than $2\pi$. Since the turning number of a closed planar curve is a positive integer multiple of $2\pi$, the total curvature of $C\sqcup_{\fat(B,C)} C'$ is $2\pi$. By Fenchel's Theorem~\cite{Fenchel}, $C\sqcup_{\fat(B,C)} C'$ is the boundary of a planar convex body. Since every facet is exposed, the boundary complex of this convex body must coincide with $C\sqcup_{\fat(B,C)} C'=C\cup C'$.

We proceed to prove the lemma for dimension $d> 2$. If $v$ is a vertex of $B$, then $\Lk(v,C\sqcup_{\fat(B,C)} C')$ is obtained by gluing the two complexes $\Lk(v,C)$ and $\Lk(v,C')\cup \Lk(v,\fat(B,C))$ along $\fat(B,C)$.  Each of these is of codimension $1$ and in convex position, so the resulting complex is a polytopal sphere in convex position by induction on the dimension. In particular, $C\sqcup_{\fat(B,C)} C'$ is a premanifold in locally convex position. Thus by Theorem~\ref{thm:heij}, $C\sqcup_{\fat(B,C)} C'=C\cup C'$ is in convex position.
\end{proof}

We will apply Theorem~\ref{thm:heij} in the following version for manifolds with boundary:
\begin{theorem}\label{thm:locglowib}
Let $C$ be a polytopal connected $(d-1)$-dimensional (pre-)manifold in locally convex position in $\R^d$ or in $S^d$ with $d\ge  3$, and assume that for all boundary components $B_i$ of $\parti C$, their fattenings $\fat(B_i,C)$ are (each on its own) in convex position. Then $C$ is in convex position. 
\end{theorem}

\begin{proof}
Consider any boundary component $B$ of $C$ and the boundary complex $\parti \conv \fat(B,C)$ of the convex hull of $\fat(B,C)$, the fattened boundary of $C$ at $B$. The subcomplex $B$ decomposes $\parti \conv \fat(B,C)$ into two components, by the (polyhedral) Jordan--Brouwer Theorem. Consider the component $A$ that does not contain the fattened boundary $\fat(B,C)$ of $C$. The $(d-1)$-complex $A\cup \fat(B,C)\subseteq \parti \conv \fat(B,C)$ is in convex position, and by Lemma~\ref{lem:convglue}, the result $A\sqcup_{\fat(B,C)} C$ of gluing $C$ and $A$ at $B$ is a premanifold in locally convex position.

Repeating this with all boundary components yields a polytopal premanifold without boundary in locally convex position. Thus it is the boundary of a convex polytope, by Theorem~\ref{thm:heij}. Since $C$ is still a subcomplex of the boundary of the constructed convex polytope, $C$ is in convex position.
\end{proof}

\subsection{Duality, reciprocals, convex liftings and cross-bedding cubical tori}\label{ssc:altproof}

In this section, we outline an elegant alternative proof of Main Theorem~\ref{mthm:Lowdim}. The punchline is that convex position of the extension (Theorem~\ref{thm:convp}) \emph{is an automatic corollary of the existence of the extension}. To show this we make use of a relation between \emph{reciprocals} (or \emph{orthogonal duals}) and \emph{convex liftings} of polytopal complexes based on the Maxwell--Cremona correspondence \cite{Cremona2} \cite{Maxwell2}. The arguments in this section are only sketched, and there are some substantial disadvantages compared to the approach detailed in the main part of this paper (based on the Alexandrov--van Heijenoort Theorem) that we will detail on at the end.

The section has two parts: In Section~\ref{ssc:od}, we sketch the necessary notions and methods for duals, reciprocals and convex liftings. In Section~\ref{ssc:odcct}, we apply these ideas to our cross-bedding cubical tori.

\subsubsection{Duals, reciprocals and liftings to convex position}\label{ssc:od}

\enlargethispage{4mm}

The following summary of basic notions and results concerning reciprocal complexes and their convex liftings loosely follows Rybnikov~\cite{Rybnikovthesis}. For details and an intuitive explanations of the following results, we refer the reader to \cite{CW}~\cite{Aurenhammer}~\cite{Rybnikovthesis}; more detailed references are collected at the end of this section.
All polytopal manifolds in this section are manifolds with~boundary.
\begin{definition}[Duality]
Let $C$ be a polytopal $d$-manifold. A complex $D$ is \Defn{dual} to $C$ if there is an injective map $\DL:D\rightarrow C$ such that
\begin{compactitem}[$\circ$]
\item $k$-dimensional faces of $D$ map to $(d-k)$-dimensional faces of $C$,
\item $\DL$ is a bijection between $D$ and those faces of $C$ not in $\partial C$, and
\item if $\tau, \sigma$ are faces of $D$, $\tau \subset \sigma$, then $\DL(\sigma)\subset \DL(\tau)$.
\end{compactitem}
For a manifold in convex position, there is a natural dual: If $C$ is a polytopal $d$-manifold in convex position in $S^{d+1}$, denote by $(\conv C)^\ast$ the \Defn{polar dual} of the convex polytope $\conv C$, i.e.\ 
\[(\conv C)^\ast:=\{x\in S^{d+1}: \langle x,y\rangle \le 0\ \text{for all}\ y\in \conv C\}.\]
Then the \Defn{polar} $C^\ast$ to $C$ is the subcomplex of $\parti(\conv C)^\ast$ consisting of faces of $\parti(\conv C)^\ast$ corresponding to faces of $C$ not in $\partial C$.
\end{definition}

\begin{definition}[Reciprocity]
Two subspaces $V, W\subset S^d$ with $ W \cap V=\{x, -x\}$ for some $x\in S^d$ are \Defn{reciprocal} if the linear subspaces $\TT_x V$ and $\TT_x W$ are orthogonal complements in  $\TT_x S^d\cong \R^{d}$ (In particular, this implies that $\dim V+\dim W=d$).
 A dual $D\subset S^d$ to a polytopal $d$-manifold $C$ in $S^d$ is a \Defn{reciprocal} of $C$ if for every face $\sigma$ of~$D$, the subspaces $\SSp(\sigma)$ and $\SSp(\DL(\sigma))$ in $S^d$ are reciprocal.
\end{definition}

\begin{obs}\label{obs:os}
Let $V, W, V'$ and $W'$ be subspaces of $S^d$. Assume that $V$ is reciprocal to $W$, that $V'$ is reciprocal to $W'$, and that $\dim \SSp (V\cup V')= \dim V +\dim V'=d-\dim W\cap W'$. Then the subspaces $\SSp(V\cup V')$  and $W\cap W'$ are reciprocal. 
\end{obs}
Hence, checking reciprocity may be restricted to edges of~$D$:

\begin{lemma}\label{lem:red}
Let $C$ be a polytopal $d$-manifold in $S^d$, and let $D$ in $S^d$ be a dual to $C$. Then $D$ is a reciprocal of $C$ if and only if for every edge $e$ of $D$, the subspaces $\SSp(e)$ and $\SSp(\DL(e))$ are reciprocal. \emph{\qed}
\end{lemma}

If $\sigma$, $\tau$ are adjacent facets of a polytopal $d$-manifold $C$ in $S^d$, then let us denote by $\mathrm{n}^\sigma_\tau$ the normal to the face $\sigma\cap \tau$ directed towards $\sigma$, i.e.\ $\mathrm{n}^\sigma_\tau$ is the midpoint of the hemisphere that contains $\sigma$, but does not intersect $\intx \tau$. With this, a reciprocal $D$ to a polytopal manifold $C$ in $S^d$ is \Defn{orientation preserving} if and only if for every pair of vertices $a$, $b$ of $D$ that are connected by an edge, we have $\langle \mathrm{n}^{\DL(a)}_{\DL(b)}, a-b \rangle>0.$

\begin{prp}\label{prp:prj}
Let $C$ be a polytopal $d$-manifold in convex position in $S^{d+1}$. If the orthogonal projection $\pp:S^{d+1}\rightarrow S^d_\eq$ is well-defined and injective on $C$ and on its polar $C^\ast$, then $\pp(C^\ast)$ is an orientation preserving reciprocal for~$\pp(C)$. \emph{\qed}
\end{prp}

Proposition \ref{prp:prj} motivates the notion of liftings of polytopal complexes.

\begin{definition}[Liftings and convex liftings]
Let $C$ be a polytopal complex in $S^d_\eq$. A complex $\widehat{C}$ in $S^{d+1}$ is a \Defn{lifting} of $C$ if the orthogonal projection $\pp$ is well-defined and injective on $\widehat{C}$ and $\pp(\widehat{C})=C$. The complex $\widehat{C}$ is a \Defn{convex lifting}, or \Defn{lifting to convex position}, of $C$ if the lifting $\widehat{C}$ of $C$ is in convex~position.
\end{definition}

\begin{theorem}\label{thm:liftsc}
Let $C$ be a polytopal $d$-manifold in $S^d_\eq$ with $H_1(C,\mathbb{Z}_2)=0$. Then $C$ admits an orientation preserving reciprocal if and only if it admits a convex lifting to $S^{d+1}$.
\end{theorem}

For our intended application, we need a generalization applicable to manifolds with general topology:

\begin{theorem}\label{thm:liftnsc}
Let $B$ be a polytopal $d$-manifold in convex position in $S^{d+1}$ on which $\pp$ is well-defined and injective, and let $C$ be a polytopal $d$-manifold in $S^d_\eq$ so that $\pp (B)$ is a subcomplex of $C$. Assume that 
\begin{compactenum}[\rm (a)]
\item the inclusion $\pp(B)\rightarrow C$ induces a surjection $H_1(\pp(B),\mathbb{Z}_2)\rightarrow H_1(C,\mathbb{Z}_2)$,
\item $C$ admits a reciprocal $D$, and
\item the natural combinatorial isomorphism $\pp(B^\ast)\rightarrow\DL^{-1}(\pp (B))$ is geometrically realized by the identity on~$S^d_\eq$. 
\end{compactenum}
Then $C$ admits a convex lifting $\widehat{C}$ such that the subcomplex of $\widehat{C}$ that projects to $\pp(B)$ coincides with $B$.
\end{theorem}

A few words on the proof: Classically, the relation between reciprocal complexes and liftings was formulated for complexes in euclidean spaces; for the euclidean plane, a version of Theorem \ref{thm:liftsc} was noticed already by Maxwell \cite{Maxwell3, Maxwell2}. For expositions of the general case and Proposition \ref{prp:prj}, compare in particular \cite[Thm.\ 1]{Aurenhammer}. The analogue of Theorem~\ref{thm:liftnsc} for the euclidean setting follows from work of Crapo \& Whiteley \cite{CW, CrapoWhiteley}; compare in particular Theorem 2.6.3 in Rybnikov's PhD thesis \cite{Rybnikovthesis}. A few authors also treated the spherical case directly: in particular, McMullen \cite{McMullenDiag} provided a proof for the special case $C\cong S^d_\eq$ of Theorem \ref{thm:liftnsc} in the spherical setting. The translation of \cite[Thm.\ 2.6.3]{Rybnikovthesis} to the spherical case, and hence the proof of Theorems \ref{thm:liftnsc} and \ref{thm:liftsc}, is quite straightforward; we omit the details. 

\subsubsection{Reciprocals and cross-bedding cubical tori}\label{ssc:odcct}
In this section we prove that reciprocity is a property that is naturally preserved when extending CCTs. The main theorem is the following.

\begin{theorem}\label{thm:natcon}
Assume that $\CT$ and $\mathrm{S}$ are CCTs in $S^3$ such that
\begin{compactenum}[\rm(a)]
\item $\CT$ is a polytopal manifold, or equivalently, $\CT$ is of width at least $5$, 
\item $\mathrm{S}$ is an orientation preserving reciprocal for $\CT$, and
\item both $\CT$ and  $\mathrm{S}$ admit elementary extensions, say $\CT'$ and  $\mathrm{S}'$.
\end{compactenum} 
Then $\mathrm{S}'$ is an orientation preserving reciprocal for $\CT'$.
\end{theorem}

\begin{rem}
CCTs are special instances of Q-nets, which are discrete analogues of conjugate nets, cf.~\cite[Sec.\ 1 \& 2]{BobSur}. By following the proof we will sketch below, it is not hard to see that Theorem~\ref{thm:natcon} holds in an analogous form for Q-nets of dimension at least $3$. It might be interesting to further explore of the connection between Q-nets and reciprocals.
\end{rem}

\begin{proof}[Sketch of proof]
We only treat reciprocity of the extension, orientation preservation is left to the reader. By Lemma~\ref{lem:red}, we have to prove that for any $3$-cube $W$ of $\CT'$ not in $\CT$ and for any facet $A$ of $\CT$ adjacent to $W$, the subspaces $\SSp(W\cap A)$ and $\SSp \DL^{-1}(W\cap A)$ are reciprocal. For this, let $B_1$, $B_2$ denote the remaining facets of $\CT$ adjacent to $W$, and let $F_{i}$ denote the facet of $\CT$ adjacent to both $A$ and $B_i$, $i=1,2$. Moreover, we set $e_i:=W\cap A\cap B_i \cap F_i,\ i=1,2$. 

The proof is now simple: Since $\mathrm{S}$ is a reciprocal for $\CT$, the subspaces $\SSp(A\cap F_i)$ and $\SSp \DL^{-1}(A\cap F_i)$ are reciprocal, and so are the subspaces $\SSp(B_i\cap F_i)$ and $\SSp \DL^{-1}(B_i\cap F_i)$. Hence, the subspaces 
\[\SSp \left(\DL^{-1}(A\cap F_i)\cup \DL^{-1}(B_i\cap F_i)\right)= \SSp \DL^{-1}(e_i)  \quad \text{and}\quad\SSp(A\cap F_i)\cap\SSp(B_i\cap F_i) = \SSp(e_i)\]
are reciprocal by Observation~\ref{obs:os}. Finally, invoking Observation~\ref{obs:os} again shows reciprocity of the subspaces 
\[\SSp(e_1\cup e_2)=\SSp(W\cap A)\quad \text{and}\quad\SSp \DL^{-1}(e_1)\cap \SSp \DL^{-1}(e_2) = \SSp \DL^{-1}(W\cap A). \qedhere\]
\end{proof}

If we combine Theorem~\ref{thm:liftnsc},~\ref{thm:natcon} and~\ref{thm:exts}, we obtain the following theorem that can replace both Theorem~\ref{thm:ext} \emph{and} Theorem~\ref{thm:convp} for the proof of Main Theorem~\ref{mthm:Lowdim}.

\begin{theorem}\label{thm:repl}
Let $\CT$ be an ideal CCT of width $k\ge  6$ in convex position in $S^4$. Assume that its polar $\mathrm{S}=\CT^\ast$ is ideal and in convex position as well. Then there are ideal CCTs $\CT'$ and $\mathrm{S}'$, of width $k+1$ and $k-2$ respectively, such that 
\begin{compactenum}[\rm(a)]
\item both $\CT'$ and $\mathrm{S}'$ are ideal,
\item both $\CT'$ and $\mathrm{S}'$ are in convex position,  
\item $\mathrm{S}'$ is the polar dual to $\CT'$, and
\item $\RR(\CT',[0,k])=\CT$ and $\RR(\mathrm{S}',[0,k-3])=\mathrm{S}$.
\end{compactenum}
\end{theorem}

\paragraph*{Conclusion.} \hspace{-2mm} We now have two proof strategies for Main Theorem~\ref{mthm:Lowdim} that can be summarized as follows: 
\begin{compactenum}[(A)]
\item Start with an ideal $3$-CCT in convex position in $S^4$, for instance $\PS[3]$. Now, use Theorem~\ref{thm:ext} to prove that the extensions of $\PS[3]$ exist and are ideal, and use Theorem~\ref{thm:convp} to prove that these extensions are in convex position. 
\item Start with an ideal $6$-CCT in convex position in $S^4$ whose polar is ideal as well. For instance, one can verify that $\PS[6]$ is a CCT as desired. Since the polar is automatically in convex position, we can now use Theorem~\ref{thm:repl} to prove that CCTs in convex position of arbitrary width exist. 
\end{compactenum}
Approach {(B)} is arguably more intuitive and straightforward, and it avoids several tedious arguments when checking the conditions of the Alexandrov--van Heijenoort Theorem~\ref{thm:locglowib}. However, to use it we have to start with a CCT in convex position of considerable \emph{higher width} (width $6$, compared to width $3$ for approach {(A)}) and whose \emph{polar is ideal as well}. This has to be verified by hand, and is much more demanding than verifying that a $3$-CCT is ideal and in convex position. This is in particular relevant if one wants to construct CCTs based on different initial layers, as we will do in Section~\ref{sec:varieties}.

\subsection{Shephard's list}\label{ssc:Shphrdlist}

Construction methods for projectively unique $d$-polytopes
were developed by Peter McMullen in his doctoral thesis (Birmingham 1968) 
directed by G.\ C.\ Shephard; see~\cite{McMullen}, where
McMullen writes: 
\begin{quote}
 ``Shephard (private communication) has independently made a list, believed to be complete,
	  of the projectively unique $4$-polytopes. All of these polytopes can be constructed by the methods
	  described here.''
\end{quote}	  
If the conjecture is correct, the following list of eleven projectively unique $4$-polytopes
(all of them generated by McMullen's techniques, duplicates removed) should be complete:

\begin{table}[htbf]
\centering
{\small
\begin{tabular}{|l||l|l|l|c|l|l|}
	  \hline
  & Construction	& dual&	type		& $(f_0,f_1,f_2,f_3)$  & facets\\
	  \hline
	  \hline
$P_1$ &  $\Delta_4$     & $P_1$     & simplicial & (5,10,10,5) & 5\,tetrahedra					\\
$P_2$ &  $\Box * \Delta_1$ & $P_2$ & & (6,11,11,6) & 4\,tetrahedra, 2\,square\,pyramids			\\
$P_3$ &  $(\Delta_2\oplus\Delta_1)*\Delta_0$  & $P_4$  &    &  (6,14,15,7) & 6\,tetrahedra, 1\,bipyramid 	\\
$P_4$ &  $(\Delta_2\times\Delta_1) * \Delta_0$ & $P_3$  &  & (7,15,14,6) & 2\,tetrahedra, 3\,square\,pyramids, 1\,prism  \\
$P_5$  &  $\Delta_3\oplus\Delta_1$ & $P_6$ & {simplicial} & (6,14,16,8)  & 8\,tetrahedra			\\
$P_6$  &  $\Delta_3\times\Delta_1$ & $P_5$ & {simple} & (8,16,14,6)  & 2\,tetrahedra, 4\,prisms		\\
$P_7$  &  $\Delta_2\oplus\Delta_2$ & $P_8$ & {simplicial} & (6,15,18,9)  & 9\,tetrahedra			\\
 $P_8$ &  $\Delta_2\times\Delta_2$ & $P_7$ & {simple} & (9,18,15,6) & 6\,prisms				\\
$P_9$  &  $(\Box,v)\oplus(\Box,v)$ & $P_{10}$ & &(7,17,18,8)   & 4\,square\,pyramids, 4\,tetrahedra\\
$P_{10}$  &  & $P_9$     &  & (8,18,17,7)  & 2\,prisms, 4\,square\,pyramids, 1\,tetrahedron\\
$P_{11}$ & ${\rm v.split}(\Delta_2\times\Delta_1)$  & $P_{11}$  & & (7,17,17,7)  &  3\,tetrahedra, 2\,square\,pyramids, 2\,bipyramids\\
  \hline
  \end{tabular}
\caption{Shephard's list of $4$-dimensional projectively unique polytopes.}\label{tab:shlist}
}
\end{table}

\subsection{Iterative construction of CCTs}\label{ssc:expformula}

The main results of this paper were based on an iterative construction of ideal CCTs (Section~\ref{sec:bblocks}). It is natural to ask whether one can provide explicit formulas for this iteration, and indeed, a first attempt to prove Theorem~\ref{mthm:Lowdim} and Theorem~\ref{mthm:projun} would try to understand these iterations in terms of explicit formulas. Since the building block of our construction is Lemma~\ref{lem:cubecmpl}, this amounts to understanding the following problem.

\begin{problemm}\label{prb:iteration}
 Let $Q_1$, $Q_2$, $Q_3$ be three quadrilaterals in some euclidean space (or in some sphere) on vertices $\{a_1,\, a_2,\, a_3,\, a_4\}$, $\{a_1,\, a_4,\, a_5,\, a_6\}$ and $\{a_1,\, a_2,\, a_7,\, a_6\}$, respectively, such that the quadrilaterals do not lie in a common $2$-plane. Give a formula for $a_1$ in terms of the coordinates of the vertices $a_i$, $i\in \{2,\, \dots,\, 7\}$.
\end{problemm}

It is known and not hard to see that this formula is rational~\cite[Sec.\ 2.1]{BobSur}. The formula is, however, rather complicated, so that it is much easier to follow an implicit approach for the iterative construction of ideal CCTs. 

\medskip

In this section we nevertheless give, without proof, an explicit formula (Formula~\ref{fml:explicit}) to compute, given an ideal $1$-CCT $\CT$ in $S^4_+$, its elementary extension $\CT'$ by solving Problem~\ref{prb:iteration} in $S^4_+$ for cases with a certain inherent symmetry coming from the symmetry of ideal CCTs. More accurately, we provide a rational formula for a map $\mathrm{i}$ that, given two special vertices $a,\, b$ of $\CT$ in layers $0$ and $1$ respectively, obtains a vertex $c:=\mathrm{i}(a,b)$ of layer $2$ of $\CT'$. The map $\mathrm{i}$ is chosen in such a way that we can easily iterate it, i.e.\ in order to obtain a vertex $d$ of the elementary extension $\CT''$ of $\CT'$, we simply compute $\mathrm{i}(b,\rot^{2}_{1,2} c)$ (cf.\ Proposition~\ref{prp:it}).\footnote{Explicit calculations of this and the following section were performed using \href{http://www.sagemath.org/}{\textbf{sagemath}}, Ver. 5.10. }

\begin{rem}
A word of caution: Formula~\ref{fml:explicit} for $\mathrm{i}(a,b)$ is also well-defined for some values of $a,\, b$ for which the extension $\CT'$ of $\CT$ does not exist. In particular, one should be careful not to interpret the well-definedness of $\mathrm{i}(a,b)$ as a direct proof of Theorem~\ref{thm:ext}, rather the opposite: Theorem~\ref{thm:ext} proves that the extensions of ideal CCTs exist, which allows us, if we are so inclined, to use the explicit formula for $\mathrm{i}$ to compute them. For the rest of this section we will simply ignore this problem; we shall assume the extension exists whenever we speak of an extension of a CCT.
\end{rem}

\paragraph*{Explicit formula for the iteration:} To define $\mathrm{i}$, choose vertices $a\in \RR(\CT,0)$ and $b\in \RR(\CT,1)$ of the ideal $1$-CCT $\CT$ as in Figure~\ref{fig:setupit}, and, to simplify the formula, such that $\langle a,e_4 \rangle$ and $\langle b,e_4 \rangle$ vanish.

\begin{figure}[htbf]
\centering 
  \includegraphics[width=0.6\linewidth]{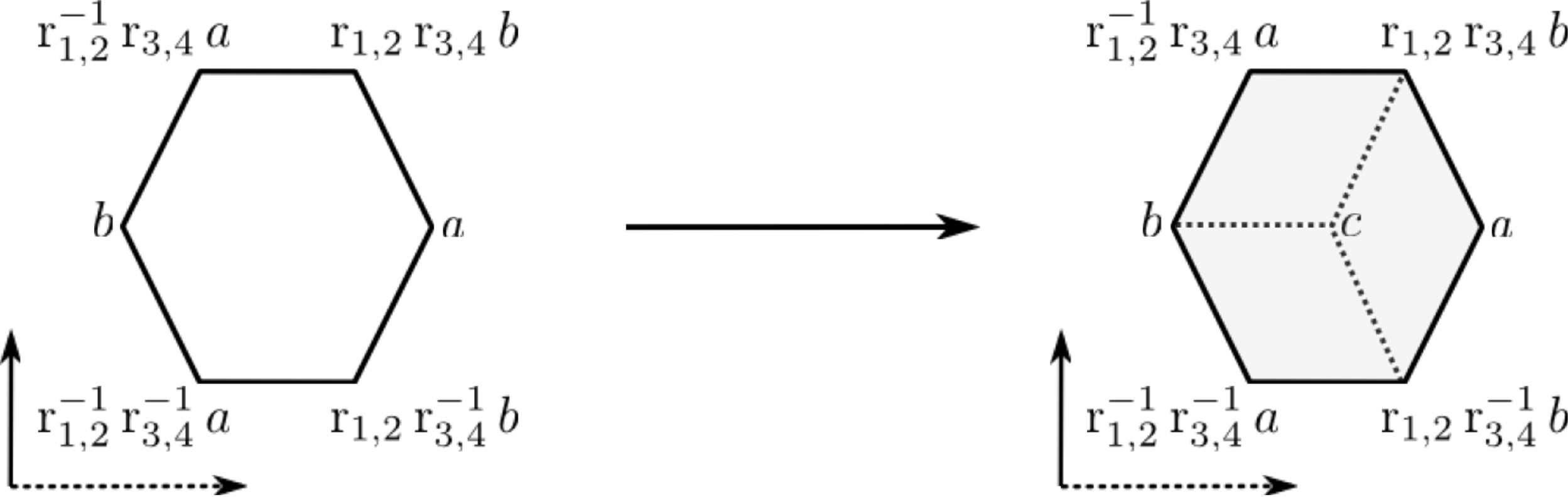} 
  \caption{\small Set-up for an explicit iterative formula; the boundary vertices and edges of the disk are in $\CT$, the interior faces (those containing $c$) are added in the extension to $\CT'$.} 
  \label{fig:setupit}
\end{figure}
We are going to give the formula for the vertex $c:=\mathrm{i}(a,b) \in \RR(\CT',2)$ as indicated in Figure~\ref{fig:setupit}. Then, it is easy to compute all extensions of an ideal CCTs in~$S^4$ explicitly using iterations of $\mathrm{i}$.

\begin{prp}\label{prp:it}
Let $\CT$ be an ideal $1$-CCT in $S^4_+$, and assume that $a\in \RR(\CT,0)$ and $b\in \RR(\CT,1)$ are chosen as before. Let us denote by $\CT^{[k]},\ k\ge 0$, the $k$-th layer of the $k$-CCT extending $\CT$. Set $\kappa_0:=a$, $\kappa_1:=b$ and define
\[\kappa_{k+1}:=\rot^{2}_{1,2}\mathrm{i}(\kappa_{k-1},\kappa_{k}).\]
Then $\kappa_{k}\in \CT^{[k]}$ for all $k$. \emph{\qed}
\end{prp}

\begin{fml}\label{fml:explicit} Consider an ideal $1$-CCT $\CT$ in $S^4_+$. We use homogeneous coordinates for the vertices of~$\CT$. Then, we have the desired formula for $\mathrm{i}$:
\[\mathrm{i}(a,b)=\mu(a,b)a+(1-\mu(a,b)) \frac{\rot_{1,2}\rot_{3,4} b + \rot_{1,2}\rot_{3,4}^{-1} b}{2},\] 
where the parameter $\mu(a,b)$ is given by
\[\mu(a,b)=\frac{\big(\mathrm{S}+\mathrm{L}_{\hat{b}}+\mathrm{D}\big)\big(3\mathrm{S}+\mathrm{L}_{\hat{b}}+\mathrm{D}\big)}{\mathrm{L}_{\hat{b}}^2+4\mathrm{L}_{\hat{a}}\mathrm{S}+2\mathrm{L}_{\hat{b}}\mathrm{D} + 2\mathrm{S}^2+4\mathrm{D}\mathrm{S}}\ \ \text{and therefore}\ \ 
1-\mu(a,b)=\frac{4(\mathrm{L}_{\hat{a}}-\mathrm{L}_{\hat{b}})\mathrm{S}-\mathrm{S}^2-\mathrm{D}^2}{\mathrm{L}_{\hat{b}}^2+4\mathrm{L}_{\hat{a}}\mathrm{S}+2\mathrm{L}_{\hat{b}}\mathrm{D} + 2\mathrm{S}^2+4\mathrm{D}\mathrm{S}}.\]
Here $\mathrm{S}= \langle\hat{a},\hat{b}\rangle$ denotes the scalar product of the vectors $\hat{a}=(a_1,a_2)$ and $\hat{b}=(b_1,b_2)$, while $\mathrm{L}_{\hat{a}}=||\hat{a}||_2^2=a_1^2+a_2^2$ and $\mathrm{L}_{\hat{b}}=||\hat{b}||_2^2=b_1^2+b_2^2$ denote their respective square lengths. Finally, $\mathrm{D}$ is the signed volume of the parallelepiped spanned by $\hat{a}$ and $\hat{b}$, i.e.,\ $\mathrm{D}$ is the determinant of the matrix $\left( \begin{smallmatrix}
a_1 & b_1\\
a_2 & b_2\end{smallmatrix}\right)$.\qed
\end{fml}

\begin{rem}\label{rem:observations}
One can conclude several interesting facts about our construction from Formula~\ref{fml:explicit}:
\begin{compactenum}[(a)]
\item It follows directly from the existence of the extension that the term $\mu(a,b)$ is negative if the CCT it is applied to is ideal.
\item We have $\langle\hat{b},\hat{c}\rangle=\mu(a,b)\langle\hat{a},\hat{b}\rangle$, where $\hat{c}=(c_1,c_2)$. Hence, the term $\mathrm{S}$ in Formula~\ref{fml:explicit} converges to $0$ at an exponential rate as ideal CCTs are extended, and does not change sign since $\mu(a,b)<0$.
\item If we consider iterative extensions of ideal CCTs in convex position, then as the construction progresses, the square length of vectors in the newest layer converges (by a simple compactness argument). Hence,  in the notation of Formula~\ref{fml:explicit}, $\mathrm{L}_{\hat{a}}-\mathrm{L}_{\hat{b}}$ tends to $0$.
\item Combining the previous two observation, we see that $\mu(a,b)$ tends to $0$ as the extension progresses.
\item We can conclude from this that as we iteratively build an ideal CCT, the squared norm of the last two coordinates of the vertices in the highest layer $i$ can be bounded above by a constant multiple~of~$(\nicefrac{\sqrt{3}}{{2}})^{i}$.
\end{compactenum}
\end{rem}

\begin{example}\label{ex:expc}
If $\PS[1]$, as given in Section~\ref{ssc:example}, is our starting CCT for the proof of Theorem~\ref{mthm:Lowdim}, then we can choose \[a=(\sqrt{2}-1,\,1-\sqrt{2},\, 2,\, 0,\, 1)= \vartheta_0\in \RR(\PS[1],0)\quad \text{and}\quad b=(-1,\,0,\, 1,\, 0,\, 1)=\rot^2_{1,2}\vartheta_1\in \RR(\PS[1],1).\]  Then, $\mu(a,b)=\tfrac{1}{23}(3-4\sqrt{2})$ and \[\mathrm{i}(a,b)= \big(\tfrac{1}{23}(-11+7\sqrt{2}),\,\tfrac{1}{23}(-9-11\sqrt{2}),\, \tfrac{1}{23}(16-6\sqrt{2}),\, 0 ,\, 1\big)=\vartheta_2\in \RR(\PS[2],2).\]
More generally, by setting $\kappa_0:=\vartheta_0$, $\kappa_1:=\rot^2_{1,2}\vartheta_1$, we can use the iteration procedure of Proposition~\ref{prp:it} to inductively construct the complexes $\PS[n]$ (Table~\ref{tab:list1}). We compute $\kappa_i$ explicitly for $i\le  10$; the fourth coordinate is always $0$ and the fifth is always $1$, so we omit them from the list. Furthermore we compute the value $\lambda(\kappa_i)$ such that $\pp(\kappa_i)$ lies in $\mathcal{C}_{\lambda(\kappa_i)}$. We constructed the complexes $\PS[n]$ towards $\mathcal{C}_0$, so these values should decrease. In fact, by Remark \ref{rem:observations}(c) and (e), we obtain that $\lambda(\kappa_i)=O \big((\nicefrac{3}{4})^{i}\big)\xrightarrow{i\rightarrow \infty} 0$.
\renewcommand{\arraystretch}{1.2}

\begin{table}[h!tbf]
\centering
$\begin{array}{|c||c|c|c|c|}
\hline
\text{Vertex}&\text{First coordinate} & \text{Second coordinate} &\text{Third coordinate} & \lambda(\kappa_i)\ \text{(float)} \\
\hline
\hline
\kappa_0 & \sqrt{2}-1&-\sqrt{2}+1 & 2&1.8419\\
\hline
\kappa_1 & -1 & 0& 1&1\\
\hline
\kappa_2 & \frac{-7\sqrt{2}+11}{23} &\frac{11\sqrt{2}+9}{23}& \frac{16-6\sqrt{2}}{23}&0.1709\\ 
\hline
\kappa_3 &\frac{11\sqrt{2}+37}{49} &\frac{6\sqrt{2}-11}{49}& \frac{-12\sqrt{2} + 22}{49}&0.0181\\
\hline
\kappa_4 &\frac{145\sqrt{2} - 241}{697} &\frac{-241\sqrt{2} - 407}{697} &\frac{-168\sqrt{2} + 260}{697}&  1.7906{\cdot}10^{-3} \\ 
\hline
\kappa_5 &\frac{ -192\sqrt{2} - 457 }{679}& \frac{ -111\sqrt{2} + 192 }{679} &\frac{ -138\sqrt{2} + 202}{679}  & 1.7580{\cdot}10^{-4}\\
\hline
\kappa_6 & \frac{-341\sqrt{2} + 577}{1837}& \frac{577\sqrt{2} + 1155}{1837} &  \frac{-324\sqrt{2} + 464}{1837}&1.7247{\cdot}10^{-5}\\
\hline
\kappa_7 &\frac{11471\sqrt{2} + 25057}{38473}& \frac{6708\sqrt{2} - 11471}{38473} & \frac{ -5712\sqrt{2} + 8116}{38473}& 1.6920{\cdot}10^{-6}\\ 
\hline
\kappa_8 & \frac{137\sqrt{2} - 233}{761}&\frac{ -233\sqrt{2} - 487}{761} & \frac{ -96\sqrt{2} + 136}{761}& 1.6598{\cdot}10^{-7}\\
\hline
\kappa_9 & \frac{-165588\sqrt{2} - 353893}{548089}&\frac{-97098\sqrt{2} + 165588}{548089} & \frac{-58344\sqrt{2} + 82564}{548089}&1.6283{\cdot}10^{-8}\\ \hline
\kappa_{10} & \frac{-2955751\sqrt{2} + 5033675}{16549127}&\frac{5033675\sqrt{2} + 10637625}{16549127}&   \frac{-1490520\sqrt{2} + 2108416}{16549127}&1.5974{\cdot}10^{-9}\\
\hline \end{array}$
\caption{Coordinates for the polytopes $\mathrm{CCTP}_4[n]$.}
\label{tab:list1}
\end{table}

\end{example}

\subsection{Many more projectively unique polytopes}\label{sec:varieties}
In this section we construct, for \emph{any finite field extension} $F\subset \R$ over $\mathbb{Q}$, infinitely many projectively unique polytopes in fixed dimension that are characteristic to $F$ (Section~\ref{ssc:rat}). Moreover, we construct infinitely many \emph{inscribed} projectively unique polytopes in fixed dimension (Section~\ref{ssc:ins}). 

Compared to Theorem \ref{mthm:projun}, we shall not use explicit construction methods but rather rely on general results of \cite{AP} to obtain projectively unique polytopes from polytopes with low-dimensional realization space. Hence, the ``fixed dimension'' of the projetively unique polytopes constructed here is only implicit (although in principle computable) and in general much larger than $69$.

We work with polytopes and point configurations in $S^d$; coordinates of points and vertices in the upper hemisphere $S^d_+\subset S^d$ referred to are always homogeneous coordinates.

\subsubsection{Many projectively unique polytopes over any field}\label{ssc:rat}
Perles not only constructed exponentially many projectively unique polytopes~\cite{McMullen}~\cite{PerlesShephard}, he also famously established the existence of a projectively unique polytope not realizable in any rational vector space~\cite[Sec.\ 5.5, Thm.\ 4]{Grunbaum}. In~\cite{AP}, the first named author and Padrol remark that his results extend to any finite field extension: For any finite field extension $F\subset \R$ over $\mathbb{Q}$, there is a projectively unique polytope $P$ that is realizable in a vector space over $F$, but not in any strict subfield $G$ of $F$. In this section, we prove that there is not only one such polytope, but there even are infinitely many in some fixed dimension that depends only on $F$. The main result is the following analogue of Main Theorem \ref{mthm:projun}.

\begin{theorem}\label{thm:anyfield}
Let $F\subset \R$ be any finite field extension over $\mathbb{Q}$. For any $d \ge D=D(F)$, there is an infinite family of projectively unique $d$-dimensional polytopes $\mathrm{PCCTP}^F_d[n]\subset S^d_+,\ n\ge 1,$ on $12(n+1)+d+D(F)-9$ vertices with coordinates in $F$, but not realizable (in $S^d_+$) with coordinates in $G$, where $G$ is any strict subfield of $F$.
\end{theorem}

A point configuration or polytopal complex is \Defn{rational} if the coordinates of all of its points (resp.\ its vertices) are rational numbers. 
\begin{cor}
 There is a $D=D(\mathbb{Q})$ such that for any $d \ge D$, there is an infinite family of \emph{rational} projectively unique $d$-dimensional polytopes $\mathrm{PCCTP}^{\mathbb{Q}}_d[n],\ n\ge 1,$ on $12(n+1)+d+D-9$ vertices.
\end{cor}

For the proof, let us first recall a fundamental result going back to von Staudt~\cite{Staudt}.

\begin{prp}[cf.\ {\cite[Cor.\ U.17]{AP}}]\label{prp:AP}
Let $Q$ denote any point configuration in $S^d_+\subset S^d$ whose elements are described by algebraic coordinates. Then there is a projectively unique point configuration $\mathrm{COOR}[Q]$ in $S^d$ that contains $Q$.
\end{prp}

\begin{cor}\label{cor:AP}
For $F\subset \R$ any finite field extension over $\mathbb{Q}$, and for any rational point configuration $Q$ in $S^d_+\subset S^d$, $d\ge 3$, there exists a projectively unique point configuration $\mathrm{COOR}^F[Q]\in S^d_+$ that contains $Q$ and such that $\mathrm{COOR}^F[Q]$ is realizable with coordinates in $F$, but not realizable with coordinates in $G$, where $G$ is any strict subfield of $F$.
\end{cor}

\smallskip

\begin{proof}[\textbf{Proof of Theorem~\ref{thm:anyfield}}] As for Theorem~\ref{mthm:projun}, the proof consists of two parts: We first give a family of rational polytopes $\mathrm{CCTP}^{\mathbb{Q}}_4[n]$ for which $\dim \cR (\cdot)$ is uniformly bounded, and then produce projectively unique polytopes from them.
\paragraph*{Construction of the polytopes $\mathrm{CCTP}^{\mathbb{Q}}_4[n]$:}
We start by constructing CCTs analogous to those given in Section~\ref{ssc:example}. Mirroring the notation of that section, let us replace $\vartheta_0$ and $\vartheta_1$ by
\[\vartheta^\mathbb{Q}_0:=\left(\tfrac{1}{3}, -\tfrac{1}{3}, 2, 0, 1\right)\quad \text{and} \quad \vartheta^\mathbb{Q}_1:=\left(1, 0, \tfrac{3}{5}, 0, 1\right),\]
respectively. With this, we obtain $\CT^\mathbb{Q}[1]$, the starting CCT for our construction. We note, without proof, some facts about this complex:
\begin{compactitem}[$\circ$]
\item The complex $\CT^\mathbb{Q}[1]$ can be extended twice to $\CT^\mathbb{Q}[3]$, which is ideal. Hence, Theorem~\ref{thm:ext} shows that $\CT^\mathbb{Q}[1]$ can be extended to $\CT^\mathbb{Q}[n]$ for any $n\ge 1$.
\item The complex $\CT^\mathbb{Q}[3]$ is in convex position. Hence, the complexes $\CT^\mathbb{Q}[n],\ n\ge 1,$ are in convex position by Theorem~\ref{thm:convp}.
\item Due to symmetry of ideal CCTs, $\CT^\mathbb{Q}[1]$ cannot be rational. However, it is linearly equivalent to a complex $\varTheta(\CT^\mathbb{Q}[1])$ with rational vertex coordinates. For instance, $\varTheta$ can be chosen as the linear transformation given by the diagonal matrix with entries $(1,1,1,\sqrt{3},1)$.
\item As mentioned in Section \ref{ssc:expformula}, the extensions of CCTs are given by a rational functions. Moreover, extension clearly commutes with any projective transformation. Any extension of the rational complex $\varTheta(\CT^\mathbb{Q}[1])$ is therefore rational.
\end{compactitem}
Similar to Example~\ref{ex:stdet}, we give the coordinates of $\CT^\mathbb{Q}[n]$ up to layer $10$ in Table \ref{tab:list2} below. Again, we set $\kappa_0:=\vartheta^\mathbb{Q}_0$, $\kappa_1:=\rot^2_{1,2}\vartheta^\mathbb{Q}_1$, and use Proposition~\ref{prp:it} and Formula~\ref{fml:explicit} to inductively generate the complexes $\CT^\mathbb{Q}[n]$. We omit the fourth and fifth coordinate, since these are always $0$ and $1$, respectively.

\renewcommand{\arraystretch}{1.2}
\begin{table}[htbf]
\centering
$\begin{array}{|c||c|c|c|c|}
\hline
\text{Vertex}&\text{First coordinate} & \text{Second coordinate} &\text{Third coordinate} & \lambda(\kappa_i)\ \text{(float)} \\
\hline
\hline
\kappa_0 & \frac{1}{3}&\frac{\mbox{-}1}{3} & 2&1.8947\\
\hline
\kappa_1 & -1 & 0& \frac{3}{5} &0.5294\\
\hline
\kappa_2 & \frac{1}{25} &\frac{27}{25}& \frac{12}{125}&0.0157\\ 
\hline
\kappa_3 &\frac{179}{165} &\frac{\mbox{-}7}{165}& \frac{4}{275}& 3.5892 \cdot 10^{-04}\\
\hline
\kappa_4 &\frac{\mbox{-}93}{2185} &\frac{\mbox{-}2371}{2185}&\frac{24}{10925}&  8.1842 \cdot 10^{-06} \\ 
\hline
\kappa_5 &\frac{\mbox{-}7851}{7235}& \frac{308}{7235}&\frac{12}{36175} & 1.8660 \cdot 10^{-07}\\
\hline
\kappa_6 & \frac{8159}{191655}& \frac{207973}{191655} & \frac{16}{319425}&4.2549 \cdot 10^{-09}\\
\hline
\kappa_7 &\frac{1377301}{1269235}& \frac{\mbox{-}54033}{1269235} & \frac{48}{6346175}
& 9.7016 \cdot 10^{-11}\\ 
\hline
\kappa_8 & \frac{\mbox{-}715667}{16811015}&\frac{\mbox{-}18242349}{16811015} & \frac{96}{84055075}
& 2.2121 \cdot 10^{-12}\\
\hline
\kappa_9 & \frac{\mbox{-}60404969}{55665465}&\frac{2369752}{55665465} & \frac{16}{92775775}
&5.0438 \cdot 10^{-14}\\ \hline
\kappa_{10} & \frac{62774721}{1474577945} &\frac{1600127387}{1474577945}&   \frac{192}{7372889725}&1.1501 \cdot 10^{-15}\\
\hline \end{array}$
\caption{Coordinates for the polytopes $\mathrm{CCTP}^{\mathbb{Q}}_4[n]$.}\label{tab:list2}

\end{table}

As in the case of the original cross-bedding cubical torus polytopes, we define the desired polytopes $\mathrm{CCTP}^{\mathbb{Q}}_4[n]$ as convex hulls of the CCTs constructed: \[\mathrm{CCTP}^{\mathbb{Q}}_4[n]:=\conv(\varTheta(\CT^{\mathbb{Q}}_4[n])).\]
With this, we have \[f_0(\mathrm{CCTP}^{\mathbb{Q}}_4[n])=f_0(\CT^{\mathbb{Q}}_4[n])=12(n+1),\]
and hence, as in the proof of Theorem~\ref{mthm:Lowdim}, \[\dim\cR (\mathrm{CCTP}^{\mathbb{Q}}_4[n])\le 4f_0(\CT^{\mathbb{Q}}[1])= 96.\]

\paragraph*{Construction of the polytopes $\mathrm{PCCTP}^{F}_4[n]$:}
Set $Q:=\varTheta(\F_0(\CT^{\mathbb{Q}}[1]))=\F_0(\mathrm{CCTP}^{\mathbb{Q}}_4[1])$. Let $V$ be a set of points in $S^4_+$ that span a hyperplane not intersecting any of the polytopes $\mathrm{CCTP}^{\mathbb{Q}}_4[n]$, similar to part VII.\ in the proof of Lemma~\ref{lem:affine}. Apply Corollary~\ref{cor:AP} to find a projectively unique point configuration $\mathrm{COOR}^F[Q\cup V]$ that contains $Q$ and $V$. By Lemma~\ref{lem:uniextem}, $\varTheta(\F_0(\CT^{\mathbb{Q}}[1]))=\F_0(\mathrm{CCTP}^{\mathbb{Q}}_4[1])$ frames $\mathrm{CCTP}^{\mathbb{Q}}_4[n]$ for all $n\ge 1$. Hence \[(\mathrm{CCTP}^{\mathbb{Q}}_4[n],Q,R),\ \ n\ge 1,\ \ R:=\mathrm{COOR}^F[Q\cup V]\setminus Q,\] is a weak projective triple. Realizability in vector spaces over fields is not affected by the operations subdirect cone (Definition~\ref{def:subd} and Lemma~\ref{lem:subdc}) and Lawrence extension (Proposition~\ref{prp:mlwextn}), hence applying both yields the desired polytopes $\mathrm{PCCTP}^{F}_{D(F)}[n]$. These polytopes have dimension ${D(F)}=f_0(Q)+f_0(R)+5=f_0(R)+29$ and the number of vertices computes to
\[f_0(\mathrm{PCCTP}^{F}_{D(F)}[n])=2(f_0(Q)+f_0(R))+f_0\left(\mathrm{CCTP}^{F}_4[n]\right)+1=12(n+1)+2D(F)-9. \qedhere\]
\end{proof}

\subsubsection{Many inscribed projectively unique polytopes}\label{ssc:ins}
It is a classical and elementary fact that if $W$ is a $3$-cube and $S$ is a sphere in $\R^d$ such that seven vertices of $W$ lie in $S$, then all vertices of $W$ lie in $S$ \cite[Rec.~2]{Miquel}. In fact, if $W\in \R^d$ is a $3$-cube and $\mathcal{Q}$ is a quadric that contains seven vertices of ${W}$, then the last vertex of $W$ lies on $\mathcal{Q}$ as well, cf.~\cite[Sec.\ 3.2]{BobSur}. 
As a consequence, we obtain the following beautiful result:

\begin{prp}\label{prp:quadric}
Let $\CT[2]$ denote a CCT in $S^d$ or $\R^d$, and let $S$ be a sphere in $S^d$ resp.\ $\R^d$ that contains $\F_0(\CT[2])$. Then for all extensions $\CT[n]$ of $\CT[2]$, we have $\F_0(\CT[n])\subset S$. 
\end{prp}

Here, a \Defn{sphere} in $S^d$ is the intersection of $S^d$ with some affine subspace of $\R^{d+1}$. This opens the door to the use of cross-bedding cubical tori in the theory of inscribed polytopes. A polytope in $S^d$ or $\R^d$ is \Defn{inscribed} if all its vertices are contained in some sphere. Combinatorial types of polytopes that can be realized in an inscribed way are \Defn{inscribable}.

Inscribable polytopes are a classical and intriguing subject in polytope theory. Perhaps overly optimistic, Steiner~\cite{Steiner} asked in 1832 for a classification of inscribable polytopes. For a long time, it was not even known whether all combinatorial types of polytopes are inscribable, until Steinitz provided an example of a non-inscribable polytope~\cite{Steinitz}, cf.\ \cite[Sec.\ 13.5]{Grunbaum}. Much later, interest in inscribed $3$-polytopes experienced a revival due to their importance in the theory of hyperbolic $3$-manifolds \cite{Thurston} and Delaunay triangulations \cite{Brown}. Conversely, the connection to hyperbolic geometry led to an almost complete understanding of inscribable polytopes of dimension~$3$~\cite{Rivin, Rivin2}. 
Many problems concerning inscribed polytopes remain; for some recent progress, compare~\cite{Gonska}, \cite{AP2}.

In this section, we present some progress in the direction of understanding \emph{high-dimensional} inscribable polytopes by proving the following analogues of Theorems~\ref{mthm:Lowdim} and \ref{mthm:projun} for inscribed polytopes.
\begin{theorem}\label{thm:lowdimi}
For each $d\ge  4$, there exists an infinite family of combinatorially distinct inscribed $d$-dimensional polytopes $\mathrm{CCTP}^{\hspace{.04em} \mathrm{in}}_d[n]$ with $12(n+1)+d-4$ vertices such that $\dim \cR (\mathrm{CCTP}^{\hspace{.04em} \mathrm{in}}_d[n])\le 76+d(d+1)$ for all $n\ge  1$.
\end{theorem}

\begin{theorem}\label{thm:projuni}
There is a $D\ge 0$ such that for each $d\ge  D$, there exists an infinite family of combinatorially distinct inscribed $d$-polytopes $\mathrm{PCCTP}^{\hspace{.04em} \mathrm{in}}_{d} [n],\, n\ge 1,$ with $12(n+1)+d+D-9$ vertices, all of which are projectively unique.
\end{theorem}

By polar duality, analogous results hold for \Defn{circumscribed} polytopes (i.e.\ polytopes all whose facets are tangent to some sphere).

\paragraph*{Many inscribed $4$-polytopes with small realization space}
\enlargethispage{5mm}
\begin{proof}[\textbf{Proof of Theorem \ref{thm:lowdimi}}]
By Proposition~\ref{prp:quadric}, it suffices to provide an ideal CCT $\CT^{\hspace{.04em} \mathrm{in}}[3]$ in convex position in $S^4$ whose vertices lie in some sphere in $S^4$. We do this analogously to Section~\ref{ssc:example}. Parallel to the notation of that section, let us replace $\vartheta_0$ and $\vartheta_1$ by
\[\vartheta^{\hspace{.04em} \mathrm{in}}_0:=(1, -1, y, 0, 1)\quad \text{and} \quad \vartheta^{\hspace{.04em} \mathrm{in}}_1:=(\nicefrac{11}{10}, x, z, 0, 1),\]
respectively. 
Here, $x$ is defined as \begin{align*}
x&=\frac{{{\left(2 i  \sqrt{3} - 2\right)} {\left(-45 i 
\sqrt{566805} + 83895\right)}^{\frac{2}{3}} 60^{\frac{2}{3}} - 263(i\sqrt{3}+1)60^{\frac{4}{3}} - 1560  {\left(-45 i  \sqrt{566805} +
83895\right)}^{\frac{1}{3}}} }{3600  {\left(-45 i 
\sqrt{566805} + 83895\right)}^{\frac{1}{3}}} \\
&=\frac{2^{\frac{5}{6}}}{60}  {\left(2^{\frac{2}{3}} \sqrt{789}
\sin\left(\frac{1}{3}  \arctan\left(\frac{3\sqrt{566805}}{5593}  \right)\right) -
2^{\frac{2}{3}} \sqrt{263} \cos\left(\frac{1}{3} 
\arctan\left(\frac{3\sqrt{566805}}{5593}  \right)\right) - 13  2^{\frac{1}{6}}\right)}
\end{align*}
and we obtain $y$ as
\[y = \frac{ 24 -20  x}{\sqrt{559-400x- 100  x^{2}}}\]
and analogously 
\begin{align*}
z=\frac{319-100   x^{2} - 200   x}{10\sqrt{559-400x- 100  x^{2}}} 
=\frac{{\left(100  x^{2} +
200  x - 319\right)} y}{40(5x-6)}.
\end{align*}
This gives $\CT^{\hspace{.04em} \mathrm{in}}[1]$, and $\CT^{\hspace{.04em} \mathrm{in}}[3]$ is obtained by extending on $\CT^{\hspace{.04em} \mathrm{in}}[1]$ twice. The coordinates of $\vartheta^{\hspace{.04em} \mathrm{in}}_2$ and $\vartheta^{\hspace{.04em} \mathrm{in}}_3$ are too complicated to give them here directly; indeed, it is not even trivial to see that the coordinates for $\vartheta^{\hspace{.04em} \mathrm{in}}_0$ and $\vartheta^{\hspace{.04em} \mathrm{in}}_1$, as given above, are real numbers. Their approximate values are given as
\[x\approx 1.0226363,\qquad y\approx0.5266533,\quad \text{and}\quad z\approx 0.1468968.\]
The CCT $\CT^{\hspace{.04em} \mathrm{in}}[3]$ is an ideal CCT in convex position; in homogeneous coordinates, its vertices lie on a sphere with radius $\approx 1.8103$. Theorems~\ref{thm:ext} and~\ref{thm:convp} demonstrate that $\CT^{\hspace{.04em} \mathrm{in}}[3]$ can be extended to ideal CCTs $\CT^{\hspace{.04em} \mathrm{in}}[n]$ in convex position in $S^4$ of arbitrary width. Analogous to  Example~\ref{ex:stdet} and Section~\ref{ssc:rat}, we give their coordinates in Table~\ref{tab:list3}, this time only approximate and up to layer five, when exact calculations became infeasible. Again, we set $\kappa_0:=\vartheta^{\hspace{.04em} \mathrm{in}}_0$, $\kappa_1:=\rot^2_{1,2}\vartheta^{\hspace{.04em} \mathrm{in}}_1$, and inductively construct the complexes $\CT^{\hspace{.04em} \mathrm{in}}[n]$. As before, we leave out the fourth and fifth coordinates.

\renewcommand{\arraystretch}{1.2}
\begin{table}[h!tbf]
\centering
$\begin{array}{|c||c|c|c|c|c|}
\hline
\text{Vertex}&\text{First coordinate} & \text{Second coordinate} &\text{Third coordinate} & \lambda(\kappa_i) & ||\kappa_i||_2 \\
\hline
\hline
\kappa_0 & 1&
-1&
0.5266533&
0.2435&
1.8103\\
\hline
\kappa_1 & -1.1&
-1.0226363&
0.1468968&
0.0189&
1.8103\\
\hline
\kappa_2 & -1.0243331&
1.1074958&
0.0394770&
1.3686 \cdot 10^{-03}&
1.8103\\ 
\hline
\kappa_3 &1.1080357&
1.0244553&
0.0105801&
9.8305 \cdot 10^{-05}&
1.8103\\
\hline
\kappa_4 &1.0244641&
-1.1080745&
0.0028350&
7.0582 \cdot 10^{-06}&
1.8103\\ 
\hline
\kappa_5 &-1.1080770&
-1.0244648&
0.0007596&
5.0675 \cdot 10^{-07}&
1.8103\\
\hline
\end{array}$
\caption{Coordinates for the polytopes $\mathrm{CCTP}^{\hspace{.04em} \mathrm{in}}_4[n]$.}\label{tab:list3}
\end{table}

As usual, we define
\[ \mathrm{CCTP}^{\hspace{.04em} \mathrm{in}}_4[n]:=\conv\, \CT^{\hspace{.04em} \mathrm{in}}[n].\]
These polytopes are inscribed for all $n\ge 1$ because the vertices of $\CT^{\hspace{.04em} \mathrm{in}}[n]$ lie in a common sphere by Proposition~\ref{prp:quadric}.
Since \[f_0(\CT^{\hspace{.04em} \mathrm{in}}[n])=f_0(\mathrm{CCTP}^{\hspace{.04em} \mathrm{in}}_4[n])=12(n+1),\] Lemma~\ref{lem:uniextem} permits us to conclude the desired bound
\[\dim\cR (\mathrm{CCTP}^{\hspace{.04em} \mathrm{in}}_4[n])\le \cR (\CT^{\hspace{.04em} \mathrm{in}}[1]) \le\dim \cR (\CT^{\hspace{.04em} \mathrm{in}}[1])\le 4f_0(\CT^{\hspace{.04em} \mathrm{in}}[1])= 96.\qedhere\]
\end{proof}

\paragraph*{Many inscribed projectively unique polytopes:}
What does remain is to construct inscribed projectively unique polytopes from the polytopes $\mathrm{CCTP}^{\hspace{.04em} \mathrm{in}}[n]$. The two key tools are the following observations on Lawrence extensions and subdirect cones, respectively. For an inscribed polytope $P$, let us denote by $\mathcal{U}(P)$ the \Defn{circumscribed sphere} to $P$, and let $\mathcal{B}(P):=\conv\, \mathcal{U}(P)$ denote the ball enclosed by it.

\begin{prp}[Lawrence extensions with regard to inscribed polytopes]\label{prp:li}
Let $(P,R)$ denote a PP configuration in $S^d$ such that $P$ is an inscribed polytope and such that $\mathcal{B}(P)$ contains no point of $R$. Then the Lawrence polytope $\mathrm{L}(P,R)$ of $(P,R)$ is inscribable.
\end{prp}

We call such a PP configuration an \Defn{inscribed} PP configuration; a PP configuration is \Defn{inscribable} if it is Lawrence equivalent to some inscribed PP configuration. 
\begin{rem}
The converse to Proposition~\ref{prp:li} holds as well:
If  $(P,R)$ is not inscribable,  then $\mathrm{L}(P,R)$ is not inscribable either.
\end{rem}

\begin{proof}
As in the case of Lawrence extensions (Proposition~\ref{prp:mlwextn}), we only treat the case where $R=\{r\}$ consists of a single point. Recall that the Lawrence polytope $\mathrm{L}(P,R)$ is obtained as \[\mathrm{L}(P,R)=\conv P \cup \{\underline{r},\overline{r}\},\]
where $\underline{r}$ and $\overline{r}$ are points in $S^{d+1}\supset S^d$ that lie in a common line with $r$ such that $\underline{r}\in [r,\overline{r}]$. If $r\notin\mathcal{B}(P)$, then $\underline{r}$ and $\overline{r}$ can be chosen as the intersection points of some line containing $r$ with $\mathcal{S}$, where $\mathcal{S}$ is any sphere in $S^{d+1}$ with $\mathcal{S}\cap S^d={\mathcal{U}(P)}.$ With this choice, all vertices of the Lawrence polytope $\mathrm{L}(P,R)$  lie in~$\mathcal{S}$.
\end{proof}

A similar result holds for weak projective triples and subdirect cones.

\begin{lemma}[Subdirect cones with regard to inscribed polytopes]\label{lem:sbdci}
Let $(P,Q,R)$ denote any weak projective triple such that $P$ is inscribed, and such that $H\cap R$ does not intersect $\mathcal{B}(P)$, where $H$ denotes the wedge hyperplane of the weak projective triple. Then the subdirect cone $(P^v,Q\cup R)$ is an inscribable PP configuration. \emph{\qed}
\end{lemma}

\begin{rem}
If $(P,Q,R)$ is a weak projective triple such that no projective transformation of $(P,Q,R)$ satisfies the conditions of Lemma~\ref{lem:sbdci}, then the subdirect cone $(P^v,Q\cup R)$ is not inscribable.
\end{rem}

\begin{rem}[Universality of inscribed polytopes and Delaunay triangulations]
While simple observations, Proposition \ref{prp:li} and Lemma \ref{lem:sbdci} can be applied to provide finer characterizations of realization spaces of inscribed polytopes. In \cite{AP2}, we apply them to derive the following results:
\begin{compactitem}[$\bullet$]
\item For every primary basic semialgebraic set $S$ defined over $\Z$, there is an inscribed polytope (resp.\ Delaunay triangulation) whose realization space is homotopy equivalent to $S$. This extends results of Mn\"ev \cite{Mnev} to the inscribed setting.
\item For every inscribed polytope $P$, there is an inscribed, projectively unique polytope containing $P$ as a face. In particular, there is a Delaunay triangulation that contains $P$ and is unique up to similarities. This generalizes results of Padrol and the first named author \cite{AP}.
\end{compactitem}
\end{rem}

\begin{proof}[\textbf{Proof of Theorem~\ref{thm:projuni}}]
Set $Q:=\F_0(\mathrm{CCTP}^{\hspace{.04em} \mathrm{in}}_4[1])$. Then $Q$ frames the polytopes $\mathrm{CCTP}^{\hspace{.04em} \mathrm{in}}_4[n]$ for any $n\ge 1$ by Lemma~\ref{lem:uniextem}, and its elements have algebraic coordinates by construction. Let $V$ be any collection of points in $S^4$ whose span does not intersect the ball $\mathcal{B}(\mathrm{CCTP}^{\hspace{.04em} \mathrm{in}}_4[1])=\mathcal{B}(\mathrm{CCTP}^{\hspace{.04em} \mathrm{in}}_4[n])$ analogous to part VII.\ of the proof of Lemma~\ref{lem:affine}. Consider the weak projective triples 
\[\left(\mathrm{CCTP}^{\hspace{.04em} \mathrm{in}}_4[n],Q,R)\right),\ \ n\ge 1,\ \ R:=\mathrm{COOR}[Q\cup V]\setminus Q, \]
where $\mathrm{COOR}[Q\cup V]$ is the projectively unique point configuration provided by Proposition~\ref{prp:AP}, and let $D=f_0(Q)+f_0(R)+5=f_0(R)+29$. The polytopes $\mathrm{CCTP}^{\hspace{.04em} \mathrm{in}}_4[n]$ are inscribed, hence, the subdirect cones of these triples are inscribable PP configurations (Lemma~\ref{lem:sbdci}). The Lawrence extension of the PP configurations yield polytopes 
$\mathrm{PCCTP}^{\hspace{.04em} \mathrm{in}}_D[n]$, which are inscribable by Proposition~\ref{prp:li} and projectively unique by Proposition~\ref{prp:mlwextn}.
These polytopes have dimension $D$ and satisfy
\[f_0\left(\mathrm{PCCTP}^{\hspace{.04em} \mathrm{in}}_D[n]\right)=2(f_0(Q)+f_0(R))+f_0\left(\mathrm{CCTP}^{\hspace{.04em} \mathrm{in}}_4[n]\right)+1=12(n+1)+2D-9. \qedhere\]
\end{proof}

\subsection{Some technical details}\label{sec:Lemmas}

\subsubsection{Proof of Lemma~\ref{lem:dihang}}\label{ssc:lemdihang}

We now prove Lemma~\ref{lem:dihang}, which was used to prove that ``locally'' every ideal CCT admits an extension. We do so by first translating it into the language of \Defn{dihedral angles}, and then applying a local-to-global theorem for convexity. We stay in the notation of Lemma~\ref{lem:dihang}: $\CT$ is an ideal CCT of width at least $3$ and $\CT^\circ:=\RR(\CT,[k-2,k])$ denotes the subcomplex of $\CT$ induced by the vertices of the last three layers. Recall that $\CT^\circ$ is homeomorphic to $S^1\times S^1$, and it is in particular a manifold that is not a sphere.

\begin{definition}[Dihedral angle]
Let $M$ be a $d$-manifold with polytopal boundary in $\R^d$ or $S^d$. Let $\sigma,\, \tau$ be two facets in $\parti M$ that intersect in a $(d-2)$-face. The \Defn{(interior) dihedral angle} at $\sigma\cap \tau$ w.r.t.\ $M$ is the angle between the hyperplanes spanned by $\sigma$ and $\tau$, respectively, measured in the interior of $M$.
\end{definition}

With this notion, we formulate a lemma that generalizes Lemma~\ref{lem:dihang}.

\begin{lemma}\label{lem:dihang2}
Let $M$ denote the closure of the component of $S^3_\eq{\setminus} \CT$ that contains $\mathcal{C}_0$. Then the dihedral angles at edges in $\RR(\CT,[k-2,k-1])\subset \CT^\circ$ w.r.t.\ $M$ are strictly smaller than $\pi$.
\end{lemma}

\begin{proof}[Lemma~\ref{lem:dihang2} implies Lemma~\ref{lem:dihang}]
Let $v$ denote any vertex of $\RR(\CT,k-2)$. By Lemma~\ref{lem:dihang2} all dihedral angles of $\Lk(v,\CT^\circ)$ are smaller than $\pi$, so $\Lk(v,\CT^\circ)$ is the boundary of the convex triangle $\NO_v^1 M\subset \NO_v^1 S^3_\eq$, proving the first statement of Lemma~\ref{lem:dihang}. Furthermore, since $\CT$ is transversal, the tangent direction of $[v,\pi_0(v)]$ at $v$ lies in $\NO_v^1 M=\conv\Lk(v,\CT^\circ)$, proving the second statement.
\end{proof}
 
\smallskip

The rest of this section is consequently dedicated to the proof of Lemma~\ref{lem:dihang2}. We need the following elementary observation.

\begin{obs}\label{obs:geodesictr}
Consider the union $s$ of three segments $[a,b]$, $[b,c]$ and $[a,c]$ on vertices $a$, $b$ and $c$, and any component $B$ of the complement of $s$ in $S^2$. Then the angles between the three segments w.r.t.\ $B$ are either all smaller or equal to $\pi$ or all greater or equal to $\pi$. If $a,\,b,\,c$ do not all lie on some common great circle, these inequalities are strict.
\end{obs}

\begin{proof}[\textbf{Proof of Lemma~\ref{lem:dihang2}}]

\begin{figure}[htbf]
\centering 
  \includegraphics[width=0.24\linewidth]{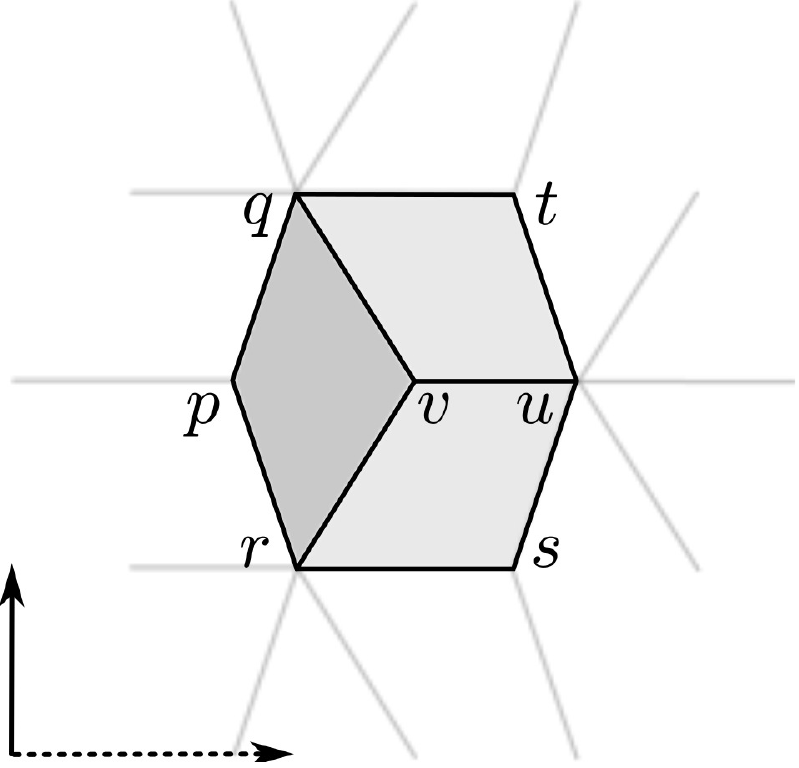} 
  \caption{\small We show a part of the CCT $\CT^\circ\subset \CT$. $v$ is a vertex of layer $k-2$, $q,r,u$ are vertices of layer $k-1$, and the vertices $p,s,t$ lie in layer $k$ of $\CT$.} 
  \label{fig:torusdihan}
\end{figure}
                                                                                                                                                                                                                                                                                                                
Since the facets of $\CT$ are convex, all dihedral angles at vertices of layer $k$ must be larger than $\pi$. Thus the dihedral angles at vertices of layer $k-2$ must be smaller or equal to $\pi$, and at least one of the dihedral angles at each edge of layer $k-2$ must be strictly smaller than $\pi$, since the contrary assumption would imply that $\CT^\circ$ is the boundary of a convex body in~$S^3_\eq$ by Theorem~\ref{thm:heij}, or the more elementary Theorem of Tietze~\cite[Satz 1]{Tietze} \cite{Nakajima}, and in particular homeomorphic to a $2$-sphere, contradicting the assumption that $\CT^\circ$ is a torus.

Consider the star $\St(v,\CT^\circ)$ (Figure~\ref{fig:torusdihan}) for any vertex $v$ of $\RR(\CT,k-2)$. As observed, one of the dihedral angles at edges $[v,u]$, $[v,r]$ and $[v,q]$ must be smaller than $\pi$. In particular, not all vertices of $\Lk(v,\CT^\circ)$ lie on a common great circle, and all dihedral angles at edges $[v,u]$, $[v,r]$ and $[v,q]$ w.r.t.\ $M$ are smaller than $\pi$ by Observation~\ref{obs:geodesictr}. 
\end{proof}

\subsubsection{Proof of Proposition~\ref{prp:loccrt1lay}}\label{ssc:localtoglobal}
The goal of this section is to prove Proposition~\ref{prp:loccrt1lay}, which provides a local-to-global criterion for the convex position of ideal CCTs of width $3$. We work in the $4$-sphere $S^4\subset \R^5$, and the equator sphere~$S^3_\eq$. Furthermore, $\pp$ shall denote the projection from $S^4{\setminus} \{\pm e_5\}$ to~$S^3_\eq$.

\begin{lemma}\label{lem:localcrit}
Let $\CT$ be an ideal $3$-CCT in~$S^4$ such that for every facet $\sigma$ of $\CT$, there exists a halfspace $H(\sigma)$ containing $\sigma$ in the boundary and such that $H(\sigma)$ contains all remaining vertices of layers $1,\, 2$ connected to $\sigma$ via an edge of $\CT$ in the interior. Then $H(\sigma)$ contains all vertices of $\F_0(\CT){\setminus} \F_0(\sigma)$ in the interior. 
\end{lemma}

\begin{figure}[htbf]
\centering 
\includegraphics[width=0.3\linewidth]{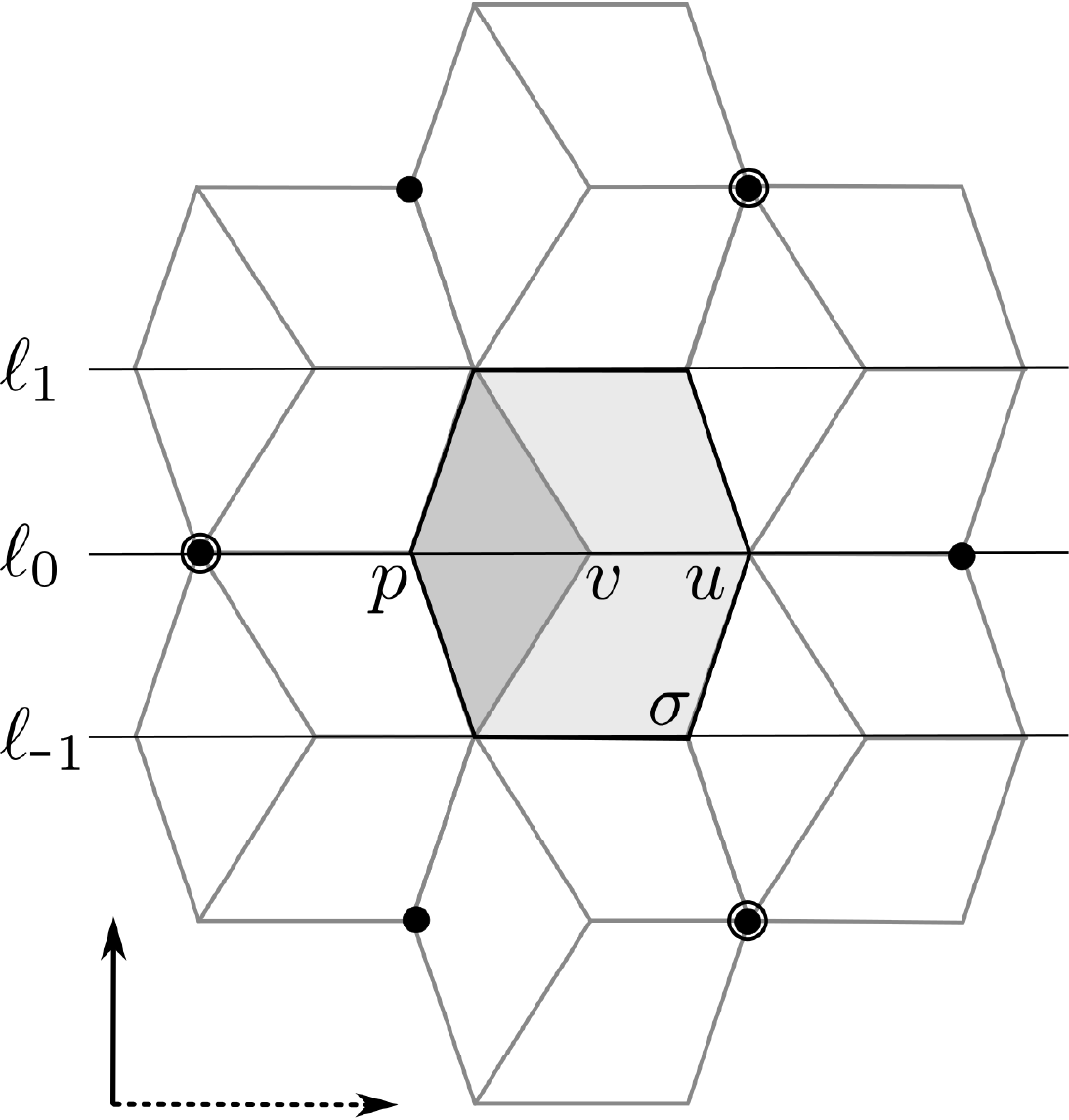} 
\caption{\small The picture shows part of the complex $\RR(\CT,[0,2])$. Lemma~\ref{lem:localcrit} is concerned with the marked by $\bullet$ (connected to $\sigma$ via an edge, in layers $1$ and $2$); Proposition~\ref{prp:loccrt1lay} considers those vertices marked with an additional $\circ$ (connected to $\sigma$ via an edge and in layer $1$).}
\label{fig:localglobalsetup}
\end{figure}

\begin{proof}
We will prove the claim by contraposition. Let $H=H(\sigma)$ denote a halfspace in~$S^4$ containing $\sigma$ in the boundary, with outer normal $\n=\n(H)$. To prove the claim of the lemma, assume that $H^c:=S^4{\setminus} \mathrm{int} H$ contains a vertex $w$ of $\CT$ that is not a vertex of $\sigma$. We have to prove that there is another vertex that
\begin{compactenum}[(i)]
\item lies in layers $1$ or $2$ of $\CT$,
\item lies in $H^c$,
\item is not a vertex of $\sigma$, but that
\item is connected to $\sigma$ via some edge of $\CT$.
\end{compactenum}

\noindent We consider this problem in four cases. Let $\sigma$ be any facet of $\CT$, and let us denote the  orthogonal projection of $x\in \R^5$ to the $\Sp\{e_i,\,e_j\}$-plane in $\R^5$ by $x_{i,j}$. By symmetry, $\vv{n}=(\ast,\,\ast,\,\e v_3,\, \e v_4,\,\ast)$, where $v=v(\sigma)=(v_1(\sigma),\,\dots,\,v_5(\sigma))$ denotes the only layer $0$ vertex of $\sigma$, and $\e=\e(H)$ is some real number. We can take care of the easiest case right away:

\medskip \textbf{Case (0)} If $w$ lies in facet ${\tau}$ of $\CT$ that is obtained from $\sigma$ by a rotation of the $\Sp\{e_3,\,e_4\}$-plane, then $\e=\e(H)\le  0$ by Proposition~\ref{prp:alignsymm}(b) and (d). Thus $H^c$ must contain ${\tau}$ and all other facets of $\CT$ obtained from $\sigma$ by rotation of the $\Sp\{e_3,\,e_4\}$-plane. In particular, it contains all vertices of adjacent facets that are obtained from $\sigma$ by a rotation of the $\Sp\{e_3,\,e_4\}$-plane, among which we find the desired vertex, even a vertex satisfying (i) to (iv) among the vertices of $\RR(\CT,1)$. We may assume from now on that $\e$ is positive.

\medskip

It remains to consider the case in which $w$ satisfies (ii) and (iii) and is obtained from a vertex of $\sigma$ only from a nontrivial rotation of the $\Sp\{e_1,\,e_2\}$-plane followed by a (possibly trivial) rotation of the $\Sp\{e_3,\,e_4\}$-plane. Since $\e$ is positive, there exists a vertex $w'$ satisfying (ii) and (iii) in the same layer of $\CT$ as $w$, but which lies in \[\ell_0=\pp^{-1}\pi_2^{\mathrm{f}}(\pp(v)),\ \ell_1=\pp^{-1}\pi_2^{\mathrm{f}}(\pp(q))=\rot_{3,4} \ell_0\ \text{or}\  \ell_{-1}=\pp^{-1}\pi_2^{\mathrm{f}}(\pp(r))=\rot_{3,4}^{-1}\ell_0.\]
The existence of a vertex satisfying (i)-(iv) now follows from the following observation:

\smallskip

\emph{Let $x,y$ be any two non-antipodal points in $S^1$, and let $m$ be any point in the segment $[x,y]$. Assume $n$ is any further point in $S^1$ such that $\langle n,m\rangle\le  \langle n,-m\rangle$. Then $\langle n,y\rangle\le  \langle n,-y\rangle$ or $\langle n,x\rangle\le  \langle n,-x\rangle$.}

\smallskip 

\begin{figure}[htbf]
\centering 
 \includegraphics[width=1\linewidth]{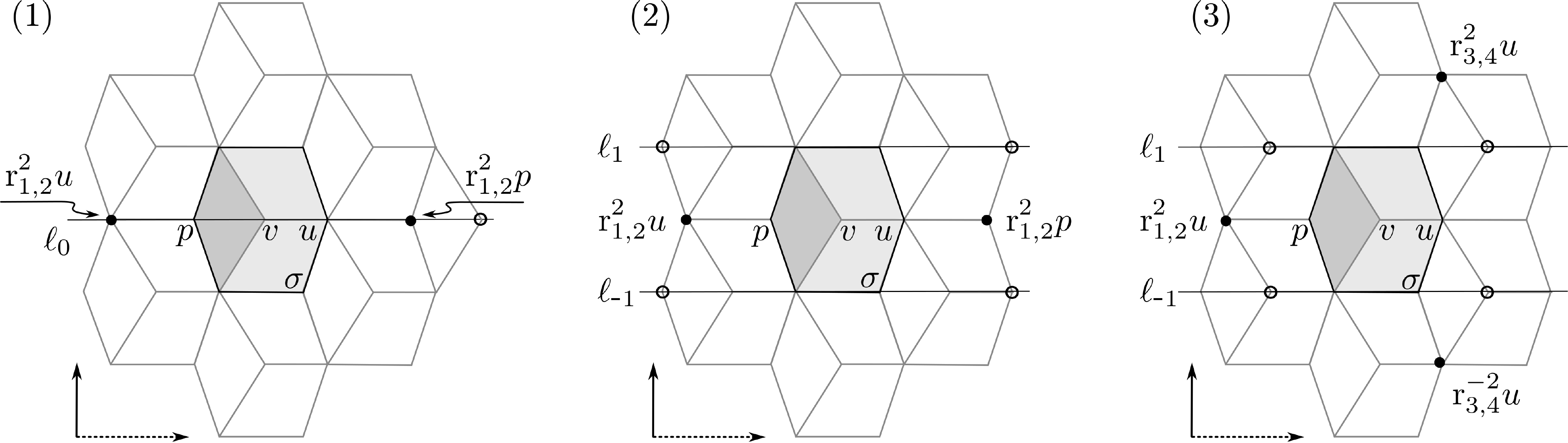} 
\caption{\small The three cases show how to, given a vertex satisfying (ii) and (iii) (circled black), obtain a vertex satisfying (i)-(iv) (found among the vertices marked by black disks).}
  \label{fig:localcglobal}
\end{figure}

\medskip \textbf{Case (1)} Assume $w'\in\ell_0$. This case is only nontrivial if $w'$ is not in $\RR(\CT,[1,2])$. Thus assume (w.l.o.g.) that $w'$ lies in layer $0$ (i.e.\ it is the vertex circled in Figure~\ref{fig:localcglobal}(1)), the other case is fully analogous. Then $\rot_{1,2}^2 w'=v$. To construct the desired vertex $x$, note that $\pi_2(\pp(v))$ lies in the segment $[\pi_2(\pp(u)),\pi_2(\pp(p))]$ by Proposition~\ref{prp:alignsymm}(e). Now, \[w'_{1,2}=-v_{1,2},\, w'_{3,4}=v_{3,4}\ \text{and}\ w'_{5}=v_{5},\] and since $w'\in H^c$, $\langle \n,v_{1,2}\rangle\le  \langle \n,w'_{1,2}\rangle$. Consequently, for $x={u}$ or $x={p}$, $\langle \n,x_{1,2}\rangle\le  \langle \n,-x_{1,2}\rangle$ and thus \[\langle \n,x\rangle\le  \langle \n,\rot_{1,2}^2 x\rangle\ \Longleftrightarrow \ \rot_{1,2}^2 x\in H^c.\] 
Since both $\rot_{1,2}^2{u}$ and $\rot_{1,2}^2{p}$ satisfy (i), (iii) and (iv), this vertex satisfies the properties (i) to (iv).

\medskip \textbf{Case (2)} Assume $w'\in \RR(\CT,[1,2])\cap (\ell_1\cup\ell_{-1})$. Then $w'$ is obtained from a vertex $y$ of $\sigma$ by a rotation of the $\Sp\{e_1,\,e_2\}$-plane, i.e.\ $w'=\rot_{1,2}^2y$. It then follows, as in Case (1), that $\rot_{1,2}^2{u}$ or $\rot_{1,2}^2{p}$ lie in $H^c$, since $\pi_2(\pp(y))\in[\pi_2(\pp(u)),\pi_2(\pp(p))]$ by Proposition~\ref{prp:alignsymm}(e). This vertex satisfies the properties (i) to (iv), because both $\rot_{1,2}^2{u}$ and $\rot_{1,2}^2{p}$ satisfy (i), (iii) and (iv).

\medskip \textbf{Case (3)} 
If $w'\in \RR(\CT,\{0\}\cup\{3\})\cap (\ell_1\cup\ell_{-1})$, then it must lie in a $2$-face $F$ which intersects $\sigma$ in an $1$-face (cf.\ Figure~\ref{fig:localcglobal}(3) for the case of layer $0$ vertices). The remaining vertex of $F$ that does not lie in $\sigma$ is a vertex of layers $1$ or $2$. Since $e\subseteq \parti H^c$ and $w'\in H^c$, this vertex lies in $H^c$ and must be connected to $\sigma$ via an edge, and consequently satisfies properties (i) to (iv), as desired.
\end{proof}

We can now prove Proposition~\ref{prp:loccrt1lay}.

\begin{proof}[\textbf{Proof of Proposition~\ref{prp:loccrt1lay}}] For the labeling of vertices, we refer to Figure~\ref{fig:localglobal2}(1). We stay in the notation of Lemma~\ref{lem:localcrit}. Using its result, we only have to prove that if $\sigma$ is a facet of $\CT$ and $H(\sigma)$ with $\sigma\subseteq \parti H(\sigma)$ is a halfspace such that the vertices $\rot^2_{1,2}{u},\, \rot^2_{3,4}{u}$ and $\rot^{\,-2}_{3,4}{u}$ of $\RR(\CT,1)\subset \CT$ lie in the interior of $H(\sigma)$, then the vertices $\rot^2_{1,2}{p},$ $ \rot^2_{3,4}{p}$ and $\rot^{\,-2}_{3,4}{p}$ of $\RR(\CT,2)\subset \CT$ lie in $\intx H(\sigma)$ as well. Let $\n=\n(\sigma)$ denote the outer normal to $H(\sigma)$.

\begin{figure}[htbf]
\centering 
 \includegraphics[width=0.9\linewidth]{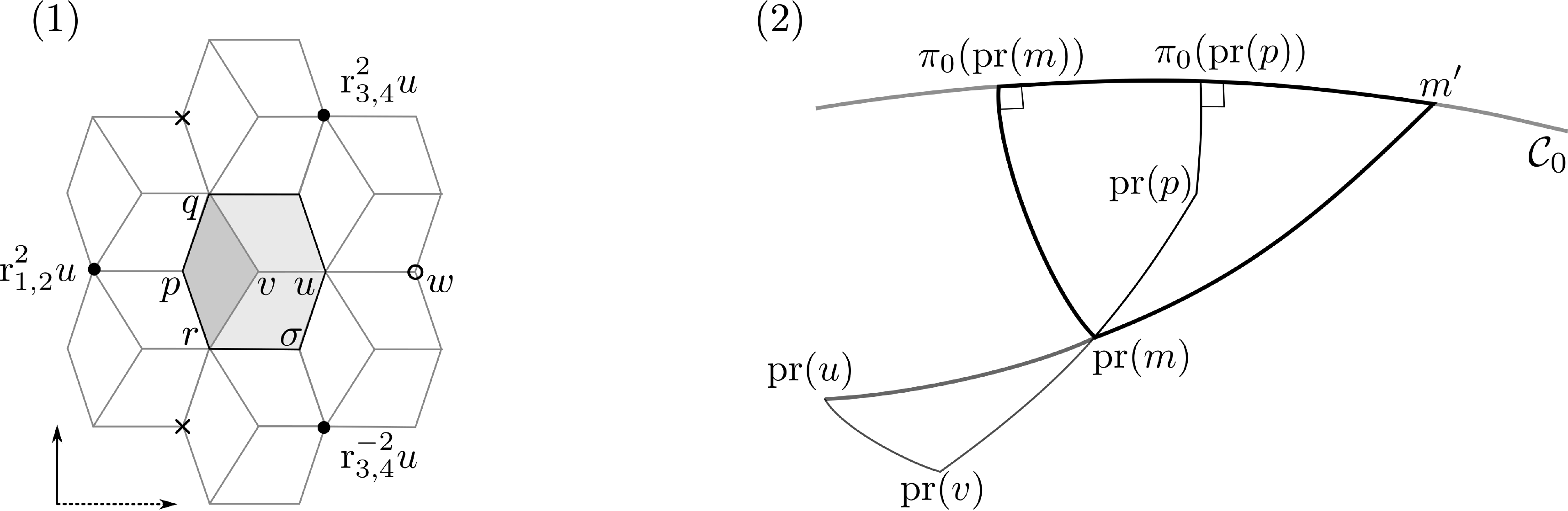} 
\caption{\small  (1) As before, the picture shows part of the complex $\RR(\CT,[0,2])$. We have to show that if all layer $1$ vertices satisfying (iii) and (iv) lie in $H(\sigma)$ (full black disks), then so do all vertices of layer $2$ satisfying (iii) and (iv) (circled black). For the vertices of layer $2$ marked by crosses, this follows as in Lemma~\ref{lem:localcrit}.
\newline (2) Illustration for inequality $\mathrm{d}({p},q)\le  \mathrm{d}(m',q)$.} 
  \label{fig:localglobal2}
\end{figure}

As already observed in the proof of Lemma~\ref{lem:localcrit}, $\n$ is of the form $\n=(\ast,\, \ast,\,\e v_3,\, \e v_4,\, \ast)$, and if $\rot^2_{3,4}{u}\in \intx H(\sigma)$, then $\e>0$ and \[\langle\n,\rot^{\,\pm 2}_{3,4}{p}\rangle<\langle\n,{p}\rangle\ \Longleftrightarrow\ \rot^{\,\pm 2}_{3,4}{p}\in \intx H(\sigma) .\]
Thus it remains to be proven that $w=\rot^2_{1,2}{p}$ lies in $\intx H(\sigma)$, which we will do in two steps. If $x$ and $y$ lie in $S^4{\setminus} \{z\in S^4:\ z_{1,2}=0\}$, let $\mathrm{d}(x,y)$ measure the distance between $\pi_0(\pp(x))$ and $\pi_0(\pp(y))$ in~$S^4$.

\begin{compactenum}[\rm(a)]
\item We prove that $\mathrm{d}({p},q)$ is smaller than $\nicefrac{\pi}{4}$.
\item From $\mathrm{d}({p},q)<\nicefrac{\pi}{4}$ we conclude that $w=\rot_{1,2}^2 {p}$ lies in the interior of $H(\sigma)$.
\end{compactenum}

\noindent To see the inequality of (a), let $m$ denote the midpoint of the segment $[r,q]$. Consider the unique point $m'$ in $\mathcal{C}_0$ such that $\pp(m)\in[\pp({u}),m']$, which is guaranteed to exist since $\pi_2(\pp({u}))=\pi_2(\pp(m))$ (Proposition~\ref{prp:alignsymm}(b) and (d)). We locate $\pp(p)$ w.r.t.\ $\pp({m})$, $\pi_0(\pp({m}))$ and $m'$: 
\begin{compactitem}[$\circ$]
\item By Proposition~\ref{prp:alignsymm}(f), $\pp({p})$ and $m'$ lie in the same component of $S^3_\eq{\setminus} \pi^{\SSp}_0(\pp(q))$.
\item Furthermore, the complement of $\SSp\{\pp(r),\, \pp(q),\, \pp({u})\}$ in~$S^3_\eq$ contains $\pi_0(\pp(m))$ and $\pp({p})$ in the same component, since $\pp(v)$ and $\pi_0(\pp(m))$ lie in different components.
\item Finally, by (b) and (d) of the same Proposition, $\pi_2(\pp(m))=\pi_2(\pp({p}))$. 
\end{compactitem}

\noindent Thus $\pp({p})$ is contained in the triangle on vertices $m'$, $\pi_0(\pp(m))$ and $\pp(m)$ in $\pi_2^{\SSp}(m)$, see also Figure 
\ref{fig:localglobal2}. As the projection of $\pp(m)$ to $\mathcal{C}_0$ lies in the segment $[\pi_0(\pp(m)),m']$, $\pi_0(\pp({p}))$ lies in the segment from $\pi_0(\pp(q))=\pi_0(\pp(m))$ to $m'$. Thus \[\mathrm{d}({p},q)=\mathrm{d}({p},m)\le  \mathrm{d}(m',m)=\mathrm{d}(m',q).\] To compute the latter, notice that, after applying rotations of planes $\Sp\{e_1,\,e_2\}$ and $\Sp\{e_3,\,e_4\}$, we may assume that \[{u}=\big(u_1,0,u_3,0,1\big),\ \   u_1,\, u_3>0.\]
Then the coordinates of $q$ and $r$ are given as \[\big(0,-u_1,\tfrac{1}{2}u_3,\pm\tfrac{\sqrt{3}}{2}u_3,1\big)\ \text{and} \ m'=\tfrac{1}{\sqrt{5}}\big(1,-2,0,0,0\big).\] 
In particular, \[\mathrm{d}({p},q)\le  \mathrm{d}(m',q)=\arctan \tfrac{1}{2}<\tfrac{\pi}{4}.\]
As for step (b): After applying rotations of planes $\Sp\{e_1,\,e_2\}$ and $\Sp\{e_3,\,e_4\}$ of $\CT$, we may assume that $(u_1,0,u_3,0,1)$, as above. Then, as before, the coordinates of the remaining layer $1$ vertices $q$ and $r$ of $\sigma$ are $(0,-u_1,\nicefrac{1}{2}u_3,\pm\nicefrac{\sqrt{3}}{2}u_3,1)$. A straightforward calculation shows that any normal to $\conv\{{u},\,q,\,r\}$ is a dilate of 
\[\vv{n}=\big(\mu,\,-(\mu+\tfrac{u_3}{2u_1}),\,1,\,0,\,n_5\big),\ \ \mu\in \R,\ n_5=-\mu u_1-u_3\]
and if $\vv{n}$ is the outer normal to a halfspace $H(\sigma)$ that exposes the triangle $\conv\{{u},\,q,\,r\}$ among all vertices of $\RR(\CT,1)$ connected to $\sigma$, then the dilation is by a positive real and $\mu>0$. 

Since $\CT$ is transversal, Proposition~\ref{prp:alignsymm}(e) gives ${p}_1<0$. If additionally $\mathrm{d}({p},q)<\nicefrac{\pi}{4}$, then ${p}_2<{p}_1<0$. In particular, 
\[2\mu({p}_1-{p}_2)-\tfrac{u_3}{u_1}{p}_2>0,\]
which, due to the fact that $w_{1,2}=-{p}_{1,2},\, w_{3,4}={p}_{3,4}$ and $w_{5}={p}_{5}=1$, is equivalent to 
\[
\langle\vv{n},{p}\rangle={p}_1 \mu-\big(\mu+\tfrac{u_3}{2u_1}\big){p}_2+{p}_3+n_5>-{p}_1 \mu + \big(\mu+\tfrac{u_3}{2u_1}\big){p}_2+{p}_3+n_5= \langle\vv{n},w\rangle,
\]
and consequently $w$ is in the interior of $H(\sigma)$.
\end{proof}

\subsubsection*{Acknowledgements}
We are grateful to Alexander Bobenko for background information on Q-nets, to Igor Pak for valuable discussions, to Francisco Santos for particularly helpful suggestions, and to Miriam Schl\"oter for some of the figures in this paper. The first author thanks the Hebrew University Jerusalem, whose hospitality he enjoyed on several occasions while working on this paper.

{\small
\bibliographystyle{myamsalpha}
\bibliography{References}}

\end{document}